\documentclass[11pt, a4]{amsart}
\usepackage{amsmath}
\usepackage{amsthm}
\usepackage{amssymb}
\usepackage{sseq}
\usepackage[T1]{fontenc}
\usepackage[pagebackref, breaklinks]{hyperref}
\usepackage[curve]{xypic}
\usepackage[left=3.2cm,right=3.2cm,top=3cm,bottom=3cm]{geometry} 
\usepackage{pdfpages}
\usepackage{stmaryrd}
\usepackage{marvosym}
\usepackage[capitalise]{cleveref}
\usepackage{wrapfig}
\usepackage{url}
\usepackage{todonotes}
\usepackage{breakurl}

\let\oldmarginpar\marginpar
\renewcommand\marginpar[1]{\-\oldmarginpar[\raggedleft\footnotesize #1]%
{\raggedright\footnotesize #1}}

\numberwithin{equation}{section}
\newtheorem{thm}[equation]{Theorem}

\newtheorem{lemma}[equation]{Lemma}

\newtheorem{cor}[equation]{Corollary}

\newtheorem{prop}[equation]{Proposition}
\newtheorem{claim}[equation]{Claim}

\newtheorem*{conjecture*}{Conjecture}

\newtheorem*{question*}{Question}

\theoremstyle{definition}
\newtheorem{defi}[equation]{Definition}

 \newtheorem{example}[equation]{Example}
  \newtheorem*{example*}{Example}
 \newtheorem{examples}[equation]{Examples}

\theoremstyle{remark}

\newtheorem{remark}[equation]{Remark}

\def\co{\colon\thinspace}

\newcommand{\G}{\mathbb{G}}
\newcommand{\Z}{\mathbb{Z}}
\newcommand{\F}{\mathbb{F}}

\newcommand{\Q}{\mathbb{Q}}

\newcommand{\N}{\mathbb{N}}
\newcommand{\A}{\mathbb{A}}

\newcommand{\C}{\mathbb{C}}

\newcommand{\Gm}{\mathbb{G}_m}

\newcommand{\Ac}{\mathcal{A}}
\newcommand{\BB}{\mathcal{B}}
\newcommand{\CC}{\mathcal{C}}
\newcommand{\DD}{\mathcal{D}}
\newcommand{\EE}{\mathcal{E}}
\newcommand{\FF}{\mathcal{F}}
\newcommand{\GG}{\mathcal{G}}
\newcommand{\UU}{\mathcal{U}}

\newcommand{\PP}{\mathcal{P}}

\newcommand{\XX}{\mathcal{X}}

\newcommand{\YY}{\mathcal{Y}}

\newcommand{\OO}{\mathcal{O}}

\newcommand{\MM}{\mathcal{M}}

\newcommand{\MMb}{\overline{\mathcal{M}}}

\newcommand{\tensor}{\otimes}

\DeclareMathOperator{\Aut}{Aut}

\DeclareMathOperator{\ev}{ev}
\DeclareMathOperator{\Mat}{Mat}
\DeclareMathOperator{\UpT}{UpT}

\DeclareMathOperator{\cusps}{cusps}

\DeclareMathOperator{\im}{im}

\DeclareMathOperator{\Tor}{Tor}

\DeclareMathOperator{\coker}{coker}
\DeclareMathOperator{\Gal}{Gal}
\DeclareMathOperator{\Ext}{Ext}
\DeclareMathOperator{\Hom}{Hom}
\DeclareMathOperator{\Spec}{Spec}
\DeclareMathOperator{\QCoh}{QCoh}

\DeclareMathOperator{\Pic}{Pic}

\DeclareMathOperator{\rk}{rk}

\DeclareMathOperator{\End}{End}

\newcommand*{\myproofname}{Proof of claim}
\newenvironment{myproof}[1][\myproofname]{\begin{proof}[#1]}{\end{proof}}

\newcommand{\lbreak}{\leavevmode \\\vspace{-0.5cm}}

\DeclareMathOperator{\XXb}{\overline{\XX}}
\DeclareMathOperator{\Yb}{\overline{\YY}}

\DeclareMathOperator{\Prj}{\mathbb{P}}
\DeclareMathOperator{\Char}{char}
\DeclareMathOperator{\Tr}{Tr}

\begin{document}
\title[Additive decompositions for modular forms]{Additive decompositions for rings of modular forms}
\author{Lennart Meier}
\address{Mathematical Institute\\
Budapestlaan 6\\
3584 CD Utrecht\\
Nederland}
\email{f.l.m.meier@uu.nl}

\begin{abstract}
 We study rings of integral modular forms for congruence subgroups as modules over the ring of integral modular forms for $SL_2\Z$. In many cases these modules are free or decompose at least into well-understood pieces. We apply this to characterize which rings of modular forms are Cohen--Macaulay and to prove finite generation results. These theorems are based on decomposition results about vector bundles on the compactified moduli stack of elliptic curves. 
\end{abstract}

\maketitle

\section{Introduction}
Rings of modular forms for congruence subgroups are of great importance in number theory. An example is $M_*(\Gamma_1(n);R)$, the ring of $\Gamma_1(n)$-modular forms over a ring $R$. If $R$ is a subring of $\C$, we can simply define it as the subspace of $\Gamma_1(n)$-modular forms with $q$-expansion in $R$. In other cases, we define it via global sections of line bundles on the compactified modular curves $\MMb_1(n)_R$ (which we view as stacks for $n<5$).\footnote{If $R$ does not contain an $n$-th root of unity, the coincidence of these two definitions can depend on the way one views sections of line bundles on $\MMb_1(n)_{\C}$ as functions on the upper half-plane. Apart from Appendix \ref{app:CuspForms}, we will only use $q$-expansions in the presence of an $n$-th root of unity.} We denote the corresponding ring of weakly holomorphic modular forms by $\widetilde{M}_*(\Gamma_1(n);R)$. 

While there is a lot of information available for low $n$, in general these rings are hard to understand. For example, it is an equivalent form of a famous theorem of Mazur \cite{Maz77} that the ring $\widetilde{M}_0(\Gamma_1(n);\Q)$ of modular functions only admits a ring homomorphism to $\Q$ if $n\leq 10$ or $n=12$, i.e.\ exactly if $\MMb_1(n)_{\C}$ has genus $0$. There has certainly been progress to understand the rings $M_*(\Gamma_1(n);R)$ in general (see e.g.\ \cite{Rus16} or \cite{V-Z15}), but a complete understanding seems to be difficult to obtain. 

The aim of the present article is instead the more modest goal of an \emph{additive} understanding of the ring of modular forms for congruence subgroups. More precisely, we aim to understand $M_*(\Gamma;R)$ for $\Gamma \in \{\Gamma_0(n), \Gamma_1(n), \Gamma(n)\}$ as a module over the ring $M_*^R = M_*(SL_2(\Z);R)$ of modular forms itself. 
Remarkably, these modules are free if $R$ is a field.

\begin{thm}\label{thm:fieldfree}
Let $\Gamma = \Gamma_1(n), \Gamma(n)$ or $\Gamma_0(n)$ and $K$ be a field of characteristic not dividing $n$. Then $M_*(\Gamma;K)$ is free as graded module over $M_*^K$ of rank $[SL_2(\Z):\Gamma]$ if $\Char K \neq 2$ and of rank $\frac12[SL_2(\Z):\Gamma]$ if $\Char K =2$ and $n\geq 2$.
\end{thm}

This will be proven as Theorem \ref{thm:decfield}. The result has been observed for $\Char K = 0$ already in \cite[Proposition 6.6]{C-F17}. It actually has a uniform and easy proof for $\Char K \neq 2,3$ as in these cases the compactified moduli stack of elliptic curves $\MMb_{ell, K}$ is a weighted projective line, but the cases of characteristic $2$ or $3$ are more subtle. In these cases we have the following example application.
\begin{cor}
	Let $K$ be a field of characteristic $p=2$ or $3$ and let $A$ be the corresponding Hasse invariant. Then $M_*(\Gamma; K)/(A-1)$ is a free module over $K[\Delta]$ if $\Gamma$ is as in the preceding theorem. 
\end{cor}
This is of interest as the $q$-expansion of $A$ equals $1$ and thus the $q$-expansion factors over this quotient $M_*(\Gamma; K)/(1-A)$. Noting that $M_*(\Gamma_1(n); \F_p) \cong M_*(\Gamma_1(n); \Z[\frac1n])/p$ if $n\geq 2$ (and $\ast \geq 2$), one sees that more generally \cref{thm:fieldfree} allows to determine $M_*(\Gamma_1(n); \Z[\frac1n])/(p, f_1, \dots, f_k)$ for arbitrary $SL_2(\Z)$-modular forms $f_1, \dots, f_k$. 

If we want to consider also similarly congruences for ideals of modular forms not containing $p$, we should not restrict or attention to modular forms with coefficients in a field. We obtain our cleanest results if we consider coefficients in $R= \Z_{(l)}$ instead and restrict to tame congruence subgroups $\Gamma$. Here, we call $\Gamma \in \{\Gamma_1(n), \Gamma(n), \Gamma_0(n)\}$ \emph{tame} if either $\Gamma = \Gamma_1(n)$, $\Gamma = \Gamma(n)$ or $\gcd(\phi(n),6)$ is invertible in the given base ring $R$. Even under these restrictions, $M_*(\Gamma; \Z_{(l)})$ will not be a free $M_*^{\Z_{(l)}}$-module anymore unless $l\geq 5$. But even if $l=2$ or $3$, we can often still determine its indecomposable pieces and these are well-understood.  

\begin{thm}\label{thm:main}Let $n\geq 2$ and $l$ be a prime not dividing $n$. Let furthermore $\Gamma = \Gamma_1(n), \Gamma(n)$ or $\Gamma_0(n)$ be tame with respect to $\Z_{(l)}$. Then
 \begin{enumerate}
  \item $M_*(\Gamma;\Z_{(l)})$ is a free graded module over $M_*^{\Z_{(l)}}$ if and only if $l\geq 5$ and every weight-$1$ modular form for $\Gamma$ over $\F_l$ is liftable to $\Z_{(l)}$,
  \item $M_*(\Gamma;\Z_{(3)})$ decomposes into shifted copies of $M_*(\Gamma_1(2); \Z_{(3)}) \cong \Z_{(3)}[b_2, b_4]$ as a graded module over $M_*^{\Z_{(3)}}$ if and only if every weight-$1$ modular form for $\Gamma$ over $\F_3$ is liftable to $\Z_{(3)}$,
   \item $M_*(\Gamma;\Z_{(2)})$ decomposes into shifted copies of $M_*(\Gamma_1(3); \Z_{(2)})\cong \Z_{(2)}[a_1, a_3]$ as a graded module over $M_*^{\Z_{(2)}}$ if and only if every weight-$1$ modular form for $\Gamma$ over $\F_2$ is liftable to $\Z_{(2)}$.
 \end{enumerate}
\end{thm}

More generally, we can often prove decomposition results if we replace $\Z_{(l)}$ by an $l$-local ring or if we demand that $6$ is invertible. The more precise statements will be given in Theorems \ref{thm:decint6} and \ref{thm:decint23} for the ``if'' parts and in Corollary \ref{cor:onlyiflift} and Proposition \ref{prop:onlyif5} for the ``only if'' parts. There are explicit formulae for the shifts in (2) in terms of dimensions of spaces of modular forms, which follow from our results in Section \ref{sec:dec}. An example of a similar decomposition result in a non-tame situation has been obtained in \cite{M-OModular}. 

Regarding non-liftable modular forms, we remark that Mestre was the first to construct an example of such, namely a mod-$2$ eigenform of weight $1$ for $\Gamma_0(1429)$. (Its associated $2$-dimensional Galois representation has image $SL_2(\F_8)$ and as this is not a quotient of a finite subgroup of $GL_2(\C)$, this modular forms is not liftable to characteristic zero \cite{Edi06, Wie14}.) Nowadays, there are many more known examples of non-liftable forms known; for example, we have obtained the result that at the prime $2$ the minimal level of such a form in the tame case is $65$ (see Remark \ref{rem:cuspy} for more information). 

It turns out that these subtleties become irrelevant if we are interested in modular functions instead of modular forms:
\begin{thm}
Let $R$ be a commutative ring in which $n\geq 2$ is invertible and let $\Gamma = \Gamma_1(n), \Gamma(n)$ or $\Gamma_0(n)$. 
Then $\widetilde{M}_0(\Gamma; R)$ 
is free as a module over 
$\widetilde{M}_0(SL_2(\Z); R) \cong R[j]$. 
\end{thm}

There are several corollaries of these theorems. In all of them, let $n$ and $\Gamma$ be as in the last theorem and $R$ a commutative $\Z[\frac1n]$-algebra. 

\begin{cor}
 The ring $M_*(\Gamma; R)$ is finitely generated as an $M_*^R$-module and likewise $\widetilde{M}_*(\Gamma;R)$ is finitely generated as an $\widetilde{M}_*^R$-module. If $\Gamma$ is tame, the generators can be chosen in degrees at most $17$, and if $\Gamma = \Gamma_0(n)$ and $\frac12\in R$ the generators can be chosen in degrees at most $21$.\footnote{A different proof for the qualitative statement was sketched to me by Fran\c{c}ois Brunault on mathoverflow: It suffices to show finite generation for $\Gamma = \Gamma(n)$ and $R$ containing an $n$-th root of unity. The ring $M_*(\Gamma(n); R)$ has a $SL_2(\Z/n)$-action with fixed points $M_*^R$. As every $f\in M_*(\Gamma(n); R)$ is a zero of the monic polynomial $\prod_{g\in SL_2(\Z/n)}(x-g. f)$, the ring $M_*(\Gamma(n);R)$ is integral over $M_*^R$. In \cite[Theorem 3]{Scholl79}, Scholl shows that $M_*(\Gamma(n);R)$ is a finitely generated $R$-algebra, which together with integrality implies the statement. On the other hand, the bounds on the degrees of the generators obtained in this way definitely depend on $n$ (which is in contrast to our sharper quantitative results).}
\end{cor}

\begin{cor}\label{cor:CM}
 For $K$ a field with $\Char(K)$ not dividing $n$, the ring $M_*(\Gamma;K)$ is Cohen--Macaulay. If $\Gamma$ is tame for $\Z[\frac1n]$, the ring $M_*(\Gamma;\Z[\frac1n])$ is Cohen--Macaulay if and only if every weight-$1$ modular form for $\Gamma$ over $\F_l$ is liftable to $\Z[\frac1n]$ for every $l$ not dividing $n$.
\end{cor}
While authors like Eichler, Freitag, Tsuyumine and Gottesman (see e.g.\ \cite{Tsu86} or \cite{Got19}) have proven results about the Cohen--Macaulayness of (Hilbert, Siegel and vector-valued) modular forms over the complex numbers, the author is not aware of previous work on this question over a field of different characteristic or with integral coefficients. This result has proven useful in \cite{M-OModular} when studying moduli of cubic curves. 

Let us say a few words on how we obtain our main theorems. For simplicity, we concentrate on the case $\Gamma = \Gamma_1(n)$. Let $f_n\colon \MMb_1(n)_R \to \MMb_{ell,R}$ be the projection map to the compactified moduli stack of elliptic curves. There is a line bundle $\omega$ on $\MMb_{ell,R}$ such that 
$$M_k(\Gamma_1(n);R) \cong H^0(\MMb_{ell,R};(f_n)_*(f_n)^*\omega^{\tensor k})\cong H^0(\MMb_{ell,R};\omega^{\tensor k}\tensor (f_n)_*\OO_{\MMb_1(n)}).$$ 
Thus, all our splitting results above follow from corresponding splitting results for the vector bundle $(f_n)_*\OO_{\MMb_1(n)_R}$. If $6$ is invertible, these are easy to prove since, in this case, $\MMb_{ell,R}$ is a weighted projective stack, and we know a lot about vector bundles on these. If $6$ is not invertible, the arguments become more delicate. We first obtain splitting results over a field using a Krull--Schmidt theorem and use this to obtain our integral results. 

Our results about vector bundles also have applications to the study of topological modular forms as explored in \cite{MeiTopLevel}. This is also one reason for our insistence on integral results and not ignoring the small primes: topological modular forms are most interesting and have most applications precisely at the primes $2$ and $3$. 

The structure of the present article is as follows. Section \ref{sec:ModularCurves} contains background on the various moduli stacks we consider. In Section \ref{sec:decexistence}, we prove our decomposition results for vector bundles, first in the case of a field and then over more general rings; if you are only interested in the situation of base fields of characteristic not $2$ or $3$ you can skip this section. In Section \ref{sec:dec}, we will give explicit formulae for the decompositions. Section \ref{sec:modforms} will deduce from these considerations the results stated in this introduction. As several of the occurring stacks are weighted projective stacks, we devote Appendix \ref{sec:weighted} to them; our main result gives a partial classification of vector bundles in the case of weighted projective lines, which may be of independent interest. Appendix \ref{app:Hasse} contains a proof how to lift the Hasse invariant to a characteristic zero modular form in the presence of a level structure. Appendix \ref{app:CuspForms} is about the computation of weight-$1$-cusp forms over $\F_p$, showing in particular that $n=65$ in the lowest $n$ for which there is a level-$n$ cusp form over $\F_2$ that does not lift to characteristic zero. Appendix \ref{sec:tables} contains tables of decompositions.

\subsection*{Acknowledgements}
I want to thank Viktoriya Ozornova for many discussions and computations, on which several of the ideas of the present article are based. Moreover, her comments on earlier versions of this paper have been a great help. I thank K\k{e}stutis \v{C}esnavi\v{c}ius for answering questions about $\MMb_0(n)$ and Gabor Wiese for help with computer calculations. I furthermore thank the mathoverflow community for many helpful questions and answers; in particular, the user Electric Penguin for their argument for lifting the Hasse invariant, which is crucial for significant parts of the present paper. The author was supported by SPP 1786 of the DFG. 

\subsection*{Conventions}
All rings and algebras will be assumed to be commutative and with unity (except when clearly otherwise). The symbol $/$ applied to a group action on a scheme will always denote \emph{stack} quotients. The dual of a module $M$ will be denoted by $M^{\vee}$ and similarly for sheaves. For an abelian group $A$ and a natural number $n$, we will denote by $A[n]$ its $n$-torsion. We will sometimes use the notation $C_n$ for the cyclic group of order $n$. 

\section{Background on modular curves}\label{sec:ModularCurves}
This section collects and proves background results about modular curves. Significant portions of its first two subsections are certainly known to experts, basic references being \cite{D-R73} and \cite{Con07}. In the interest of completeness, we provide proofs here without claiming originality. 

\subsection{Basics and examples}\label{sec:basics}
Denote by $\MM_{ell}$ the uncompactified moduli stack of elliptic curves and by $\MMb_{ell}$ its compactification. We define the stacks
$\MM_0(n)$, $\MM_1(n)$ and $\MM(n)$ by
\begin{align*}
 \MM_0(n)(S) &= \text{ Elliptic curves }E \text{ over }S\text{ with chosen cyclic subgroup scheme of order }n \\
 \MM_1(n)(S) &= \text{ Elliptic curves }E \text{ over }S\text{ with chosen point }P\in E(S)\text{ of exact order }n \\
 \MM(n)(S) &= \text{ Elliptic curves }E \text{ over }S\text{ with chosen isomorphism }(\Z/n)^2\cong E[n](S), 
\end{align*}
where we always assume $n$ to be invertible on $S$ and where $E[n]$ denotes the $n$-torsion points. More precisely, we demand for $\MM_1(n)$ that for every geometric point $s\colon \Spec K \to S$ the pullback $s^*P$ spans a cyclic subgroup of order $n$ in $E(K)$ or, equivalently, that $P$ defines a closed immersion $(\Z/n)_S \to E$. Moreover, we call a group scheme over $S$ \emph{cyclic} if it is \'etale locally isomorphic to $(\Z/n)_S$.  

We can define the compactified versions $\MMb_0(n)$, $\MMb_1(n)$ and $\MMb(n)$ as the normalizations of $\MMb_{ell}$ in $\MM_0(n)$, $\MM_1(n)$ and $\MM(n)$, respectively \cite[IV.3]{D-R73}. These are all Deligne--Mumford stacks. For the corresponding modular interpretations see also \cite{Con07} and \cite{Ces17}. These moduli interpretations are based on the notion of a \emph{generalized elliptic curve}, which we will recall only over an algebraically closed field. By \cite[Lemme II.1.3]{D-R73}, a generalized elliptic curve is in this case either a (smooth) elliptic curve or a N\'eron $k$-gon. The \emph{N\'eron $k$-gon} $C$ over a scheme $S$ is the scheme quotient of $\Z/k\times \Prj^1_S$ where we identify $(i, \infty)$ with $(i+1,0)$ for all $i$. Its smooth part $C^{reg}$ is isomorphic to $\Z/k \times \G_{m,S}$. With its obvious group structure, $C^{reg}$ acts on $C$. See \cite[II.1]{D-R73} or \cite[Section 2.1]{Con07} for more details.

The stack $\MMb_1(n)$ classifies generalized elliptic curves $E$ with a chosen point of exact order $n$ in the smooth part of $E$ satisfying the following condition: Over every geometric point of the base scheme every irreducible component of $E$ contains a multiple of $P$. 
For $n$ \emph{squarefree} $\MMb_0(n)$ classifies generalized elliptic curves $E$ with a chosen cyclic subgroup $H$ of order $n$ in the smooth part of $E$ satisfying the analogous condition: Over every geometric point every irreducible component of $E$ intersects $H$ nontrivially. By definition, such an $H$ is \'etale locally isomorphic to the constant group scheme $\Z/n$. It follows that in the squarefree case $\MMb_0(n)$ is equivalent to the quotient of $\MMb_1(n)$ by the obvious $(\Z/n)^\times$-action. In general, the $(\Z/n)^\times$ equivariant map $\MMb_1(n) \to \MMb_0(n)$ induces a map 
 $$c\colon \MMb_0(n)' := \MMb_1(n)/(\Z/n)^\times \to \MMb_0(n),$$
 which we will study in more detail in Proposition \ref{prop:square}.

The stack $\MMb(n)$ classifies generalized elliptic curves $E$ over $S$ with a chosen isomorphism $\alpha\colon (\Z/n)^2_S\to E^{sm}[n]$. If $E$ is smooth, the Weil or $e_n$-pairing \cite[2.8.5]{K-M85} of $\alpha(0,1)$ and $\alpha(1,0)$ is an $n$-th root of unity on $S$, which induces a morphism $\MM(n) \to \Spec \Z[\frac1n, \zeta_n]$ that is easily seen to be surjective. By \cite[Theorem 4.1.1]{Con07} this extends to a morphism $\MMb(n) \to \Spec \Z[\frac1n, \zeta_n]$. If $R$ contains a primitive $n$-th root of unity, the product $\Spec \Z[\zeta_n] \times \Spec R$ is canonically isomorphic to $\coprod_{\zeta \in \mu_n(R)} \Spec R$ and thus $\MMb(n)_R$ decomposes as $\coprod_{\zeta \in \mu_n(R)} \MMb(n)_{R,\zeta}$. 

Going back to the base case, $\MMb_{ell}$ just classifies generalized elliptic curves, which are over an algebraically closed field either smooth or a N\'eron $1$-gon. A cubic curve of the form
$$y^2 +a_1xy + a_3y = x^3+a_2x^2+a_4x+a_6$$
is called a Weierstra{\ss} curve and defines a generalized elliptic curve if and only if certain quantities $\Delta$ and $c_4$ are nowhere vanishing \cite[Prop. III.1.4]{Sil09}.

For the universal generalized elliptic curve $p\colon \CC\to \MMb_{ell}$ define $\omega = p_*\Omega^1_{\CC/\MMb_{ell}}$, which is known to be a line bundle \cite[Proposition II.1.6]{D-R73} and actually to generate $\Pic(\MMb_{ell})$ \cite{F-O10}. There is another interpretation: Consider the moduli stack $\MMb_{ell}^1$ of generalized elliptic curves with a chosen invariant differential. Quasi-coherent sheaves on $\MMb_{ell}$ correspond to graded quasi-coherent sheaves on $\MMb_{ell}^1$ and $\omega$ corresponds to $\OO_{\MMb_{ell}^1}$ viewed as concentrated in degree $1$. 

%We will sometimes also use the notation $\MMb(\Gamma_0(n))$ for $\MMb_0(n)$ and similarly $\MMb(\Gamma_1(n))$ and $\MMb(\Gamma(n))$. 

\begin{examples}\label{exa:moduli}
 Denote by $\PP_R(a,b)$ the weighted projective stack $(\mathbb{A}^2_R -\{0,0\})/\Gm$ (as in Definition \ref{def:wps}). We have equivalences
 \begin{align*}
  \MMb_{ell, \Z[\frac16]} &\simeq \PP_{\Z[\frac16]}(4,6) \\
  \MMb_1(2) &\simeq \PP_{\Z[\frac12]}(2,4) \\
  \MMb_1(3) &\simeq \PP_{\Z[\frac13]}(1,3) \\
  \MMb(2) &\simeq \PP_{\Z[\frac12]}(2,2) \\
  \MMb_1(4) &\simeq \PP_{\Z[\frac12]}(1,2).
 \end{align*}
 and in each case the pullback of $\omega$ to the weighted projective line is isomorphic to $\OO(1)$. 
These equivalences are classically well-known and we obtain the corresponding uncompactified moduli by taking the non-vanishing locus of $\Delta$. Proofs of (most of) the second, third and fourth equivalence can be found, for example, in \cite[Sec 1.3]{Beh06}, \cite[Prop 4.5]{H-M17} and \cite[Prop 7.1]{Sto12} respectively. We give a sketch of the fifth one as this is probably the hardest to find in the literature. Given an elliptic curve with a chosen invariant differential and a point $P$ of exact order $4$, we can write it uniquely in the form 
$$y^2 +a_1xy +a_3y = x^3+a_2x^2$$
such that $P = (0,0)$ and $\frac{dx}{2y+a_1x+a_3}$ is the chosen invariant differential; this is sometimes called the \emph{homogeneous Tate normal form} (see \cite[Section 4.4]{HusElliptic} or \cite[Section 1]{B-O16}). The condition that $(0,0)$ is a point of order $4$ is equivalent to $a_3 = a_1a_2$. Thus, we obtain an equivalence 
$$\MM_1(4) \simeq (\Spec \Z[\frac12][a_1,a_2, \Delta^{-1}])/\Gm$$
with $\Delta = a_1^2a_2^4(a_1^2-16a_2)$. 

The map 
$$(\Spec \Z[\frac12][a_1,a_2, \Delta^{-1}])/\Gm \to \MM_{ell, \Z[\frac12]}$$
extends to a map 
$$f\colon (\Spec \Z[\frac12][a_1,a_2])/\Gm \to \MM_{cub,\Z[\frac12]},$$
where $\MM_{cub}$ is the stack classifying all curves defined by a cubic equation \cite[Section 3.1]{Mathom} and $f$ classifies the cubic curve
$$y^2 +a_1xy +a_1a_2y = x^3+a_2x^2.$$
Let $A = \Z[\frac12][a_1,a_2,a_3,a_4,a_6]$ and consider the fpqc morphism $\Spec A \to \MM_{cub,\Z[\frac12]}$ classifying the universal Weierstra{\ss} curve
$$y^2 +a_1xy + a_3y = x^3+a_2x^2+a_4x+a_6.$$
Then 
$$\Spec A \times_{\MM_{cub,\Z[\frac12]}} (\Spec \Z[\frac12][a_1,a_2])/\Gm
\simeq \Z[\frac12][a_1,a_2][r,s,t]$$
as the morphisms of Weierstra{\ss} curves (preserving an invariant differential) are classified by parameters $r,s,t$ (see \cite[Section III.1]{Sil09}). Thus $f$ is representable and affine.  Using that $c_4 = a_1^4 - 16a_1^2a_2 + 16a_2^2$ it is easy to see that $c_4(a_1,a_2) = \Delta(a_1,a_2) = 0$ if and only if $a_1 = a_2 = 0$ with $a_1,a_2$ in a field of characteristic $\neq 2$. The pullback of $f$ along $\MMb_{ell,\Z[\frac12]} \to \MM_{cub,\Z[\frac12]}$ is thus a map 
$$f'\colon \PP_{\Z[\frac12]}(1,2) \to \MMb_{ell,\Z[\frac12]}.$$
Clearly, $f'$ is still affine and it is also proper by the valuative criteria \cite[Section 7]{LMB00} because the source is proper \cite[Section 2]{Mei13} and the target separated over $\Z[\frac12]$. Thus, $f'$ is finite. As $\PP_{\Z[\frac12]}(1,2)$ is normal, this implies that $\MMb_1(4) \simeq \PP_{\Z[\frac12]}(1,2)$ by the uniqueness of normal compactifications (see e.g.\ \cite[Lemma 4.4]{H-M17}). We have an isomorphism $f^*\omega \cong \OO(1)$ because $f'$ is induced by a $\Gm$-equivariant map $\A^2_{\Z[\frac12]} -\{0\} \to \MMb_{ell,\Z[\frac12]}^1$. 
\end{examples}

We call a Deligne--Mumford stack $\XX$ \emph{tame} if the automorphism group of every geometric point $\Spec \overline{K} \to \XX$ has order prime to the characteristic of $\overline{K}$. If $\XX$ is separated, it has by the Keel--Mori theorem \cite{Con05} a coarse moduli space $X$ and we denote the canonical map $\XX \to X$ by $p$. Then $\XX$ is tame if and only if the pushforward functor
$$p_*\colon \QCoh(\XX) \to \QCoh(X)$$
is exact as proven in \cite{AOV08} (note that while they work with Artin stacks, their theory simplifies in the case of Deligne--Mumford stacks because automorphism group schemes of geometric points are in this case \'etale and hence constant). For example $\PP_R(w_0,\dots, w_n)$ is tame if and only if all $w_i$ are invertible in $R$ by \cite[Rem. 2.2]{Mei13}. In particular, all the examples in Examples \ref{exa:moduli} are tame. 

\begin{lemma}\label{lem:tame}
Let $f\colon \XX \to \YY$ be a representable morphism into a tame Deligne--Mumford stack. Then $\XX$ is tame as well. 
\end{lemma}
\begin{proof}
Let $x\colon \Spec K \to \XX$ be a geometric point and $y$ its image in $\YY$. This defines a geometric point in the pullback $\Spec K \times_{\YY}\XX$ whose (trivial) automorphism group is the kernel of $\Aut(x) \to \Aut(y)$. Thus $\Aut(x) \subset \Aut(y)$ and $\XX$ is tame. 
\end{proof}

We will mainly work with moduli stacks of elliptic curves in the tame or even representable case and specifically with the class singled out in the following convention.
\begin{defi}\label{def:tame}
 Let $\Gamma = \Gamma_0(n), \Gamma_1(n)$ or $\Gamma(n)$. We set $\MMb(\Gamma) = \MMb_0(n),\, \MMb_1(n)$ or $\MMb(n)$ respectively and similarly for $\MM(\Gamma)$. For $n\geq 2$ and $R$ a $\Z[\frac1n]$-algebra, we say that $\Gamma$ is \emph{tame} with respect to $R$ if $\Gamma$ is $\Gamma_1(n), \Gamma(n)$ or $\Gamma_0(n)$ and we additionally demand in the final case that $\gcd(6,\phi(n))$ is invertible in $R$. 
 
 While $\MMb(\Gamma)_R$ denotes $\MMb(\Gamma) \times_{\Spec \Z[\frac1n]} \Spec R$ for $\Gamma = \Gamma_1(n)$ or $\Gamma_0(n)$, we will always assume that $R$ is an $\Z[\frac1n,\zeta_n]$-algebra when we speak about $\MMb(\Gamma(n))_R$ and denote by this $\MMb(\Gamma) \times_{\Spec \Z[\frac1n,\zeta_n]} \Spec R$; in contrast, $\MMb(n)_R$ will mean $\MM(n)\times_{\Spec \Z[\frac1n]} \Spec R$. The reason for this choice is that we want $\MMb(\Gamma)$ to be geometrically irreducible (see Proposition \ref{prop:irreducible}). 
 
 We will denote the projection $\MMb(\Gamma)_R \to \MMb_{ell,R}$ by $g$ and give the same name to the restriction $\MM(\Gamma)_R \to \MM_{ell,R}$. We will also sometimes use the notation $f_n$ for $g$ if $\Gamma = \Gamma_1(n)$. 
\end{defi}

\begin{prop}\label{prop:basicprops}Let $R$ be a $\Z[\frac1n]$-algebra and $\Gamma = \Gamma_0(n), \Gamma_1(n)$ or $\Gamma(n)$.
\begin{enumerate}
 \item The map $g\colon \MMb(\Gamma)_R \to \MMb_{ell,R}$ is finite, representable and flat.
 \item If $\Gamma$ is tame with respect to $R$, the stack $\MMb(\Gamma)_R$ is tame. In fact, $\MMb_1(n)_R$ (for $n\geq 5$) and $\MMb(n)_R$ (for $n\geq 3$) are even representable by projective $R$-schemes. In these cases, $\MM(\Gamma)_R$ is affine.
 \item The map $\MMb(\Gamma)_R \to \Spec R$ is in the representable case smooth of relative dimension $1$. 
 \item If $\Gamma$ is tame, we have $H^i(\MMb(\Gamma)_R;\FF) = 0$ for $i\geq 2$ for every quasi-coherent sheaf $\FF$ on $\MMb(\Gamma)$, and $H^i(\MM(\Gamma);\FF) = 0$ for $i\geq 1$ for every quasi-coherent sheaf $\FF$ on $\MM(\Gamma)$. 
\end{enumerate}
\end{prop}
\begin{proof}\lbreak
\begin{enumerate}
\item The map $g$ being integral and representable follows from the definition of normalization. By \cite[03GR]{STACKS} it is also finite because $\MMb_{ell}$ has a smooth cover by a Nagata scheme, e.g.\ by the union of the non-vanishing loci of $c_4$ and $\Delta$ in $\Spec \Z[a_1,a_2,a_3,a_4,a_6]$ . 

Furthermore, both $\MMb(\Gamma)$ and $\MMb_{ell}$ are smooth over $\Spec R$ by Theorem 3.4 of \cite{D-R73}. Every finite map between Deligne--Mumford stacks that are smooth over $\Spec R$ is automatically flat if $R$ is regular. By choosing an \'etale cover, this follows from the affine case, which in turn follows from \cite[Prop 6.1.5]{EGAIV.2}. As the universal case $R = \Z[\frac1n]$ is regular, flatness follows for all $R$. 

\item By the examples from Examples \ref{exa:moduli}, we see that $\MMb_1(n)$ is tame for $2\leq n \leq 4$ and $\MMb(n)$ is tame for $n=2$.  

Next we will show that the automorphism groups of $K$-valued geometric points for $\MMb_1(n)$ and $n\geq 5$ and for $\MMb(n)$ and $n\geq 3$ are trivial. In the interior, this follows from \cite[Cor.\ 2.7.2]{K-M85}. Now consider a geometric point of $\MMb_1(n)$ not in the interior. This corresponds to a N\'eron $k$-gon with a point $P = (i,x)$ of exact order $n$ in the smooth part such that $i$ is a generator of $\Z/k$. 

For a $k$-th root of unity $\zeta$, there are automorphisms $\tau$ and $u_{\zeta}$ of the N\'eron $k$-gon; on $(i,x) \in \Z/k \times \Gm(K)$ these are defined as 
$$\tau\colon (i,x) \mapsto (-i, x^{-1}) \quad \text{ and } \quad u_{\zeta}\colon (i,x) \mapsto (i, \zeta^i x).$$
Every automorphism of the N\'eron $k$-gon that preserves the group operation is of the form $u_{\zeta}$ or $\tau u_{\zeta}$ for a $k$-th root of unity $\zeta$ \cite[Prop.\ II.1.10]{D-R73}. 
As $n\geq 2$, we have $k=1$ or $i \neq [0]$ and thus $P$ cannot be fixed by $u_{\zeta}$ if $\zeta \neq 1$. If $P$ is fixed by $\tau u_{\zeta}$, then $i = -i$ and thus $k= 1$ or $2$, which implies $\zeta^2 = 1$. As $x = \zeta^ix^{-1}$, this shows that $x^4 = 1$. Thus, $P$ is a $4$-torsion point, contrary to the assumption that $n\geq 5$. Thus, we see that all automorphisms of geometric points of $\MMb_1(n)$ are trivial if $n\geq 5$. By the same arguments, the analogous statement follows for $\MMb(\Gamma(n))$ if $n\geq 3$ because an isomorphism $(\Z/n)^2 \cong C^{reg}[n]$ for a N\'eron $k$-gon $C$ implies that $k=n$.

By \cite[Theorem 2.2.5]{Con07} it follows that $\MMb(\Gamma)_R$ is an algebraic space for $\Gamma = \Gamma_1(n)$ for $n\geq 5$ or $\Gamma(n)$ for $n\geq 3$. The coarse moduli space of $\MMb_{ell, R}$ is $\mathbb{P}_R^1$ by \cite[VI.1]{D-R73} for $R=\Z$ and \cite[Prop 3.3.2]{Ces17} in the general case. As the map $\MMb(\Gamma)_R \to \MMb_{ell, R}$ is finite, the composition $\MMb(\Gamma)_R \to \mathbb{P}^1_R$ with the map $\MMb_{ell,R} \to \mathbb{P}_R^1$ is proper and quasi-finite as the map into the coarse moduli space is proper and quasi-finite \cite{Con05}. Thus, $\MMb(\Gamma)_R$ is a scheme by \cite[Cor 6.16]{Knutson} and then automatically a projective scheme over $R$ as a proper and quasi-finite map of schemes is finite and hence projective.  

If $\MMb(\Gamma)_R$ is representable by a scheme, then $\MM(\Gamma)_R$ is as well. The coarse moduli scheme of $\MM_{ell,R}$ is $\A^1_R$ and the composition $\MM(\Gamma)_R \to \MM_{ell,R} \to \A^1_R$ is finite again. Thus, $\MM(\Gamma)_R$ is an affine scheme if $\MM(\Gamma)_R$ is representable. 

It remains to discuss the case of $\MMb_0(n)_R$. As this stack is representable over $\MMb_{ell,R}$ and the orders of automorphism groups of elliptic curves can only have the prime factors $2$ and $3$, the same is true for the automorphism groups of points of $\MMb_0(n)_R$. The open substack $\MM_0(n)_R$ is thus tame because $\gcd(\phi(n),6)$ is invertible in $R$ and $\MM_0(n)_R$ is the quotient of the tame stack $\MM_1(n)_R$ by $(\Z/n)^\times$. Moreover, the cusp points of $\MMb_{ell}$ have automorphism group $\Z/2$, and thus the automorphism groups of the cusp points in $\MMb_0(n)_R$ can have order $2$ at most as well. As $\phi(n)$ is always even for $n\geq 2$, we know that $2$ is invertible in $R$. Thus, $\MMb_0(n)_R$ is tame. 

\item By \cite[Thm IV.3.4]{D-R73}, $\MMb(\Gamma)_R$ is smooth over $\Spec R$ and clearly of relative dimension $1$. 

\item Under our assumptions, the case $\MMb_0(n)_R$ reduces to $\MMb_1(n)_R$ as follows: We can assume that $R$ is $p$-local. If $p>3$, then $\MMb_{ell,R} \simeq \PP_R(4,6)$ itself has cohomological dimension $1$ and $\MMb_0(n)_R$ is finite over $\MMb_{ell,R}$ and thus $\MMb_0(n)_{R}$ has cohomological dimension $1$ as well. If $p=2$ or $3$, we know that $\phi(n)$ is invertible in $R$; denote by $\pi$ the canonical map $\MMb_1(n)_R \to \MMb_0(n)'_R$ (which is a $(\Z/n)^\times$-Galois cover). Furthermore, let $\FF$ be a quasi-coherent sheaf on $\MMb_0(n)_R$ and consider the map $c\colon \MMb_0(n)_R' \to \MMb_0(n)_R$ from above. In this case, the descent spectral sequence
$$H^j((\Z/n)^\times, H^i(\MMb_1(n)_R; \pi^*c^*\FF)) \Rightarrow H^{i+j}(\MMb_0(n)'_R;c^*\FF)$$
collapses to isomorphisms 
$$H^i(\MMb_0(n)'_R; c^*\FF) \cong H^i(\MMb_1(n)_R; \pi^*c^*\FF)^{(\Z/n)^\times}.$$
We will show in Proposition \ref{prop:square} that $H^i(\MMb_0(n); \FF) \cong H^i(\MMb_0(n)'_R; c^*\FF)$.

%  Then the map $\FF \to (p_*p^*\FF)^{(\Z/n)^\times}$ is an isomorphism and thus we have a chain of isomorphisms
%\begin{align*}
% H^i(\MMb_0(n)_R; \FF) &\cong H^i(\MMb_0(n)_R; (p_*p^*\FF)^{(\Z/n)^\times}) \\
%	      &\cong H^i(\MMb_0(n)_R; p_*p^*\FF)^{(\Z/n)^\times}\\
%	      &\cong H^i(\MMb_1(n)_R; p^*\FF)^{(\Z/n)^\times}
%\end{align*}
%because the order of $(\Z/n)^\times$ is invertible in $R$. 

The cases $\MMb_1(n)_R$ and $\MMb(n)_R$ are either treated in the Examples \ref{exa:moduli} (where one clearly has cohomological dimension $1$) or are representable. In the latter case, our statement for $\XXb$ follows from the item 3 (e.g.\ by reducing via \cite[Prop 9.3]{Har77} to the case of $R$ being a field). 

Similarly, we can reduce the case $\MM_0(n)_R$ to $\MM_1(n)_R$ and $\MM(n)_R$. In the representable case, these are affine. The Examples \ref{exa:moduli} can be treated by hand again.
\qedhere
\end{enumerate}
\end{proof}

\begin{example}\label{exa:5-12}
 For $n=5,\dots, 10$ or $n=12$, we have an equivalence $\MMb_1(n) \simeq \Prj^1_{\Z[\frac1n]}$. Indeed, by the last proposition, $\MMb_1(n)$ is representable by a projective $\Z[\frac1n]$-scheme. Over $\C$, the scheme $\MMb_1(n)$ is connected of genus zero (for the genus formula see for example \cite[Section 3.9]{D-S05}). As in the discussion in \cite[Section 3.3]{H-L10b}, this implies that $\MMb_1(n) \simeq \Prj^1_{\Z[\frac1n]}$ as soon as we have exhibited a $\Q$-valued point of $\MMb_1(n)$. This is easily done as a N\'eron $n$-gon with $\Gamma_1(n)$-level structure already exists over $\Q$. 
\end{example}

As already alluded to above, $\MMb_0(n)$ is more difficult to understand if $n$ is not squarefree. In many situations, we can use the following proposition though, which follows from the results of \cite{Ces17}. We remark that he uses the notation $\XX_0(n)$ for what we call $\MMb_0(n)$ etc.  
\begin{prop}\label{prop:square}
 The map 
 $$c\colon \MMb_0(n)' = \MMb_1(n)/(\Z/n)^\times \to \MMb_0(n),$$
 has the following properties. 
 \begin{enumerate}
  \item For every quasi-coherent sheaf $\FF$ on $\MMb_0(n)$, the canonical map $\FF \to c_*c^*\FF$ is an isomorphism. 
  \item For every quasi-coherent sheaf $\GG$ on $\MMb_0(n)'$, the canonical map 
  $$H^i(\MMb_0(n)'; \GG) \to H^i(\MMb_0(n); c_*\GG)$$
  is an isomorphism for all $i\geq 0$. 
 \end{enumerate}
 In particular, the map $H^0(\MMb_0(n); \FF) \to H^0(\MMb_1(n); h^*\FF)^{(\Z/n)^\times}$ is an isomorphism for every quasi-coherent sheaf $\FF$ on $\MMb_0(n)$ and $h\colon \MMb_1(n) \to \MMb_0(n)$ the canonical map.
\end{prop}
\begin{proof}
The common nonvanishing locus $D$ of $j$ and $j-1728$ on $\MMb_{ell}$ is of the form $X/C_2$ with the $C_2$-action on $X = \Spec \Z[j, (j(j-1728))^{-1}]$ trivial (see \cite[Lemma 3.2]{Shin} for the details). We denote by
\[
\xymatrix{\MMb_0(n)_X' \ar[r]^{c_X}\ar[d] & \MMb_0(n)_X \ar[d]\\
\MMb_0(n)_D' \ar[r]^{c_D} & \MMb_0(n)_D
}
\]
the base changes of $c$ along the open inclusion $D \to \MMb_{ell}$ and the map $X \to \MMb_{ell}$.  

By the Leray spectral sequence it suffices to show that $R^ic_*\GG = 0$ for every quasi-coherent sheaf $\GG$ on $\MMb_0(n)'$ and $i>0$ to obtain the second claim. As $c$ is an isomorphism on the preimage of $\MM_0(n)$, it suffices to show both the vanishing of $R^ic_*\GG$ and the first claim on $\MMb_0(n)_X$. This is a scheme as $\MMb_0(n) \to \MMb_{ell}$ is representable and thus $\MMb_0(n)_X \to X$ is as well. 

By \cite[Proposition 6.9]{Ces17}, $c$ induces an isomorphism on coarse moduli and thus $c_D$ does as well. The $C_2$-actions on $\MMb_0(n)_X$ and $\MMb_0(n)_X'$ induced by that of $X$ over $D \simeq X/C_2$ are isomorphic to the identity as the latter $C_2$-action is induced by the automorphism $[-1]$ that can be lifted to $\MMb_0(n)'$ and hence also to $\MMb_0(n)$. Thus, $\MMb_0(n)_X\to \MMb_0(n)_D$ induces an isomorphism on coarse moduli spaces and similarly for $\MMb_0(n)'$. Thus, $c_X$ induces an isomorphism on coarse moduli space as well and is thus the map into the coarse moduli itself.

Let $x$ be a geometric point of $\MMb_0(n)_X'$. If $x$ is a cusp point, then $\Aut(x) = \Z/d$ for some divisor $d$ of $n$ by the explicit description of the automorphisms of Neron $k$-gons recalled in the proof of Proposition \ref{prop:basicprops}. 

By (the proof of) Theorem 2.12 from \cite{Ols06}, we can choose an \'etale neighborhood $W = W_x$ of the image of $x$ in $\MMb_0(n)_X$ such that $c_X^*W \cong U/(\Z/d)$ for some scheme $U$. We can even assume that $W$ and $U$ are affine and denote the resulting map 
$\Spec R/(\Z/d) \to \Spec S$ by $\gamma$, where $S = R^{\Z/d}$. As $d$ is invertible, $R^i\gamma_*\GG = 0$ for every quasi-coherent sheaf $\GG$ on $\Spec R/(\Z/d)$ and $i>0$. Moreover, $M \to (M\tensor_R S)^{\Z/d}$ is an isomorphism for all $S$-modules $M$ as it is one for $M = S$ and both sides are right exact and commute with arbitrary direct sums. In other words, $\FF \to \gamma_*\gamma^*\FF$ is an isomorphism for all quasi-coherent sheaves $\FF$ on $\Spec S$. 

As the $W_x$ cover the part of $\MMb_0(n)_X$ not in the preimage of $\MM_0(n)$, the result follows.  
\end{proof}

\subsection{Modular forms and cusp forms} \label{sec:modcusp}
\begin{defi}
 Let $\Gamma \in \{\Gamma_1(n), \Gamma(n), \Gamma_0(n)\}$ and let $R$ be a $\Z[\frac1n]$-algebra. We define 
 \begin{itemize}
  \item $M_k(\Gamma; R)$ as $H^0(\MMb(\Gamma)_R; g^*\omega^{\tensor k})$ and call $M_*(\Gamma; R)$ the ring of \emph{modular forms} for $\Gamma$,
  \item $\widetilde{M}_*(\Gamma; R)$ as $H^0(\MM(\Gamma)_R; g^*\omega^{\tensor *})$, the ring of \emph{weakly holomorphic modular forms} for $\Gamma$,
  \item $S_*(\Gamma; R)$ as $H^0(\MMb(\Gamma)_R; g^*\omega^{\tensor *}\tensor \OO(-\cusps))$, the non-unital ring of \emph{cusp forms} for $\Gamma$. Here, $\mathrm{cusps}$ is the closed substack of cusps (i.e.\ the vanishing locus of the discriminant) and $\OO(-\mathrm{cusps})$ is the corresponding line bundle. Thus, cusp forms are sections of $g^*\omega^{\tensor *}$ vanishing at all cusps. 
 \end{itemize}
\end{defi}

There is an alternative way to define cusp forms using sheaves of differentials. We will also need the notion of logarithmic differentials, which we will sketch now. Let $\XX \to S$ be a smooth Deligne--Mumford stack over a regular base. Let $i\colon \DD \hookrightarrow \XX$ be a \emph{smooth divisor}, by which we mean a closed substack that is \'etale locally cut out by one non-zerodivisor and that is smooth over $S$. There is an associated log structure on $\XX$ (pushed forward from the complement of $\DD$) that we denote by $(\XX,\DD)$; see \cite{Ogus} or \cite{HL13} for the basics of log structures. 

In \cite[Chapter IV]{Ogus}, Ogus defines sheaves of differentials for log schemes and the theory easily generalizes to Deligne--Mumford stacks. In our situation, $\Omega^1_{(\XX,\DD)/S}$ coincides with the more classical sheaf of differentials with logarithmic poles. If $\DD$ is cut out by a single element $f$, the $\OO_{\XX}$-module can be described as the quotient of $\Omega^1_{\XX/S}\oplus \OO_{\XX}\frac{df}f$ by the $\OO_{\XX}$-module generated by $f \frac{df}
f - df$. This sheaf maps injectively into $\Omega^1_{\XX/S}\tensor \OO(\DD)$ by the obvious inclusions on both summands. This comparison map is surjective if every local section of $\Omega^1_{\XX/S}$ is of the form $f\alpha + hdf$ for a one-form $\alpha$, i.e.\ if $df$ generates $i^*\Omega^1_{\XX/S}$. By \cite[Proposition II.8.12]{Har77}, this happens if and only if $\Omega^1_{\DD/S} = 0$, which is automatic if $\XX$ is smooth of relative dimension $1$ over $S$ as $\DD$ is then smooth of relative dimension $0$ over $S$. 

We will apply this to the cusp $\{\infty\}$ on $\MMb_{ell}$, i.e.\ the vanishing locus of $\Delta$.
\begin{lemma}\label{lem:Omega1}
 There is an isomorphisms $\Omega^1_{\MMb_{ell}/\Z} \cong \omega^{\tensor (-10)}$. Moreover, 
 $$\Omega^1_{(\MMb_{ell},\{\infty\})/\Z} \cong \omega^{\tensor 2}.$$ 
\end{lemma}
\begin{proof}
 After base change to $\Z[\frac16]$ the first claim follows from Theorem \ref{thm:fundamentalweighted} as $\MMb_{ell,\Z[\frac16]} \simeq \PP_{\Z[\frac16]}(4,6)$. Because the map $\Pic(\MMb_{ell}) \to \Pic(\MMb_{ell,\Z[\frac16]})$ is injective (even an isomorphism by \cite{F-O10}), we have shown $\Omega^1_{\MMb_{ell}/\Z} \cong \omega^{\tensor (-10)}$. 
 
By using Weierstra\ss{} equations, we can write $\MMb_{ell}$ as the stack quotient of the complement of the common vanishing locus of $\Delta$ and $c_4$ on $\Spec \Z[a_1,a_2,a_3,a_4,a_6]$ by an action of the algebraic group $\Spec \Z[u^{\pm 1},r,s,t]$. Thus, $\MMb_{ell} \to \Spec \Z$ is smooth of relative dimension $5-4 = 1$. By the discussion above, we obtain
 $$\Omega^1_{(\MMb_{ell},\{\infty\})/\Z} \cong \omega^{\tensor (-10)} \tensor \OO(\{\infty\}).$$
 As $\infty$ is exactly the vanishing locus of $\Delta \in H^0(\MMb_{ell};\omega^{\tensor 12})$, we have $\OO(\{\infty\}) = \omega^{\tensor 12}$ and we obtain our result. 
\end{proof}

\begin{prop}
 Let $\XX \to \Spec \Z$ be a smooth Deligne--Mumford stack with a smooth divisor $\DD$. Let $g\colon (\XX,\DD) \to (\MMb_{ell},\{\infty\})$ be a log-\'etale map. Then $\Omega^1_{(\XX,\DD)/\Z} \cong g^*\omega^{\tensor 2}$ and $\Omega^1_{\XX/\Z} \cong g^*\omega^{\tensor 2} \tensor \OO(-\DD)$. 
\end{prop}
\begin{proof}
 By Proposition 3.1.3 and Theorem 3.2.3 of \cite{Ogus}, we obtain 
 $$\Omega^1_{(\XX,\DD)/\Z} \cong g^*\Omega^1_{(\MMb_{ell},\{\infty\})/\Z} \cong g^*\omega^{\tensor 2}.$$
 The discussion above Lemma \ref{lem:Omega1} moreover implies
 $$\Omega^1_{(\XX,\DD)/\Z} \cong \Omega^1_{\XX/\Z}\tensor \OO(\DD),$$ 
 which proves the result.
\end{proof}

To apply this result, we have to check that $(\MMb(\Gamma), \cusps) \to (\MMb_{ell}, \infty)$ is log-\'etale. We can check this after base change to an \'etale cover $X \to \MMb_{ell}$ to reduce to the scheme case and write $D = \infty \times_{\MMb_{ell}} X$. By \cite[7.3b, 7.6]{Ill02}, we only have to check that $\MMb(\Gamma)\times_{\MMb_{ell}} X \to X$ is tamely ramified at the cusps, but this is clear as for every $x\in D$ the local ring $\OO_{X,x}$ can only have dimension $1$ if its residue characteristic is zero. Thus, we obtain:

\begin{cor}\label{cor:Omegaomega}
 For $\Gamma \in \{\Gamma_1(n), \Gamma_0(n),\Gamma(n)\}$, we have
 $$\Omega^1_{\MMb(\Gamma)/\Z} \cong g^*\omega^{\tensor 2} \tensor \OO(-\cusps).$$
\end{cor}
For $\MMb(n)$, this result appears in \cite[Section 1.5]{Kat73} and is certainly well-known to the experts in general. It implies the following proposition.

\begin{prop}\label{cor:cusps}
 For $\Gamma \in \{\Gamma_1(n), \Gamma_0(n),\Gamma(n)\}$, the space of $\Gamma$-cusp forms with coefficients in $R$ and weight $i$ is isomorphic to 
 $$H^0(\MMb(\Gamma)_R; \Omega^1_{\MMb(\Gamma)_R/\Spec R} \tensor g^*\omega^{\tensor (i-2)}).$$
 If $\MMb(\Gamma)_R$ is representable, this in turn is isomorphic to $\Hom_R(H^1(\MMb(\Gamma)_R; g^*\omega^{\tensor (2-i)}), R)$ by Grothendieck duality (e.g.\ in the formulation of \cite[1.1.2]{ConDual}). 
\end{prop}
This is in accordance with the definition given in \cite[Def 2.8]{Del71}. 

Next, we determine the weight $0$ modular forms. We will need the following standard fact. 

\begin{lemma}\label{lem:sectionsirreducible}
 Let $f\colon X \to S$ be a smooth proper morphism with geometrically connected fibers and $S$ locally noetherian. Then $\OO_S \to f_*\OO_X$ is an isomorphism. 
\end{lemma}
\begin{proof}
 The case of $S = \Spec k$ for a field $k$ is \cite[Cor 3.21]{Liu}. We can assume that $S = \Spec R$ is affine and noetherian. By cohomology and base change (\cite[Sec 5, Cor 2]{Mum08}), we see that 
 $$H^0(X,\OO_X)\tensor_R k \to H^0(X_k, \OO_{X,k})\cong k$$
 is an isomorphism for every $\Spec k \to S$. As $H^0(X,\OO_X)$ is a finitely generated flat $R$-module, this implies that the canonical map $R \to H^0(X,\OO_X)$ is an isomorphism. 
\end{proof}

\begin{prop}\label{prop:irreducible}Let $\Gamma \in \{\Gamma_0(n), \Gamma_1(n), \Gamma(n)\}$ and let $R$ be a noetherian integral $\Z[\frac1n]$-algebra, which contains a primitive $n$-th root of unity $\zeta$ if $\Gamma = \Gamma(n)$. 
\begin{enumerate}
 \item If $k=R$ is a field, the stack $\MMb(\Gamma)_k$ is irreducible. 
 \item The ring $M_*(\Gamma; R)$ is an integral domain and the inclusion
 $$R \to M_0(\Gamma; R)$$
 of the constant functions is an isomorphism. 
\end{enumerate}
\end{prop}
\begin{proof}As $\MMb_1(n)_k \to \MMb_0(n)_k$ is closed and surjective, we have to show irreducibility only for $\Gamma = \Gamma(n)$ or $\Gamma_1(n)$, which we assume now. 

 By \cite[038H]{STACKS}, $\MMb(\Gamma)_k$ is irreducible for all $k$ if it is irreducible for $k = \overline{\F}_p$ (for all primes $p$) and for $k = \C$. It suffices to show that $\MM(\Gamma)_k$ is connected since it is then automatically irreducible because of smoothness, and it is dense in $\MMb(\Gamma)_k$. In the case $k=\C$, it can be uniformized by the upper half plane and is hence connected in the complex and thus also in the Zariski topology. 
 
 By Proposition \ref{prop:basicprops}, $\MMb(\Gamma)$ is smooth and proper over its base scheme. Thus by \cite[Thm 4.17]{DeligneMumford}, it is thus also irreducible over $\overline{\mathbb{F}}_p$. This shows the first item. 
 
 By Lemma \ref{lem:sectionsirreducible}, it follows that the inclusion $R \to H^0(\MMb(\Gamma)_R, \OO_{\MMb(\Gamma)_R})$ of constant functions is an isomorphism if $\MMb(\Gamma)_R$ is a scheme. The other cases follow by Examples \ref{exa:moduli} and because by Proposition \ref{prop:square}, the map 
 $$M_*(\Gamma_0(n);R) \to M_*(\Gamma_1(n); R)^{(\Z/n)^\times}$$
 is an isomorphism. 
 
 By the same results it suffices to show that $M_*(\Gamma; R)$ is an integral domain if $\MMb(\Gamma)_R$ is a scheme. Moreover, the field $K$ of fractions of $R$ is flat over $R$ and $M_*(\Gamma;R)$ is $R$-torsionfree; thus $M_*(\Gamma; R)$ embeds into $M_*(\Gamma;K)$.
 
If $f,g \in H^0(\MMb(\Gamma)_K; g^*\omega^{\tensor *})$ are modular forms, we can consider their vanishing loci $V(f)$ and $V(g)$. If $fg = 0$, then $V(f) \cup V(g) = V(fg) =\MMb(\Gamma)_K$. As this space is irreducible, we have $V(f) = \MMb(\Gamma)_K$ or $V(g) = \MMb(\Gamma_K)$. But $\MMb(\Gamma)_K$ is reduced as it is smooth over $K$ and thus a section of a line bundle that vanishes everywhere is actually zero.  
 \end{proof}

\subsection{Cohomology}
Let $R$ be a $\Z[\frac1n]$-algebra and $\Gamma$ tame as in Definition \ref{def:tame}. 
In this subsection, we will collect some information about the cohomology of $g^*\omega^{\tensor m}$ on $\MMb(\Gamma)_R$ and of $\omega^{\tensor m}$ on $\MMb_{ell}$. 

\begin{prop}\label{prop:coh}
 We have
 \begin{enumerate}
  \item \label{cor:nonegative} $H^0(\MMb(\Gamma)_R;g^*\omega^{\tensor m}) = 0$ for $m<0$ (i.e.\ there are no modular forms of negative weight),
  \item \label{lem:vanishingcoh}  $H^1(\MMb(\Gamma)_R;g^*\omega^{\tensor m}) = 0$ for $m\geq 2$,
  \item \label{lem:torsionfree} $H^1(\MMb(\Gamma)_R; g^*\omega^{\tensor m})$ is torsionfree for all $m\neq 1$ if $R$ is torsionfree. 
 \end{enumerate}
\end{prop}
\begin{proof}Throughout the proof, we will write $\omega$ for $g^*\omega$ when the context is clear. 
\begin{enumerate}
 \item  As $H^0(\MMb_0(n)_R; \omega^{\tensor m}) \cong H^0(\MMb_1(n)_R; \omega^{\tensor m})^{(\Z/n)^\times}$ by Proposition \ref{prop:square}, we only have to deal with $\MMb_1(n)_R$ and $\MMb(n)_R$. The non-representable cases from Example \ref{exa:moduli} can be dealt with by hand. 

 Assume now that $\MMb(\Gamma)_R$ is representable. By \cite[VI.4.4]{D-R73}, the line bundle $\omega$ on $\MMb_{ell}$ has degree $\frac1{24}$. Thus, $\omega^{\tensor m}$ has negative degree on $\MMb(\Gamma)_k$ for $m<0$ and every field $k$. Thus, $H^0(\MMb(\Gamma)_k;\omega^{\tensor m}_k) = 0$ by \cite[IV.1.2]{Har77}. Thus, the pushforward of $g^*\omega^{\tensor m}$ to $\Spec R$ vanishes at every point of $\Spec R$ and thus vanishes completely. 
 
 \item This is shown for $\MMb(n)_R$ in \cite[Thm 1.7.1]{Kat73}. We will give the proof in the case of $\MMb_1(n)_R$ to add some details. 
 
 By dealing with the cases $n\leq 4$ by hand, we can assume again that $\MMb_1(n)_R$ is representable by a projective $R$-scheme. By cohomology and base change (see e.g.\ \cite[Theorem 12.11]{Har77}), we see that it is enough to show the claim in the case where $k= R$ is an algebraically closed field. By Corollary \ref{cor:Omegaomega},
 $$\omega^{\tensor m} \cong \Omega^1_{\MMb_1(n)_k/k} \tensor \omega^{\tensor m-2} \tensor \OO(\mathrm{cusps}).$$
 Because $\omega$ has positive degree \cite[VI.4.4]{D-R73} and $m\geq 2$, we see that the degree of $\omega^{\tensor m}$ is bigger than the degree of $\Omega^1_{\MMb_1(n)_k/k}$. By Serre duality
 $$H^1(\MMb_1(n)_k; \omega^{\tensor m}) \cong H^0(\MMb_1(n)_k;\omega^{\tensor -m} \tensor \Omega^1_{\MMb_1(n)_k/k})$$
 and this vanishes as $\omega^{\tensor -m} \tensor \Omega^1_{\MMb_1(n)_k/k}$ has negative degree. 
 
It remains to prove the claim for $\MMb_0(n)_R$ if $\gcd(6,\phi(n))$ is invertible on $R$. As we already know that $H^1(\MMb_1(n)_R;\omega^{\tensor m}) = 0$ for $m\geq 2$, Proposition \ref{prop:square} implies 
$$H^i(\MMb_0(n)_R; \omega^{\tensor m}) \cong H^i((\Z/n)^\times; H^0(\MMb_1(n)_R; \omega^{\tensor m})).$$
Beginning with $i=1$, this is $2$-periodic in $i$. But we know that these cohomology groups vanish for $i>1$ by Proposition \ref{prop:basicprops}. Thus, they have to vanish for $i=1$ as well. 
 
 \item Let $l$ be a prime that does not divide $n$. Consider the short exact sequence
 \[0 \to H^0(\MMb(\Gamma)_R; \omega^{\tensor m})/l \to H^0(\MMb(\Gamma)_{R/l};\omega^{\tensor m}) \to  H^1(\MMb(\Gamma)_R;\omega^{\tensor m})[l] \to 0.\]
First note that the middle group is zero for $m<0$ by Item (\ref{cor:nonegative}) and hence also  $H^1(\MMb(\Gamma)_R;\omega^{\tensor m})[l] = 0$
for $m < 0$. 

The morphism $H^0(\MMb(\Gamma)_R; \omega^{\tensor m})/l \to H^0(\MMb(\Gamma)_{R/l};\omega^{\tensor m})$ is an isomorphism for $m = 0$ by Proposition \ref{prop:irreducible}. Thus, $H^1(\MMb(\Gamma)_R;\omega^{\tensor m})$ can only have torsion for $m=1$ as it vanishes for $m\geq 2$.\qedhere
\end{enumerate}
\end{proof}

Next, we collect some facts about the cohomology of $\MMb_{ell}$ itself, which is certainly not tame if we do not invert $6$. The cohomology of the sheaves $\omega^{\tensor m}$ on $\MMb_{ell}$ was computed by \cite{Konter}, based on \cite{Bau08}. We need essentially only the following.
\begin{prop}\label{prop:cohmell}We have isomorphisms
 \begin{align*}\
  H^1(\MMb_{ell}; \omega) &\cong \Z/2 \cdot \eta, \\
 H^1(\MMb_{ell};\omega^{\tensor 2}) & \cong \Z/12 \cdot \nu.
 \end{align*}
\end{prop}
These classes have a rather classical description as obstructions to lifting the Hasse invariant. Indeed, denote by 
$$A_p \in H^0(\MMb_{ell,\F_p}; \omega^{\tensor (p-1)}) \cong H^0(\MMb_{ell,\Z_{(p)}}; \omega^{\tensor (p-1)}/p) $$
the mod $p$ Hasse invariant (see Appendix \ref{app:Hasse} for a definition). The short exact sequence
$$0 \to \omega^{\tensor (p-1)} \xrightarrow{p} \omega^{\tensor (p-1)} \to \omega^{\tensor (p-1)}/p \to 0$$
on $\MMb_{ell}$ induces a long exact sequence
$$\cdots \to H^0(\MMb_{ell,\Z_{(p)}};\omega^{\tensor (p-1)}) \to H^0(\MMb_{ell,\Z_{(p)}};\omega^{\tensor (p-1)}/p) \xrightarrow{\partial} H^1(\MMb_{ell,\Z_{(p)}};\omega^{\tensor (p-1)}) \to \cdots.$$
Because the $H^1$-term vanishes for $p>3$, there is no obstruction to lift $A_p$ to characteristic zero for $p>3$. As the Hasse invariant does not lift to characteristic zero for $p=2,3$ (there does not even exist a nonzero integral modular form in these degrees), we must have $\partial(A_2) = \eta$ and $\partial(A_3)$ at least a nonzero multiple of $\nu$.

% \begin{remark}
%  There is also a topological interpretation of the classes $\eta$ and $\nu$. These cohomology classes detect the Hurewicz images of the Hopf maps of the same name in $\pi_*Tmf$ in the descent spectral sequence (see \cite{Bau08} for a related statement). 
% \end{remark}

\begin{prop}\label{prop:eta}
 The image of $\eta$ in $H^1(\MMb(\Gamma)_R; g^*\omega)$ is zero if $\Gamma$ is tame. 
\end{prop}
\begin{proof}
 It is enough to show $g^*\eta = 0$ for $\MMb(\Gamma)_R = \MMb_1(n)_{(2)}$ and $n$ odd. Indeed, consider the composite
 \[H^1(\MMb_{ell,\Z_{(2)}};\omega) \to H^1(\MMb_0(n)_{\Z_{(2)}}; g^*\omega) \to H^1(\MMb_1(n)_{\Z_{(2)}};g^*\omega) \to H^1(\MMb(n)_{\Z_{(2)}};g^*\omega).
 \]
 Now we only have to use that the second map is actually an injection (namely the inclusion of the $(\Z/n)^\times$-fixed points by the proof of Proposition \ref{prop:basicprops}). 
 
It is enough to show $g^*\eta = 0$ after base change to $C= \Z_{(2)}[\zeta]$ for a $\zeta$ an $n$-th root of unity. Consider the commutative square
 \[
  \xymatrix{
  0 = H^0(\MMb_{ell,C}; \omega) \ar[r]\ar[d]^{g^*} & H^0(\MMb_{ell,C}; \omega/2) \ar[r]^-{\partial} \ar[d]^{g^*} & H^1(\MMb_{ell,C}; \omega) \ar[d]^{g^*} \\
   H^0(\MMb_1(n)_C; g^*\omega) \ar[r] & H^0(\MMb_1(n)_C; g^*\omega/2) \ar[r]^-{\partial_n}  & H^1(\MMb_1(n)_C; g^*\omega)
  }
 \]
 As $\eta = \partial(A_2)$ is still true over $C$, it is enough to show that $\partial_ng^*A_2 = 0$, i.e.\ that $g^*A_2$ lifts to $H^0(\MMb_1(n)_C; g^*\omega)$. This is exactly the content of Proposition \ref{prop:Hasse}.
\end{proof}

We end this section with some of the basic structure of $H^*(\MMb_0(n)_R;g^*\omega^{\tensor *})$ if $\Gamma_0(n)$ is not tame for $R$. For this, we will need a well-known lemma about transfers. 

\begin{lemma}\label{lem:transfer}
 Let $\XX$ be an algebraic stack and $\Ac$ a sheaf of $\OO_{\XX}$-algebras whose underlying $\OO_{\XX}$-module is a vector bundle of rank $n$. Then there exists an $\OO_{\XX}$-linear transfer map $\Tr_{\Ac}\colon\Ac \to \OO_{\XX}$ with the following properties:
 \begin{itemize}
  \item[(a)] The morphism $\Tr_{\Ac}$ is natural in $\Ac$ and also in $\XX$ in the following sense: If $g\colon \UU \to \XX$ is a morphism, $g^*\Tr_{\Ac}$ corresponds to $\Tr_{g^*\Ac}$ under the isomorphism $g^*\OO_{\XX} \xrightarrow{\cong} \OO_{\UU}$.
  \item[(b)] Precomposing $\Tr_{\Ac}$ with the unit map $e\colon \OO_{\XX}\to \Ac$ yields the multiplication-by-$n$-map on $\OO_{\XX}$.
  \item[(c)] If $f\colon \YY \to \XX$ is a $G$-Galois cover for a finite group $G$, the transfer $\Tr\colon f_*\OO_{\YY} \to \OO_{\XX}$ corresponds under the isomorphism $\OO_{\XX} \xrightarrow{\cong} (f_*\OO_{\YY})^G$ to the sum $\sum_{g\in G}l_g$, where $l_g$ is the action of $g$ on $f_*\OO_{\YY}$. 
 \end{itemize} 
\end{lemma}
\begin{proof}
 We construct $\Tr_{\Ac}$ as the composite
 $$\Ac \to \mathcal{H}om_{\OO_{\XX}}(\Ac,\Ac) \xleftarrow{\cong} \mathcal{H}om_{\OO_{\XX}}(\Ac,\OO_{\XX}) \tensor_{\OO_{\XX}} \Ac \xrightarrow{\ev} \OO_{\XX},$$
 where the middle map is an isomorphism because it is one locally. The map $\Tr_{\Ac}$ is natural by construction. If $\XX = \Spec R$ and $\Ac(\XX)$ is isomorphic to $R^n$ as an $R$-module, the map $\mathcal{H}om_{\OO_{\XX}}(\Ac,\Ac)  \to \OO_{\XX}$ corresponds to the trace map $\Hom_R(R^n, R^n) \to R$. As the the trace of the unit matrix is $n$, the composite $\Tr_{\Ac} e$ equals $n$ locally and hence also globally, which implies (b). 
 
 In the case $\Ac = f_*\OO_{\YY}$ with $f\colon \YY \to \XX$ a $G$-Galois cover, $f_*\OO_{\YY}$ is locally as an algebra isomorphic to $\prod_G \OO_{\XX}$ and the transfer $\prod_G \OO_{\XX} \to \OO_{\XX}$ is the summing map whose postcomposition with the isomorphism $\OO_{\XX} \to \left(\prod_G \OO_{\XX}\right)^G$ equals $\sum_{g\in G}l_g$. As this is true locally, it must be true globally as well and thus (c) is true. 
\end{proof}

\begin{prop}\label{prop:torsion}
 Let $R$ be a $\Z[\frac1n]$-algebra. For every $i>0$ and $m\geq 2$, the groups $H^i(\MMb_0(n)_R;g^*\omega^{\tensor m})$  are $3$-torsion if $R$ is $3$-local and $8$-torsion if $R$ is $2$-local. 
 
 Moreover if $R$ is $3$-local, all $H^i(\MMb_0(n)_R;g^*\omega^{\tensor m})$ for $i>0$ are $c_4$-power torsion and all $H^i(\MM_0(n)_R;g^*\omega^{\tensor m})$ for $i>0$ are even $c_4$- and $c_6$-torsion for $c_4, c_6 \in M_*^{\Z}$.
\end{prop}
\begin{proof}
Let $R$ be a $\Z[\frac1n]$-algebra that is $2$- or $3$-local. By Proposition \ref{prop:square}, 
$$H^i(\MMb_0(n)_R;g^*\omega^{\tensor m}) \cong H^i(\MMb_0(n)_R';g^*\omega^{\tensor m})$$
 and we will work throughout with the latter.

 Let $p$ be a prime that is relatively prime to $6n$ and such that $p^2-1$ is neither divisible by $9$ nor by $16$; for example, we can take $p\equiv 2 \mod 9$ and $p\equiv 3 \mod 16$. Conrad constructs in \cite{Con07} a moduli stack $\MMb(\Gamma_1(p; n))$ classifying generalized elliptic curves $E$ with both a point $P$ of exact order $p$ and a cyclic subgroup $C$ of order $n$ in the smooth locus such that the subgroup generated by $P$ and $C$ meets all irreducible components of all geometric fibers of $E$. By \cite[Lemma 4.2.3]{Con07}, there is a finite ``contraction map'' $\MMb(\Gamma_1(p;n))_R \to \MMb_1(p)_R$. As the target has cohomological dimension $1$ by Proposition \ref{prop:basicprops}, the same is true for the source as well. Because $H^1(\MMb_1(pn)_R; g^*\omega^{\tensor m}) = 0$ for $m\geq 2$, the descent spectral sequence for the Galois cover $\MMb_1(pn)_R \to \MMb(\Gamma_1(p;n))_R$ collapses to isomorphisms  
 $$H^i(\MMb(\Gamma_1(p;n))_R; g^*\omega^{\tensor m}) \cong H^i((\Z/n)^\times; H^0(\MMb_1(pn)_R; g^*\omega^{\tensor m}))$$ 
 for $m\geq 2$. The same argument as in the second part of Proposition \ref{prop:coh} shows that these groups vanish for $i=1$ as they do it for $i>1$. 
 
 We claim that the contraction map $c\colon \MMb(\Gamma_1(p;n))_R \to \MMb_0(n)_R'$ (constructed using \cite[IV.1.2]{D-R73}) is finite and flat. Indeed, the contraction $\MMb_1(pn)_R \to \MMb_1(n)_R$ is finite as both are finite over $\MMb_{ell,R}$. Considering the diagram, 
\[
\xymatrix{
\MMb_1(pn)_R \ar[r]\ar[d] & \MMb(\Gamma_1(p;n))_R \ar[d]^c \\
\MMb_1(n)_R \ar[r] & \MMb_0(n)_R',
}
\] 
we see that $\MMb_1(pn)_R \to \MMb_0(n)_R'$ is finite and thus by Chevalley's theorem \cite[Corollary 12.40]{G-W10} also $c$ is finite; here we use that $\MMb(\Gamma_1(p;n))_R$ is a scheme. By \cite[Theorem 4.1.1]{Con07}, again both source and target of $c$ are regular, and we argue as in the first part of Proposition \ref{prop:basicprops} to deduce that $c$ is flat.
 
 On the part classifying smooth elliptic curves, one easily counts that $c$ has degree $p^2-1$. Thus, for every vector bundle $\FF$ with an $\OO_{\MMb_0(n)}$-algebra structure on $\MMb_0(n)_R'$ the composite
 \[H^i(\MMb_0(n)_R'; \FF) \to H^i(\MMb_0(n)_R'; c_*c^*\FF)\cong H^i(\MMb(\Gamma_1(p;n))_R; c^*\FF) \xrightarrow{\Tr_{c_*c^*\FF}} H^i(\MMb_0(n)_R'; \FF)\]
 is multiplication by $p^2-1$ and the first result follows. 
 
 Let $R$ be $3$-local. The nonvanishing locus $D(c_4)$ of $c_4$ on $\MMb_{ell, R}$ coincides with the nonvanishing locus of the $j$-invariant. Thus, the coarse moduli space of $D(c_4)$ is the complement of $\{0\}$ in $\mathbb{P}^1_R$, which is isomorphic to $\A^1_R$. Moreover, $D(c_4)$ is tame by \cite[Theorem III.10.1]{Sil09}. Thus, for every quasi-coherent sheaf $\FF$ on $D(c_4)$ the cohomology groups $H^i(D(c_4);\FF)$ vanish for $i>0$. As the non-vanishing locus of $g^*c_4$ on $\MMb_0(n)_R$ is finite over $D(c_4)$, we see
 $$H^i(D(g^*c_4); g^*\omega^{\tensor m}) \cong H^i(D(c_4); g_*g^*\omega^{\tensor m}) = 0$$
 for $i>0$. 
 But $H^i(D(g^*c_4); \omega^{\tensor *}) \cong H^i(\MMb_0(n)_R; \omega^{\tensor *})[c_4^{-1}]$ and we see that for $i>0$ the cohomology group $H^i(\MMb_0(n)_R;g^*\omega^{\tensor m})$ is $c_4$-power torsion. 
 
 It follows from Example \ref{exa:moduli} that $\widetilde{M}_*(\Gamma(2);R) = H^0(\MM(2)_R;\omega^{\tensor *}) \cong R[x_2,y_2, \Delta^{-1}]$. The action of $S_3 \cong GL_2(\Z/2)$ on $\MM(2)$ induces an $S_3$-action on this ring of modular forms and Stojanoska computed in \cite[Lemma 7.3]{Sto12} that there is a transposition $\tau\in S_3$ interchanging $x_2$ and $y_2$ and a $3$-cycle $\sigma$ with $\sigma(x_2) =y_2-x_2$ and $\sigma(y_2) = -x_2$. Using Lemma \ref{lem:transfer}, one easily calculates that the transfer
 $$\Tr_f = \sum_{g\in S_3}l_g\colon \widetilde{M}_*(\Gamma(2);R) \cong H^0(\MM^1_{ell,R};f_*\OO_{\MM^1(2)_R}) \to H^0(\MM^1_{ell,R};\OO_{\MM^1_{ell,R}}) \cong \widetilde{M}_*^R$$
 sends $4x_2^2$ to $c_4$ and $-32x_2^2y_2$ to $c_6$. Here, $\MM^1(2)$ classifies elliptic curves with level-$2$ structure and a choice of invariant differential and $f\colon \MM^1(2)_R \to \MM^1_{ell,R}$ is the projection. Consider the pullback square
 \[
  \xymatrix{\PP^1 \ar[d]^{f'}\ar[r]^{g'} & \MM^1(2)_R \ar[d]^f \\
  \MM^1_0(n)_R\ar[r]^g &\MM^1_{ell,R}
  }
 \]
 By naturality, $g^*\Tr_f(z) = \Tr_{f'}((g')^*z)$ for $z\in H^0(\MM^1(2)_R, \OO_{\MM^1(2)_R})$. Moreover, 
 $$\Tr_{f'}(a)\cdot b = \Tr_{f'}(ab)$$
  for $a \in H^0(\PP^1; \OO_{\PP^1})$ and $b \in H^*(\MM^1_0(n);\OO_{\MM^1_0(n)})$ by the $\OO$-linearity of the transfer. This implies that for any such $b$, the classes $c_4b$ and $c_6b$ are in the image of the transfer from $\PP^1$. But as $g$ is finite and $\MM^1(2)$ is affine, $\PP^1$ is affine as well and thus $c_4b$ and $c_6b$ have to vanish if the cohomological degree of $b$ is at least $1$.  
\end{proof}

\section{The existence of decompositions}\label{sec:decexistence}
Throughout this section let $R$ be a $\Z[\frac1n]$-algebra in which $2$ or $3$ is invertible. Set 
\[
 \MMb' = \MMb_R' = \begin{cases}
                    \MMb_1(3)_R & \text{ if }\frac13\in R, \text{ but }\frac12\notin R\\
                    \MMb_1(2)_R & \text{ if }\frac12\in R, \text{ but }\frac13\notin R\\
                    \MMb_{ell,R} & \text{ if }\frac16\in R.
                   \end{cases}
\]
Denote the projection map $\MMb' \to \MMb_{ell,R}$ by $f$ and denote by $\MM'$ the preimage of $\MM_{ell,R}$. Sometimes we also use the notations $\MMb_R'$ and $f_R$ if we want to stress the dependence on $R$. Furthermore let $\Gamma \in \{\Gamma_1(n), \Gamma(n), \Gamma_0(n)\}$ throughout this section be tame with respect to $R$ and denote by $g$ the projection map $\MMb(\Gamma)_R \to \MMb_{ell,R}$, where we assume $R$ to contain a primitive $n$-th root of unity if $\Gamma = \Gamma(n)$. Note that $g_*\OO_{\MMb(\Gamma)_R}$ is a vector bundle as $g$ is finite and flat by \ref{prop:basicprops}. We want to prove the following theorem. 

\begin{thm}
 Let $R$ be a field, a local ring or a $\Z[\frac16]$-algebra and assume $R$ to be normal and noetherian. The vector bundle $g_*\OO_{\MMb(\Gamma)_R}$ splits into vector bundles of the form $f_*\OO_{\MMb'}\tensor \omega^{\tensor k}$ if and only if $H^1(\MMb(\Gamma)_R;g^*\omega)$ is a free $R$-module. The latter conditions holds if $H^1(\MMb(\Gamma); g^*\omega)$ has no $l$-torsion for $l$ not invertible or $0$ in $R$.  
\end{thm}

We will first prove these over a field. Afterward, we will use our results on vector bundles on weighted projective lines to deduce it also over rings under the given conditions.

\subsection{The vector bundle $f_*\OO_{\MMb_R'}$}
The aim of this subsection is to show that $f_*\OO_{\MMb_R'}$ is indecomposable, but decomposes after pullback to $\MMb(\Gamma)_R$; here we use the notation from the beginning of this section. We will need the following standard base change lemma and will often use it implicitly.
\begin{lemma}\label{lem:pushpull}
 Let 
 \[
  \xymatrix{X' \ar[r]^{v}\ar[d]^{q} & X \ar[d]^p \\
  Y' \ar[r]^u & Y }
 \]
 be a cartesian diagram of Deligne--Mumford stacks, where $p$ is representable and affine, and let $\FF$ be a quasi-coherent sheaf on $X$. Then the natural map 
 $$u^*p_*\FF \to q_*v^*\FF$$
 is an isomorphism.
\end{lemma}
\begin{proof}
Note first that the lemma is true without any assumptions on $p$ if $u$ is \'etale because on $U\to Y'$ \'etale, both source and target can be identified with $\FF(U\times_Y X)$. Thus we can work \'etale locally and assume that $Y = \Spec A$ and $Y' = \Spec A'$ are affine schemes and hence also $X = \Spec B$ and $X' = \Spec B\tensor_{A}A'$. If $\FF$ corresponds to the $B$-module $M$, our assertion just becomes
\[M\tensor_A A' \cong M\tensor_B (B\tensor_A A').\qedhere\]
\end{proof}

\begin{prop}\label{lem:indec}
 The vector bundle $(f_R)_*\OO_{\MM'_R}$ is indecomposable. 
\end{prop}
\begin{proof}
 If $\frac16\in R$, the result is clear. Assume $\frac12\notin R$ so that $\MM_R' = \MM_1(3)_R$ and let $R \to k$ be a morphism to an algebraically closed field of characteristic $2$. 
 Consider the elliptic curve $E\colon y^2 +y = x^3$ over $k$. This has, according to \cite{Sil09}, III.10.1, automorphism group $G$ of order $24$.  By \cite[2.7.2]{K-M85}, the morphism $G\to GL_2(\F_3)$ (given by the operation of $G$ on $E[3]$) is injective. Using elementary group theory, $GL_2(\F_3)$ has a unique subgroup of order $24$, namely $SL_2(\F_3)$; thus $G$ embeds onto $SL_2(\F_3)$. This induces a map $e\colon \Spec k / SL_2(\F_3) \to \MM_{ell}$. 
 
Pulling $(f_R)_*\OO_{\MM'_R}$ back along this map gives an $8$-dimensional $k$-vector space $V$ with $SL_2(\F_3)$-action; e.g.\ using Lemma \ref{lem:pushpull} this can be identified with the permutation representation defined by the action of $SL_2(\F_3)$ on the eight points of exact order $3$ in $\F_3^2$. The quaternion subgroup $Q\subset SL_2(\F_3)$ acts freely and transitively on these points; thus, $V$ restricted to $Q$ is isomorphic to the regular representation of $Q$. 
 
 The regular representation $K[P]$ of a finite $p$-group $P$ over a field $K$ of characteristic $p$ is always indecomposable. Indeed, $K[P]^P \cong K$, but every $K$-representation of $P$ has a 
nonzero $P$-fixed vector by \cite[Prop 26]{Ser77}. Thus every nontrivial decomposition of $K[P]$ would yield that $\dim K[P]^P \geq 2$. 

This means that $e^*(f_R)_*\OO_{\MM'_R}$ is indecomposable and hence $(f_R)_*\OO_{\MM'_R}$ as well if $\frac12 \notin R$.  
 If $\frac13\notin R$, we can either do an analogous argument or cite \cite[Cor 4.8]{Mei13}. 
\end{proof}

\begin{prop}\label{lem:splitting}
For $g\colon \MMb(\Gamma)_R \to \MMb_{ell,R}$, the vector bundle $g^*f_*\OO_{\MMb'_R}$ is a direct summand of a sum of line bundles of the form $g^*\omega^{\tensor n}$. If $\frac12\in R$, then $g^*f_*\OO_{\MMb'}$ is the sum of such line bundles itself.
\end{prop}
\begin{proof}
The case $\frac16\in R$ is clear. Assume $\frac13 \notin R$ with $\MMb' = \MMb_1(2)_R$. We will leave the base change to $R$ implicit in the following to simplify notation. By \cite[Prop 4.14]{Mathom}, there are extensions
  \[ 0 \to \OO_{\MMb_{ell}} \to f_*\OO_{\MMb_1(2)} \to E_{\nu} \tensor \omega^{\tensor (-2)} \to 0\]
  and 
  \[0 \to \OO_{\MMb_{ell}} \to E_\nu \to \omega^{\tensor (-2)} \to 0\]
  classified by $\tilde{\nu} \in \Ext^1(E_\nu \tensor \omega^{\tensor (-2)}, \OO_{\MMb_{ell}})$ and $\nu \in \Ext^1(\omega^{\tensor (-2)},\OO_{\MMb_{ell}})$ respectively, where $\nu$ is as in Proposition \ref{prop:cohmell} and $\tilde{\nu}$ is a lift of $\nu$. The classes $g^*\nu$ and $g^*\tilde{\nu}$ are zero by Proposition \ref{prop:coh}. Thus $g^*f_*\OO_{\MMb_1(2)}$ splits into line bundles of the form $g^*\omega^{\tensor n}$. 
  
  The same argument -- only more complicated -- works if $\frac12 \notin R$ and $\MMb'= \MMb_1(3)_R$. Again, we will leave the base change to $R$ implicit. In \cite[Section 4.1]{Mathom}, Mathew constructs a vector bundle $\FF(Q)$ on $\MMb_{ell}$, which arises via two short exact sequences
  $$ 0 \to E_{\eta} \to \FF(Q) \to \omega^{\tensor (-3)} \to 0$$
  and
  $$ 0\to \OO_{\MMb_{ell}} \to E_{\eta} \to \omega^{\tensor (-1)} \to 0.$$
  The latter extension is classified by $\eta \in H^1(\MMb_{ell}; \omega) \cong \Ext^1(\omega^{\tensor (-1)}, \OO_{\MMb_{ell}})$, while the former is classified by a lift $\tilde{\nu}\in\Ext^1(\omega^{\tensor (-3)}, E_\eta)$ of $\nu$. By Proposition \ref{prop:eta}, we know that $g^*\eta = 0$ so that $g^*E_{\eta} \cong \OO_{\MMb(\Gamma)}\oplus g^*\omega^{\tensor(-1)}$. The other extension splits as well because by Proposition \ref{prop:coh} we see that $\Ext^1(g^*\omega^{\tensor (-3)}, g^*E_{\eta}) = 0$. Thus, 
  $$g^*\FF(Q) \cong \OO_{\MMb(\Gamma)}\oplus g^*\omega^{\tensor(-1)} \oplus g^*\omega^{\tensor (-3)}.$$ 
  
    By \cite[Prop 4.7, Cor 4.11]{Mathom}, the cokernel of the coevaluation map 
    $$q\colon \OO_{\MMb_{ell}} \to \FF(Q) \tensor \FF(Q)^{\vee}$$
    is isomorphic to $f_*\OO_{\MMb_1(3)}$.
    The composition of $q$ with the evaluation map $\FF(Q) \tensor \FF(Q)^{\vee} \to \OO_{\MMb_{ell}}$ equals multiplication by $3 = \rk(\FF(Q))$, which is invertible. Hence, $f_*\OO_{\MMb_1(3)}$ splits off from $\FF(Q) \tensor \FF(Q)^{\vee}$ as a direct summand. Thus, $g^*f_*\OO_{\MMb_1(3)}$ is a direct summand of a sum of line bundles of the form $g^*\omega^{\tensor i}$. 
\end{proof}

\subsection{Decompositions over a field}\label{sec:workingfield}
In this section, we let $R = k$ be a field and $\MMb'$ and $\MMb(\Gamma) = \MMb(\Gamma)_k$ for $\Gamma$ tame as before. Our goal is to show the following proposition.

\begin{prop}\label{prop:split}
Let $\EE$ be a vector bundle on $\MMb(\Gamma)$. Then $g_*\EE$ decomposes into a direct sum of vector bundles of the form $f_*\OO_{\MMb'} \tensor \omega^{\tensor i}$. 
\end{prop}
If $\Char(k)$ is not $2$ or $3$, the proof is easy. In this case, we have an equivalence $\MMb_{ell,k} \simeq \PP_k(4,6)$ with $\omega$ corresponding to $\OO(1)$. By Proposition \ref{prop:ClassField}, \emph{every} vector bundle on $\PP_k(4,6)$ decomposes into a sum of the line bundles $\OO(i)$. Thus, we will assume that $k$ has characteristic $l=2$ or $3$ in the following. 

Our strategy will be to analyze the consequences of Lemma \ref{lem:pushpull} in the case of the pullback square
 \[\xymatrix{
  \Yb = \MMb(\Gamma) \times_{\MMb_{ell,k}} \MMb'_k \ar[r]^-{g'} \ar[d]^{f'} & \MMb_k' \ar[d]^f \\
  \MMb(\Gamma) \ar[r]^{g} &  \MMb_{ell,k}
 }\]
 using a Krull--Schmidt theorem. 

\begin{prop}[Krull--Schmidt]\label{prop:Krull-Schmidt}
 Let $\XX$ be a proper Artin stack over a field $k$. Then the Krull--Schmidt theorem holds for coherent sheaves on $\XX$. This means that every coherent sheaf on $\XX$ decomposes into finitely many indecomposables and that this decomposition is unique up to permutation of the summands.
\end{prop}
\begin{proof}
 As shown by Atiyah in \cite{Ati56}, a $k$-linear abelian category has a Krull--Schmidt theorem if all $\Hom$-vector spaces are finite-dimensional. By a theorem of Faltings \cite{Fal03}, the global sections of any coherent sheaf on $\XX$ form a finite-dimensional $k$-vector space. We can apply this to the Hom-sheaf $\mathcal{H}om_{\OO_\XX}(\FF,\GG)$ for two coherent $\OO_\XX$-modules, which is coherent itself.
\end{proof}

\begin{example}\label{exa:KS}
 We give a counterexample to the Krull--Schmidt theorem if we do not assume properness. It follows from \cite[Prop 4.14]{Mathom} that $f_*\OO_{\MM_1(2)}$ splits as $\OO_{\MM_{ell}} \oplus \omega^{\tensor 2}\oplus \omega^{\tensor 4}$ after rationalization, where we denote in this example by $f$ the projection $\MM_1(2) \to \MM_{ell}$. By \cite[Sec 1.3]{Beh06}, there is an equivalence 
 $$\MM_1(2) \simeq \Spec \Z[\tfrac12][b_2,b_4][\Delta^{-1}]/\G_m,$$ 
 where $\Delta = \frac14b_4^2(b_2^2-32b_4)$. As $\Delta$ is divisible by $b_4$, the ring $\Z[\tfrac12][b_2,b_4][\Delta^{-1}]$ is $4$-periodic and we deduce that $f^*\omega^{\tensor 4} \cong \OO_{\MM_1(2)}$. In particular, it follows that 
$$\OO_{\MM_{ell}} \oplus \omega^{\tensor 2}\oplus \omega^{\tensor 4} \cong f_*\OO_{\MM_1(2)} \cong f_*\OO_{\MM_1(2)} \tensor \omega^{\tensor 4} \cong \omega^{\tensor 4}\oplus \omega^{\tensor 6} \oplus \omega^{\tensor 8}$$
on $\MM_{ell,\Q}$, contradicting a possible Krull--Schmidt theorem in the uncompactified case. With some extra work one can remove the summand $\omega^{\tensor 4}$ from both sides. 
\end{example}

\begin{proof}[Proof of Proposition \ref{prop:split}:]
 Consider again the pullback diagram
 \[\xymatrix{
  \Yb = \MMb(\Gamma)_k \times_{\MMb_{ell,k}} \MMb'_k \ar[r]^-{g'} \ar[d]^{f'} & \MMb'_k \ar[d]^f \\
  \MMb(\Gamma)_k \ar[r]^{g} &  \MMb_{ell,k}.
 }\]
 We have an isomorphism
$$g_*f'_*(f')^*\EE \cong f_*g'_* (f')^*\EE.$$
 Every vector bundle on $\MMb'_k$ decomposes into line bundles of the form $f^*\omega^{\tensor m}$ by the Examples \ref{exa:moduli} and Proposition \ref{prop:ClassField}. By the projection formula, $f_*g'_* (f')^*\EE$ is thus a sum of vector bundles of the form $f_*\OO_{\MMb'_k} \tensor \omega^{\tensor m}$. As these are indecomposable by Proposition \ref{lem:indec}, this is also the decomposition of $g_*f'_*(f')^*\EE$ into indecomposables. By the Krull--Schmidt theorem \ref{prop:Krull-Schmidt} it is thus enough to show that $g_*\EE \tensor \omega^{\tensor n}$ is a summand of $g_*f'_*(f')^*\EE$ for some $n\in\Z$. 
 
 Note first that by Proposition \ref{lem:splitting} and Lemma \ref{lem:pushpull}, we can write $g^*f_*\OO_{\MMb_k'} \cong f'_*\OO_{\Yb}$ as a direct summand of a vector bundle of the form $\bigoplus_i g^*\omega^{\tensor n_i}$. By the Krull--Schmidt theorem, $f'_*\OO_{\Yb}$ is actually itself of this form.  The projection formula implies a chain of isomorphisms
 \begin{align*}
  g_*(f')_*(f')^*\EE &\cong g_*((f')_*\OO_{\Yb}\tensor \EE) \\
		   &\cong g_*(\bigoplus_i g^*\omega^{\tensor n_i} \tensor \EE) \\
		   &\cong \bigoplus_i (g_*\EE) \tensor \omega^{\tensor n_i},
 \end{align*}
The result follows. 
\end{proof}

\begin{remark}
 The explicit decomposition (i.e.\ the powers of $\omega$) can be worked out by comparing dimensions of spaces of modular forms, as we will do later in Section \ref{sec:dec}. 
\end{remark}

\subsection{Integral decompositions}
Let $n\geq 2$ and $R$ be again a $\Z[\frac1n]$-algebra in which $2$ or $3$ is invertible, and assume that $\Gamma \in \{\Gamma_1(n),\Gamma_0(n), \Gamma(n)\}$ is tame as in Definition \ref{def:tame}.

\begin{lemma}\label{lem:Rfreeness}
 The $R$-modules $H^i(\MMb(\Gamma)_R;g^*\omega^{\tensor m})$ are free for $i=0,1$ unless $m=1$. Moreover $H^i(\MMb(\Gamma)_R;g^*\omega)$ is free for $i=0$ and $i=1$ if and only if we can write the torsion part of $H^1(\MMb(\Gamma);g^*\omega)$ as a sum of cyclic groups of the form $\Z/k$, where $k$ is invertible in $R$ or $k=0$ in $R$. 
\end{lemma}
\begin{proof}
	 As $H^i(\MMb(\Gamma); g^*\omega^{\tensor m})$ vanishes for $i>1$ by Proposition \ref{prop:basicprops}, the K\"unneth spectral sequence Lemma \ref{lem:basechange} implies
 \begin{align}\label{eq:H1}H^1(\MMb(\Gamma)_R; g^*\omega^{\tensor m}) \cong H^1(\MMb(\Gamma); g^*\omega^{\tensor m})\tensor R\end{align}
 and a short exact sequence
 \begin{align}\label{eq:H0}0 \to H^0(\MMb(\Gamma);g^*\omega^{\tensor m}) \tensor R \to H^0(\MMb(\Gamma)_R; g^*\omega^{\tensor m}) \to \Tor_\Z(H^1(\MMb(\Gamma);g^*\omega^{\tensor m}), R) \to 0.\end{align}
 By Proposition \ref{prop:coh}, we know that $H^1(\MMb(\Gamma); g^*\omega^{\tensor m})$ is torsionfree for $m\neq 1$ and hence a free $\Z[\frac1n]$-module and clearly the same is true for $H^0(\MMb(\Gamma);g^*\omega^{\tensor m})$ for all $m$. Hence, $H^i(\MMb(\Gamma)_R; g^*\omega^{\tensor m})$ is a free $R$-module for $i=0,1$ and $m\neq 1$. 
 
If we can write the torsion part of the finitely generated $\Z[\frac1n]$-module $A := H^1(\MMb(\Gamma);g^*\omega)$ as a sum of cyclic groups of the form $\Z/k$, where $k$ is invertible in $R$ or $k=0$ in $R$, then clearly $H^i(\MMb(\Gamma)_R, g^*\omega)$ is $R$-free for $i=0,1$ by Equations \eqref{eq:H1} and \eqref{eq:H0}. On the other hand, if $A\tensor R$ is $R$-free, then we can write the torsion part of $A$ as a sum of cyclic groups of the form $\Z/k$, where $k$ is invertible in $R$ or $k=0$ in $R$. This follows from \cite[Theorem 9.3]{Kap49}; for a direct proof see \cite{Guy}.
\end{proof}

Consider the projection $g\colon \MMb(\Gamma)_R \to \MMb_{ell,R}$ as above and again the diagram
\[\xymatrix{\Yb = \MMb(\Gamma)_R \times_{\MMb_{ell,R}}\MMb_R' \ar[d]^{f'}\ar[r]^-  {g'} & \MMb'_R \ar[d]^f \\
\MMb(\Gamma)_R \ar[r]^{g} & \MMb_{ell,R},
 }
\]
where $\MMb'_R$ is still as at the beginning of this section.

\begin{thm}\label{thm:deccompact}
 Assume that $R$ is a normal noetherian local $\Z[\frac1n]$-algebra with residue field $k$ of characteristic $l\geq 0$. Then $g_*\OO_{\MMb(\Gamma)_R}$ decomposes into a sum of vector bundles of the form $f_*\OO_{\MMb_R'}\tensor \omega^{\tensor m}$ if and only
 \begin{itemize}
     \item $l=0$, or
     \item $l>0$ and $H^1(\MMb(\Gamma); g^*\omega)$ does not contain $l$-torsion, or
     \item $l>0$ and $R$ is an $\F_l$-algebra.
 \end{itemize}
\end{thm}
\begin{proof}
 If $g_*\OO_{\MMb(\Gamma)_R}$ decomposes in the prescribed manner, $H^i(\MMb(\Gamma)_R, g^*\omega)$ is free for $i=0,1$ and we can apply the lemma above. Note here that $l = 0$ in $R$ if $l^k = 0$ in $R$ because $R$ is normal and in particular reduced. 

 For the opposite direction, we obtain from Proposition \ref{prop:split} an isomorphism
 \[\overline{h}\colon g_*\OO_{\MMb(\Gamma)_k} \to f_*\FF_k\]
 for some $\FF = \bigoplus_{i\in\Z}\bigoplus_{m_i} f^* \omega^{\tensor i}$. By adjunction, this corresponds to a morphism
 \[\overline{\varphi}\colon (g')_*\OO_{\Yb_k} \cong f^*g_*\OO_{\MMb(\Gamma)_k} \to \FF_k,\]
 where we use Lemma \ref{lem:pushpull} again. By Proposition \ref{prop:basicprops}, $g$ and hence $g'$ is finite and flat and thus $(g')_*\OO_{\Yb}$ is a vector bundle. We want to show that it decomposes as a sum $\GG = \bigoplus_{i\in \Z}\bigoplus_{n_i}f^*\omega^{\tensor i}$ for some $n_i$. As $\MMb_R'$ is a weighted projective line by Examples \ref{exa:moduli}, Proposition \ref{thm:VectorBundle} implies that it is enough to show that $H^i(\MMb_R';(g')_*\OO_{\Yb}\tensor f^*\omega^{\tensor j})$ is a free $R$-module for $i=0,1$ and all $j\in\Z$. We have the following chain of isomorphisms:
 \begin{align*}
  H^i(\MMb';(g')_*\OO_{\Yb}\tensor f^*\omega^{\tensor j}) &\cong H^i(\MMb_{ell,R};f_*(g')_*(g')^*f^*\omega^{\tensor j}) \\
  &\cong H^i(\MMb_{ell,R};g_*(f')_*(g')^*f^*\omega^{\tensor j}) \\
  &\cong  H^i(\MMb_{ell,R};g_*g^*f_*f^*\omega^{\tensor j})\\
  &\cong H^i(\MMb(\Gamma)_R;g^*f_*f^*\omega^{\tensor j})\\
  &\cong H^i(\MMb(\Gamma)_R;(g^*f_*\OO_{\MMb'_R}) \tensor g^*\omega^{\tensor j})
 \end{align*}
By Proposition \ref{lem:splitting}, this $R$-module is a direct summand of a sum of terms of the form $H^i(\MMb(\Gamma)_R; g^*\omega^{\tensor n})$. By Lemma \ref{lem:Rfreeness}, these $R$-modules are free and thus $H^i(\MMb';(g')_*\OO_{\Yb}\tensor f^*\omega^{\tensor j})$ is a free $R$-module as well since $R$ is local. Thus, $(g')_*\OO_{\Yb}$ is of the form $\GG$.
 
 We claim that the morphism
 \[\Hom_{\OO_{\MMb'_R}}(\GG, \FF) \to \Hom_{\OO_{\MMb'_k}}(\GG_k, \FF_k)\]
 is surjective. This surjectivity follows from 
$$H^0(\MMb_R'; f^*\omega^{\tensor i}) \to H^0(\MMb_R'; f^*\omega^{\tensor i}) \tensor_R k \cong H^0(\MMb'_k; f^*\omega^{\tensor i}_k) $$ being surjective, which in turn is true as $H^1(\MMb_R'; f^*\omega^{\tensor i})$ is $R$-free. 
 
 Thus, we can choose a map $(g')_*\OO_{\Yb} \to \FF$, which reduces to $\overline{\varphi}$ over $k$. By tracing through the adjunctions, this corresponds to a map $h\colon g_*\OO_{\MMb(\Gamma)_R} \to f_*\FF$ whose restriction to $\MMb_{ell,k}$ agrees with $\overline{h}$. We want to show that $h$ is an isomorphism and it suffices to check this after pullback to an fpqc cover, e.g.\ to $\MMb_1(5)_R$ or $\MMb_1(6)_R$, which are isomorphic to $\mathbb{P}^1_R$ (see Example \ref{exa:5-12}). Thus, $h$ is a morphism between vector bundles on $\mathbb{P}^1_R$ whose restriction to $\mathbb{P}^1_k$ is an isomorphism. We know that $h$ is an isomorphism on an open subset of $\mathbb{P}^1_R$ that contains the special fiber; its complement is a closed subset $A$. The image of $A$ under $\mathbb{P}^1_R \to \Spec R$ is closed and thus empty as it cannot contain the closed point. Thus, $A$ is empty as well and $h$ is an isomorphism. 
\end{proof}

\begin{remark}\label{rem:Z16}
This theorem is also true if $R$ is not local but a normal noetherian $\Z[\frac16]$-algebra, where we assume instead that we can write the torsion part of $H^1(\MMb(\Gamma); g^*\omega)$ as a sum of cyclic groups of the form $\Z/a$ such that $a$ is invertible or zero in $R$. Indeed, $\MMb_{ell,R} \simeq \PP_R(4,6)$ in this case and we can directly apply Theorem \ref{thm:VectorBundle} using Lemma \ref{lem:Rfreeness}. Note further that if the conclusion of the theorem above is true for some $\Z[\frac1n]$-algebra $R$, then it is also true for every algebra over $R$ by base change. The crucial cases are quotients and localizations of the integers. 
\end{remark}

\begin{remark}\label{rem:cuspy}
Consider the short exact sequence
\begin{align}\label{eq:fundamentalexact}0 \to H^0(\MMb(\Gamma); g^*\omega^{\tensor i})/l \to  H^0(\MMb(\Gamma)_{\F_l}; g^*\omega^{\tensor i}) \to H^1(\MMb(\Gamma); g^*\omega^{\tensor i})[l] \to 0\end{align}
associated with
$$0 \to \omega^{\tensor i} \xrightarrow{l} \omega^{\tensor i} \to \omega^{\tensor i}/l\to 0.$$
We see that $H^1(\MMb(\Gamma); g^*\omega)$ is $l$-torsionfree if and only if the map 		
$$M_1(\Gamma; \Z_{(l)}) \to M_1(\Gamma; \F_l)$$
is surjective, i.e.\ if every mod-$l$ modular form of weight $1$ for $\Gamma$ can be lifted to a characteristic zero form of the same kind. If there are non-liftable modular forms of this kind, there are indeed also mod-$l$ \emph{cusp} forms of weight $1$ non-liftable to characteristic zero cusp forms (of the same level), at least if $\Gamma$ is tame. Indeed by the Semicontinuity Theorem (see \cite[Section 5]{Mum08} or \cite[Appendix A.1]{Bro12} for tame stacks), $H^1(\MMb(\Gamma); g^*\omega)$ having $l$-torsion implies that the rank of $H^1(\MMb(\Gamma)_{\F_l}; g^*\omega)$ is bigger than that of $H^1(\MMb(\Gamma)_{\C};g^*\omega)$ and these ranks agree with those of $\F_l$-valued and $\C$-valued cusp forms of weight $1$ by Corollary \ref{cor:cusps}.

As in \cite[Lemma 2]{Buz}, it is easy to show by hand that non-liftable mod-$l$ weight $1$ cusp forms do not exist for $\MMb(\Gamma)  = \MMb_1(n)$ and $n\leq 28$. Further explorations benefit from computer help and were done in \cite{Edi06}, \cite{Buz}, \cite{SchThesis} and \cite{Sch14}. Some small examples from these sources where $l$-torsion occurs in $H^1(\MMb_1(n); g^*\omega)$ are 
 $$ (l,n) = (2,1429),\quad (3,74),\quad (3,133),\quad (5, 141) \;\text{ and }\; (199,82).$$
In Appendix \ref{app:CuspForms}, we will show that the smallest (odd) $n$ with $2$-torsion in $H^1(\MMb_1(n); g^*\omega)$ is $65$. More precisely, a \texttt{SAGE}-computation will show that $S_1(\Gamma_1(n);\F_2)$ is $2$-dimensional, while a \texttt{MAGMA}-computation shows that $S_1(\Gamma_1(n);\Q) = 0$. 

 There is evidence \cite{Sch14} that the torsion in $H^1(\MMb_1(n);g^*\omega)$ grows at least exponentially in $n$. 
\end{remark}

\section{Computing the decompositions}\label{sec:dec}
In this section, we will be more concrete and actually give formulas how to decompose vector bundles. As we will mostly work over a field of characteristic not $2$ or $3$, the hard work of the last section is almost entirely unnecessary for this section. But the results of this section have strong implications for the integral decompositions from the last section. We will start with some recollections and results about dimensions of spaces of modular forms. 

\subsection{Dimensions of spaces of modular forms}
Let $K$ be a field of characteristic not dividing $n$ and $\Gamma$ be one of the congruence subgroups $\Gamma_0(n)$, $\Gamma_1(n)$ or $\Gamma(n)$. Denote by $m_i^K$ the dimension of the space of weight $i$ modular forms for $\Gamma$ and by $s_i^K$ the dimension of the space of cusp forms for the same weight and group. Note $m_i^K = m_i^L$ and $s_i^K = s_i^L$ if $K\subset L$ is a field extension. 
\begin{lemma}\label{lem:inequality}
 We have the inequalities $m_i^K\geq m_i^{\C}$ (with equality for $i\geq 2$ if $\Gamma$ is tame).
\end{lemma}
\begin{proof}
As $m_i^K = m_i^{\C}$ if $K$ has characteristic zero, we can assume that $K$ has characteristic $l>0$ and actually that $K = \F_l$. Set 
$$M = M_i(\Gamma; \Z[\frac1n]) = H^0(\MMb(\Gamma); g^*\omega^{\tensor i}).$$ 
By definition, $m_i^{\C}$ equals $\dim_{\C} M \tensor \C$. Note further that $M$ is torsionfree so that $\dim_{\C} M \tensor \C = \dim_{\F_l} M/l$. 
The exact sequence \eqref{eq:fundamentalexact} reads in our notation as
$$0 \to M/l \to  M_i(\Gamma; \F_l) \to H^1(\MMb(\Gamma); g^*\omega^{\tensor i})[l] \to 0.$$

Thus, 
$$m_i^{\F_l} \geq \dim_{\F_l} M/l = m_i^{\C}$$
with equality for $i \geq 2$ and $\Gamma$ tame as then $H^1(\MMb(\Gamma); g^*\omega^{\tensor i}) = 0$ by Proposition \ref{prop:coh}. 
\end{proof}
From now on we will just write $m_i = m_i^K$ and $s_i = s_i^K$. 

We will need the following standard fact, which follows directly from Riemann--Roch, Proposition \ref{prop:coh} and Corollary \ref{cor:cusps}. 

\begin{prop}\label{prop:riemann}
Assume that $\MMb(\Gamma)_K$ is representable. Let $g(\Gamma) = s_2$ be the genus of $\MMb(\Gamma)_K$. Then 
$$\deg(g^*\omega)i + 1-g(\Gamma) = \begin{cases}m_1-s_1 & \text{ if }i =1\\
m_i & \text{ if } i\geq 2.
\end{cases}.$$
\end{prop}

Let us be more explicit about $\deg(f_n^*\omega)$ for $\Gamma = \Gamma_1(n)$. 

\begin{lemma}\label{lem:degomega}
The degree $\deg (f_n)^*\omega$ equals $\frac1{24} d_n$ for $d_n$ the degree of the map $f_n\colon \MM_1(n)_K \to \MM_{ell, K}$. We have 
\begin{align*}
 d_n &= \sum_{d|n}d\phi(d)\phi(n/d) \\
 &= n^2\prod_{p|n}(1-\frac1{p^2})
\end{align*}
\end{lemma}
\begin{proof}
By \cite[VI.4.4]{D-R73}, the line bundle $\omega$ has degree $\frac1{24}$. As  $\deg (f_n)^*\omega = \deg(f_n) \deg \omega$, the first result follows. For the formulas for $d_n$ see \cite[Sec 3.8+3.9]{D-S05}; note that the map of stacks has twice the degree of the map of coarse moduli spaces as the generic point of $\MM_{ell}$ has automorphism group of order $2$ and that the formulas from \cite{D-S05} are actually valid in all characteristics as $\MM_1(n) \to \MM_{ell,\Z[\frac1n]}$ is finite and flat. 
\end{proof}

It is not quite obvious from Proposition \ref{prop:riemann} how to obtain even a good lower bound on $m_1$. We have the following simple observation.
\begin{lemma}\label{lem:m1bound}
 Let $\Gamma$ be $\Gamma_1(n)$ for $n\geq 5$ or $\Gamma(n)$ for $n\geq 3$. Then $m_1 \geq 2$. 
\end{lemma}
\begin{proof}
 By Theorem 3.6.1 from \cite{D-S05} the quantity $2m_1^{\C}$ is at least the number of regular cusps for $\Gamma$ and p.\ 103 of loc.\ cit.\ shows that in our cases all cusps are regular. By \cite[Figure 3.3]{D-S05} the number of cusps is at least $4$ for our choices of $\Gamma$ and thus $m_1^{\C}\geq 2$. Lemma \ref{lem:inequality} implies our result. 
\end{proof}

We need the following (probably standard) lemma, which was proven jointly with Viktoriya Ozornova. 

\begin{lemma}
 Let $K$ be an algebraically closed field and $A$ a $\Z_{\geq 0}$-graded integral domain over $K$. Set $m_i = \dim_K A_i$ and assume that $m_0 = 1$. Then $m_2 \geq 2m_1-1$.\footnote{We allow ourselves the abuse of the notation $m_i$ as we will take $A$ to be a ring of modular forms later.}
\end{lemma}
\begin{proof}
 We will work throughout this proof over the field $K$ and set $n = m_1$. Let $P =K[x_1,\dots, x_n]$ with the $x_i$ of degree $1$. Without loss of generality we can assume that $A$ is generated in degree $1$ and can thus be written as $P/I$ with $I$ a homogeneous ideal generated in degrees $\geq 2$. An element in $P_2$ can be written as $\sum_{i\leq j} a_{ij}x_ix_j$ and thus we can view it as an upper triangular matrix $(a_{ij})$. Let $R$ be the space of upper triangular matrices associated with the elements of $P_2 \cap I$. For an arbitrary matrix $(a_{ij})$ we set 
 $U((a_{ij})) = (m_{ij})$ with 
 $$m_{ij} = \begin{cases}
  a_{ii} &\text{ if }i=j \\
  a_{ij}+a_{ji} &\text { if }i<j\\
  0 & \text{ else.}
 \end{cases}$$
 Thus, $U$ defines a linear map $\Mat_{n\times n} \to \UpT$ from all $n\times n$-matrices to the upper triangular ones. If we view linear homogeneous polynomials in the $x_i$ as column vectors $\mathbf{v},\mathbf{w}$ in $P_1 = A_1$, then their product corresponds to the upper triangular matrix $U(\mathbf{v}\mathbf{w}^T)$. As $A$ is an integral domain,
 no nonzero element in $R$ is of this form. 
 
  We can reformulate this in terms of the Segre embedding. Recall that this is the map $\iota\colon \mathbb{P}^{n-1}\times \mathbb{P}^{n-1} \to \mathbb{P}(\Mat_{n\times n})$ sending $([\mathbf{v}],[\mathbf{w}])$ to $[\mathbf{v}\mathbf{w}^T]$. Let $V = \mathbb{P}(\Mat_{n\times n}) \setminus \mathbb{P}(\ker(U))$. As $P$ is an integral domain, $\iota$ factors through $V$. Furthermore, $U$ defines an algebraic map $V \to \mathbb{P}(\UpT)$. The composite map $\kappa\colon \mathbb{P}^{n-1}\times \mathbb{P}^{n-1} \to \mathbb{P}(\UpT)$ is proper as the source is proper over $K$. Furthermore, $\kappa$ is quasifinite because $P$ is a unique factorization domain and thus every point in $\mathbb{P}(\UpT)$ has at most two preimages in $\mathbb{P}^{n-1}\times \mathbb{P}^{n-1}$. Thus, $\kappa$ defines a finite map $ \mathbb{P}^{n-1}\times \mathbb{P}^{n-1} \to \im(\kappa)$ and thus $\im(\kappa)$ is a projective variety of dimension $2n-2$. As $\im(\kappa)\cap \mathbb{P}(R) = \varnothing$, it follows by \cite[Thm I.7.2]{Har77} 
that $R$ has dimension at most $\dim_K \UpT - (2n-1)$ and thus that 
  \[m_2 = \dim_K \UpT - \dim_K R \geq 2n-1. \qedhere\] 
\end{proof}

Together with Proposition \ref{prop:irreducible}, this directly implies the following proposition.
\begin{prop}\label{prop:m2m1}
 We always have $m_2 \geq 2m_1 -1$. \hfill $\square$
\end{prop}

\subsection{Decompositions into powers of $\omega$}\label{subsec:dec}
In this section, we will work (implicitly) over a field $K$ of characteristic not $2$ or $3$ until further notice. Let $\Gamma$ be again one of the congruence subgroups $\Gamma_0(n)$, $\Gamma_1(n)$ or $\Gamma(n)$. Let $g\colon \MMb(\Gamma) \to \MMb_{ell}$ be the projection.  Recall that $\MMb_{ell} \simeq \PP(4,6)$ with $\OO(1) \cong \omega$. By Proposition \ref{prop:basicprops}, the $\OO_{\MMb_{ell}}$-module $g_*\OO_{\MMb(\Gamma)}$ is locally free of finite rank. Thus it decomposes by Proposition \ref{prop:ClassField} as 
 \begin{align}\label{eq:dec} g_*\OO_{\MMb(\Gamma)} \cong \bigoplus_{i\in\Z} \bigoplus_{l_i} \omega^{\tensor (-i)}.\end{align}
 
 Our aim is to determine the sequence of $l_i$ (which is well-defined by the Krull--Schmidt Theorem \ref{prop:Krull-Schmidt}). We will sometimes call it the \emph{decomposition sequence} of $g_*\OO_{\MMb(\Gamma)}$. 
 
 \begin{prop}\label{prop:dec5}We have 
  \begin{enumerate}
   \item $l_i = 0$ for $i<0$ and $i> 11$,
   \item $l_i = m_i-m_{i-4}-m_{i-6}+m_{i-10}$ for $i\leq 11$; in particular, $l_i = m_i$ for $i\leq 3$,
   \item $l_{12-i} = s_i$ for $i\leq 4$,
   \item $l_{10}$ is the genus of $\MMb(\Gamma)$, i.e.\ $\dim_K H^0(\MMb(\Gamma); \Omega^1_{\MMb(\Gamma)/K})$. 
  \end{enumerate}
 \end{prop}
 \begin{proof}
  The number $m_k$ is by definition the dimension of 
  \[H^0(\MMb(\Gamma); g^*\omega^{\tensor k}) \cong H^0(\MMb_{ell}; g_*\OO_{\MMb(\Gamma)}\tensor \omega^{\tensor k}).\]
  Denote by $d_i = \dim_K H^0(\MMb_{ell};\omega^{\tensor i})$ the dimension of the space of holomorphic modular forms of weight $i$ for $SL_2\Z$. Then \eqref{eq:dec} implies that
  \begin{align}\label{sum}m_k = \sum_{i\in \Z} l_id_{k-i}.\end{align}
  In particular, $l_i = 0$ for $i<0$ because $m_i \geq l_i = l_id_{i-i}$ and there are no modular forms of negative weight (Proposition \ref{prop:coh}). 
  
  To get more precise results, we want to use Serre duality on $\MMb_{ell}$ and $\MMb(\Gamma)$. By Theorem \ref{thm:fundamentalweighted}, the stack $\MMb_{ell} \simeq \PP_K(4,6)$ has dualizing sheaf $\omega^{\tensor (-10)}$. Using this and that $g$ is affine and thus $g_*$ is exact, we get the following chain of isomorphisms: 
  
  \begin{align*}
   H^0(\MMb_{ell}; (g_*\OO_{\MMb(\Gamma)})^\vee \tensor \omega^{\tensor -10-k}) &\cong H^0(\MMb_{ell}; (g_*\OO_{\MMb(\Gamma)}\tensor \omega^{\tensor k})^\vee \tensor \omega^{\tensor -10}) \\
   &\cong H^1(\MMb_{ell}; g_*\OO_{\MMb(\Gamma)}\tensor \omega^{\tensor k})^{\vee} \\
   &\cong H^1(\MMb_{ell}; g_*(\OO_{\MMb(\Gamma)}\tensor g^*\omega^{\tensor k}))^{\vee}\\
   &\cong H^1(\MMb(\Gamma); g^*\omega^{\tensor k})^{\vee} \\
   &\cong H^0(\MMb(\Gamma); \Omega^1_{\MMb(\Gamma)/K} \tensor g^*\omega^{\tensor -k})
  \end{align*}
 By Proposition \ref{prop:coh}, this vanishes for $k\geq 2$.
  
  Since the rank of $g_*\OO_{\MMb(\Gamma)}$ is finite, only finitely many $l_i$ can be nonzero. Let $j$ be the largest index such that $l_j \neq 0$. Then $H^0(\MMb_{ell};(g_*\OO_{\MMb(\Gamma)})^\vee \tensor \omega^{\tensor -j})$ has dimension $l_j$. In particular, we see by the computation above that $j\leq 11$, proving part (1) of our proposition. Using that the ring of holomorphic modular forms for $SL_2\Z$ is isomorphic to $K[c_4,c_6]$ and thus $d_0 = d_4 = d_6 = d_8 = d_{10} = 1$ and $d_i = 0$ for all other $i\leq 11$, we obtain the recursive equation
  $$l_i = m_i-l_{i-4}-l_{i-6}- l_{i-8} - l_{i-10}$$
  from Equation (\ref{sum}). Part (2) of our proposition follows by a straightforward computation.
  
  The equality $\dim_K H^0(\MMb_{ell};(g_*\OO_{\MMb(\Gamma)})^\vee \tensor \omega^{\tensor -i}) = l_i$ holds even for all $i\geq j-3$ (and in particular for $i\geq 8$) as $d_k = 0$ for $0<k<4$. Part (3) follows then from Corollary \ref{cor:cusps}. Part (4) follows from the previous computations and the definition of the genus. 
 \end{proof}
 
 \begin{example}\label{exa:m12}
  Let us consider the case $n=2$. By Example \ref{exa:moduli}, the ring of modular forms for $\Gamma_1(2)$ is isomorphic to $K[b_2, b_4]$. Furthermore, we know that the rank of $(f_2)_*\OO_{\MMb_1(2)}$ is $3$. Thus, there can be at most three nonzero $l_i$ and these are $l_0 =l_2 = l_4 = 1$. In other words: $(f_2)_*\OO_{\MMb_1(2)} \cong \OO_{\MMb_{ell}} \oplus \omega^{\tensor -2} \oplus \omega^{\tensor -4}$.
 \end{example}
 
 \begin{example}\label{exa:m13}
  Now consider the case $n=3$. By Example \ref{exa:moduli}, the ring of modular forms for $\Gamma_1(3)$ is isomorphic to $K[a_1,a_3]$. By the last proposition, it follows easily that in this case
  $$l_0 =1,\; l_1=1,\; l_2=1,\; l_3=2,\; l_4=1,\; l_5 =1,\; l_6 = 1$$
  and all the other $l_i$ are zero. 
 \end{example}

\subsection{Refined decompositions}
We will use the notation of the last subsection and start with the following corollary to Proposition \ref{prop:dec5}.
 \begin{cor}\label{cor:dec2}
 Let $\Gamma$ be $\Gamma_1(n)$ for $n\geq 5$ or $\Gamma(n)$ for $n\geq 3$. We have a decomposition
 \[g_*\OO_{\MMb(\Gamma)} \cong \bigoplus_{i\in\Z} \bigoplus_{k_i} (f_3)_*\OO_{\MMb_1(3)} \tensor \omega^{\tensor -i}.\]
 The $k_i$ are uniquely determined and satisfy
 \begin{enumerate}
  \item $k_i = 0$ for $i<0$ and $i>5$,
  \item $k_i = m_i-m_{i-1}-m_{i-3}+m_{i-4}$,
  \item $k_5 = s_1$ and $k_4 =s_2-s_1$,
  \item $k_0 + k_3= k_1 +k_4 = k_2+k_5$.
 \end{enumerate}
\end{cor}
\begin{proof}
First we want to show the existence of a decomposition of the form 
\begin{align}\label{eq:1(3)}g_*\OO_{\MMb(\Gamma)} \cong \bigoplus_{i\in\Z} \bigoplus_{k_i} (f_3)_*\OO_{\MMb_1(3)} \tensor \omega^{\tensor -i}.
\end{align}
To that purpose set $k_i = m_i-m_{i-1}-m_{i-3}+m_{i-4}$. We start by showing that $k_i\geq 0$. The two essential properties of the sequence $m_i$ we need is that there are constants $a>0$ and $b$ such that $m_i = ai +b$ for all $i\geq 2$ by Proposition \ref{prop:riemann} and that $m_1 \neq 0$ by Lemma \ref{lem:m1bound}. 

We see directly that $k_i = 0$ for $i>5$ or for $i<0$. As the ring of modular forms has no zero divisors, we also see that $m_i \geq m_{i-1}$ and likewise $s_i\geq s_{i-1}$ by multiplying with a nonzero modular form of weight $1$. This implies $k_i \geq 0$ for $i\leq 2$. We have 
$$k_3 = m_3-m_2-m_0 = a -1 \geq 0$$
using that $a \in\Z$. 
From Proposition \ref{prop:riemann}, we can compute $k_5 = s_1$ and $k_4 = s_2-s_1$, which are also clearly at least $0$. By Example \ref{exa:m13}, we just have to check now that 
$$l_i = k_i+k_{i-1}+k_{i-2}+2k_{i-3}+k_{i-4}+k_{i-5}+k_{i-6},$$
with the $l_i$ as in the last section, which is a straightforward computation from Proposition \ref{prop:dec5}. Thus, we obtain the existence of a decomposition into summands of the form $(f_3)_*\OO_{\MMb_1(3)} \tensor \omega^{\tensor -i}$.  

Next we show that the formula from item 2 has to hold for every decomposition as in Equation \eqref{eq:1(3)}. This follows by a straightforward computation from $k_i = 0$ for $i>5$ or for $i<0$ (as follows from Proposition \ref{prop:dec5} and Example \ref{exa:m12}) and from the equation
\[k_i = m_i-k_{i-1}-k_{i-2} -2k_{i-3}-2k_{i-4}-2k_{i-5},\]
which we obtain from the dimension of the space of modular forms for $\Gamma_1(3)$ being $1,1,1,2,2$ and $2$ in weights $0$ to $5$, respectively. 
 Item 2 implies the following equations:
 \begin{align*}
  k_0+k_3&= m_3-m_2 \\
  k_1 +k_4 &= m_4-m_3 \\
  k_2+k_5 &= m_5-m_4
 \end{align*}
As the function $i\mapsto m_i$ is affine linear in $i$ for $i\geq 2$, this implies item 4. 
\end{proof}

Using item (4) of the last corollary, we can actually prove a stronger result.

\begin{cor}\label{cor:decwild}
 Let $\Gamma$ be $\Gamma_1(n)$ for $n\geq 5$ or $\Gamma(n)$ for $n\geq 3$.
 We have a decomposition
 \[g_*\OO_{\MMb(\Gamma)} \cong \bigoplus_{i\in\Z} \bigoplus_{\kappa_i} (f_q)_*\OO_{\MMb_1(q)} \tensor \omega^{\tensor -i},\]
 for $q=4$, $5$ and $6$, where we assume $\Char(K) \neq 5$ if $q=5$. In the case of $q=5$ or $6$, we have the formulae
 $$\kappa_0 = 1,\, \kappa_1 = m_1- 2,\, \kappa_2 = m_2 -2m_1+1,\, \kappa_3 = s_1,\, \kappa_i = 0 \text{ for }i\geq 4.$$
\end{cor}
\begin{proof}
 From the formulae in Proposition \ref{prop:dec5}, it is straightforward to see that $(f_5)_*\OO_{\MMb_1(5)} \cong (f_6)_*\OO_{\MMb_1(6)}$ and that this decomposes as $$(f_4)_*\OO_{\MMb_1(4)}  \oplus (f_4)_*\OO_{\MMb_1(4)} \tensor \omega^{\tensor (-1)}.$$
 Thus it suffices to consider the case $q=5$. By the last corollary, we have a decomposition of the form
 \[g_*\OO_{\MMb(\Gamma)} \cong \bigoplus_{i\in\Z} \bigoplus_{k_i} (f_3)_*\OO_{\MMb_1(3)} \tensor \omega^{\tensor -i}.\]
 For example, if $\Gamma=\Gamma_1(5)$, we get $k_0 = k_1=k_2= 1$ and all other $k_i = 0$. In the general case, we set $\kappa_0 = k_0 = 1$, $\kappa_1 = k_1-k_0$, $\kappa_2 = k_2-k_1$ and $\kappa_3 = k_5$. We have to check that $k_i = \kappa_i + \kappa_{i-1} +\kappa_{i-2}$, which is a straightforward computation from item 4 of Corollary \ref{cor:dec2}. The inequalities $\kappa_i \geq 0$ are clear for $i= 0$ and $i=3$. For $i=1$ and $i=2$, it translates to $m_1 \geq 2$ and $m_2 \geq 2m_1-1$. These are exactly the statements of Lemma \ref{lem:m1bound} and Proposition \ref{prop:m2m1}.  
\end{proof}

The vector bundle $(f_4)_*\OO_{\MMb_1(4)}$ decomposes as 
$$(f_2)_*\OO_{\MMb_1(2)} \tensor (\OO_{\MMb_{ell}} \oplus \omega^{\tensor (-1)} \oplus \omega^{\tensor (-2)}).$$
Maybe the easiest way to see this is by noting that for $h\colon \MMb_1(4) \to \MMb_1(2)$ the projection, Examples \ref{exa:moduli} and Proposition \ref{prop:ClassField} imply that $h_*\OO_{\MMb_1(4)}$ splits into a sum of line bundles of the form $(f_2)^*\omega^{\tensor m}$. A simple dimension count for spaces of modular forms implies that $(h_*)\OO_{\MMb_1(4)} \cong \OO_{\MMb_1(2)} \oplus (f_2)^*\omega^{\tensor (-1)} \oplus (f_2)^*\omega^{\tensor (-2)}$. Thus, the last corollary implies:
 \begin{cor}\label{cor:dec3}
 Let $\Gamma$ be $\Gamma_1(n)$ for $n\geq 4$ or $\Gamma(n)$ for $n\geq 3$. Then we have a decomposition
 \[g_*\OO_{\MMb(\Gamma)} \cong \bigoplus_{i\in\Z} \bigoplus_{k_i} (f_2)_*\OO_{\MMb_1(2)} \tensor \omega^{\tensor -i}.\]
 The $k_i$ are uniquely determined and satisfy
 \begin{enumerate}
  \item $k_i = 0$ for $i<0$ and $i>7$,
  \item $k_i = m_i-m_{i-2}-m_{i-4}+m_{i-6}$; in particular, $k_i = m_i$ for $i\leq 1$,
  \item $k_7 = s_1$ and $k_6=s_2$ is the genus of $\MMb(\Gamma)$,
 \end{enumerate}
\end{cor}
\begin{proof}
The only things not straightforward to deduce from the last corollary are the uniqueness of the $k_i$ and $k_6 = s_2$. The uniqueness follows as the existence of the decomposition directly implies $k_i = 0$ for $i<0$ and from the equation
\[k_i = m_i-k_{i-2}-2k_{i-4} -2k_{i-6},\]
which we obtain from the dimension of the space of modular forms for $\Gamma_1(2)$ being $1,1,2$ and $2$ in weights $0$, $2$, $4$ and $6$ respectively and zero else in this range. Moreover, we see $k_6 = l_{11}$ by Example \ref{exa:m12} and $l_{11} = s_2$ by Proposition \ref{prop:dec5}. 
\end{proof}

We have obtained results for decompositions into $(f_q)_*\OO_{\MMb_1(q)}$ for all $q\leq 6$. For larger $q$ such decomposition results are impossible in general. 

\begin{cor}\label{cor:dectoowild}
 For every $q>6$, there is an arbitrarily large $n$ such that $(f_n)_*\OO_{\MMb(\Gamma_1(n))}$ does not decompose as
 $$\bigoplus_{i\in\Z} \bigoplus_{\kappa_i} (f_q)_*\OO_{\MMb_1(q)} \tensor \omega^{\tensor -i}.$$
\end{cor}
\begin{proof}
 Let $d_m$ be the degree of the map $\MM_1(m) \to \MM_{ell}$, which is also the rank of $(f_m)_*\OO_{\MMb_1(m)}$. By Lemma \ref{lem:degomega}, the function $d_m = d(\Gamma_1(m))$ is multiplicative and for primes $p$, we have $d_{p^k} = p^{2k-2}(p^2-1$); for example $d_5 = d_6 = 24$, $d_7 = d_8 = 48$, $d_9=72$ and $d_{12} = 96$. Moreover, $d_p > 24$ for primes $p>5$. These facts imply that $d_q>24$ for every $q>6$. 
 
 For a decomposition as in the statement of the corollary, it is necessary that $d_q$ divides $d_n$. Thus, we only need to show that for every $D>24$, there are infinitely many $p$ with $d_p$ not 
divisible by $D$. Every $D>24$ has a divisor of the form $d= 16$, $d=9$ or $d$ a prime that is at least $5$. Pick an $a$ that is coprime to $d$ and not congruent to $\pm 1 \mod d$; for $d=16$ we take $a=3$ and for $d=9$ we take $a=2$. By Dirichlet's prime number theorem, there are infinitely many primes $p$ such that $p\equiv a \mod d$. If $d$ is prime, this implies that $d$ does not divide $d_p = (p-1)(p+1)$. If $d =16$, this implies that $d_p \equiv 8 \mod d$, and for $d=9$, this implies $d_p \equiv 3\mod d$. In any case, $d$ does not divide $d_p$ for infinitely many primes $p$ and thus $D$ does not as well. 
\end{proof}

\begin{remark}
The only obstruction presented in the last proof for decomposing $(f_n)_*\OO_{\MMb_1(n)}$ into copies of $(f_m)_*\OO_{\MMb_1(m)}\tensor \omega^{\tensor k}$ was that $d_m|d_n$. But in general it is not true that $d_m|d_n$ implies the possibility of such a decomposition. For example $d_7|d_{31}$, but $(f_{31})_*\OO_{\MMb_1(n)}$ does not decompose into copies of $(f_7)_*\OO_{\MMb_1(7)}\tensor \omega^{\tensor k}$.
\end{remark}

So far, we only worked over a field. Now we demonstrate the implication for our integral decomposition results. 

\begin{thm}\label{thm:l}
 Let $l$ be a prime not dividing $n$ and $\Gamma$ either $\Gamma_1(n)$ for $n\geq 5$ or $\Gamma(n)$ for $n\geq 3$. Assume that $H^1(\MMb(\Gamma); g^*\omega)$ contains no $l$-torsion. Then for every $q$ of the form 
  \begin{itemize}
 	\item $1 \leq q \leq 6$ if $l>3$,
 	\item $q = 2, 4, 5$ if $l = 3$, and
 	\item $q = 3, 5$ if $l = 2$. 
 \end{itemize}
the vector bundle $g_*\OO_{\MMb(\Gamma)_{(l)}}$ decomposes into summands of the form $(f_q)_*\OO_{\MMb_1(q)_{(l)}}$. 

\end{thm}
\begin{proof}
 We do the proof in the case $l=3$, the other cases being similar. By Theorem \ref{thm:deccompact}, $g_*\OO_{\MMb(\Gamma)_{(l)}}$ decomposes as 
 \[\bigoplus_{i\in\Z} \bigoplus_{k_i} (f_2)_*\OO_{\MMb_1(2)} \tensor \omega^{\tensor -i}.\]
 By Corollary \ref{cor:decwild}, the result holds for $q = 4,5$ after base change to $\Q$. As the $k_i$ are uniquely determined as noted in Corollary \ref{cor:dec3}, this implies that $g_*\OO_{\MMb(\Gamma)_{(l)}}$ decomposes indeed into summands of the form 
 $$(f_2)_*\OO_{\MMb_1(2)} \tensor (\OO_{\MMb_{ell}} \oplus \omega^{\tensor (-1)} \oplus \omega^{\tensor (-2)}) \tensor \omega^{\tensor i} \cong (f_4)_*\OO_{\MMb_1(4)}\tensor \omega^{\tensor i}$$
 and similarly for $q = 5$. 
\end{proof}

Recall from Example \ref{exa:5-12} that $\MMb_1(5) \cong \mathbb{P}^1_{\Z[\frac15]}$. Thus, $\omega$ becomes trivial on $\MM_1(5) \subset \A^1_{\Z[\frac15]}$. This implies for $l$ and $\Gamma$ as in the last theorem that $g_*\OO_{\MM(\Gamma)}$ is isomorphic to $(f_5)_*\OO_{\MM_1(5)}^{\oplus \frac{d}{24}}$ for $d$ being the degree of $g\colon \MM(\Gamma) \to \MM_{ell}$. 

\section{Consequences for rings of modular forms}\label{sec:modforms}
The aim of this section is to apply the splittings of vector bundles proved in the last sections to rings of modular forms and thus complete our proofs of the claims in the introduction. As before, we will write sometimes $\MMb(\Gamma_1(n))$ for $\MMb_1(n)$ etc. and denote the map $\MMb(\Gamma)_R \to \MMb_{ell,R}$ by $g$ for any $\Z[\frac1n]$-algebra $R$. Let $\Gamma = \Gamma_1(n),\Gamma(n)$ or $\Gamma_0(n)$. Recall that we define the space of weight-$k$-modular forms for $\Gamma$ by 
$$M_k(\Gamma;R) = H^0(\MMb(\Gamma)_R; g^*\omega^{\tensor k}).$$
Likewise, we define the space of weight-$k$ weakly holomorphic modular forms for $\Gamma$ by 
$$\widetilde{M}_k(\Gamma;R) = H^0(\MM(\Gamma)_R; g^*\omega^{\tensor k}).$$
As before, we will write just $M^R_k$ and $\widetilde{M}^R_k$ if $\Gamma = SL_2(\Z)$. For example, we have
$$M_*^\Z \cong \Z[c_4,c_6,\Delta]/(c_4^3-c_6^2 = 1728\Delta)$$
as shown in \cite{Del75}.

\subsection{Results over a field}
We will need Proposition 6.2 from \cite{Del75}:
\begin{prop}\label{prop:dapres}
 We have $M^{\F_2}_* = \F_2[a_1,\Delta]$ and $M^{\F_3}_* = \F_3[b_2,\Delta]$. The classes $a_1$ and $b_2$ have the usual definitions in terms of the Weierstra{\ss} forms and can be identified (up to sign) with the respective Hasse invariants. 
\end{prop}
Note in particular that $M^{\F_2}_*$ is a finitely generated $M^\Z_*$-module with generating system $\{1,a_1,a_1^2, a_1^3\}$ as $c_4$ goes to $a_1^4$. Likewise $M^{\F_3}_*$ is a finitely generated $M^\Z_*$-module with generating system $\{1,b_2\}$ as $c_4$ goes to $b_2^2$. Note further that $M^{\F_p}_* \cong M^\Z_*/p \cong \F_p[c_4,c_6]$ for $p\geq 5$ as $H^1(\MMb_{ell}, \omega^{\tensor *})_{(p)} = 0$ for $p\geq 5$.

\begin{lemma}\label{lem:freeness}
 The ring $M_*(\Gamma_1(3);\F_2) = \F_2[a_1,a_3]$ is free of rank $4$ as a graded $M^{\F_2}_*$-module. The ring $M_*(\Gamma_1(2);\F_3) = \F_3[b_2,b_4]$ is free of rank $3$ as a graded $M^{\F_3}_*$-module.
\end{lemma}
\begin{proof}
 Consider first the case of $M_*(\Gamma_1(3);\F_2)$. The map $M^{\F_2}_* \to M_*(\Gamma_1(3);\F_2)$ sends $a_1$ to $a_1$ and $\Delta$ to $a_3^4+a_1^3a_3^3$. Thus, we obtain $\{1,a_3,a_3^2,a_3^3\}$ as a basis. The map $M^{\F_3}_* \to M_*(\Gamma_1(2);\F_3)$ sends $b_2$ to $b_2$ and $\Delta$ to $b_2^2b_4^2-b_4^3$. We obtain $\{1,b_4,b_4^2\}$ as a basis.
\end{proof}

We will need the following linear algebra lemma to transfer later our results from $\Gamma_1(n)$ to $\Gamma_0(n)$. 
\begin{lemma}\label{lem:kernel}
 Let $K$ be a field and $R$ a graded $K$-algebra that is connected (i.e.\ $R_0 = K$ and $R_i = 0$ for $i<0$). Let $F$ be a graded free module of finite rank with basis elements in degrees $\leq k$. Let $f\colon F \to F$ be an $R$-linear grading-preserving endomorphism. Then $\ker(f)$ is a graded free $R$-module with basis elements in degrees $\leq k$ again. 
\end{lemma}
\begin{proof}
 Order the basis element $b_i$ ascendingly by degree and represent $f$ by a matrix $M = (m_{ij})$. Note that $|b_i| = |b_j|$ implies $m_{ij} \in R_0=K$ and $|b_i| < |b_j|$ implies $m_{ij} = 0$. Thus we can use elementary row transformations to transform $M$ into an upper triangular matrix $M'$ whose diagonal entries still lie in $R_0 = K$. The standard procedure to compute the kernel of an upper triangular matrix implies the result. 
\end{proof}

From now on we will fix an integer $n\geq 2$ and a group $\Gamma$, which is $\Gamma_1(n)$, $\Gamma(n)$ or $\Gamma_0(n)$. 
\begin{thm}\label{thm:decfield}
 Let $K$ be a field of characteristic $l\geq 0$ not dividing $n$. Then $M_*(\Gamma; K)$ is free as a graded $M_*^K$-module of rank $[SL_2(\Z):\Gamma]$ if $l \neq 2$ and of rank $\frac12[SL_2(\Z):\Gamma]$ if $l=2$. The basis elements are in degrees at most $B_l$ with 
 $$B_l = \begin{cases}
  14 & \text{ if } l = 2 \\
  15 & \text{ if } l = 3\\
  11 & \text{ if } l>3
 \end{cases}$$
\end{thm}
\begin{proof} 
 Assume first that $\Gamma$ is tame. By Proposition \ref{prop:split}, the vector bundle $g_*\OO_{\MMb(\Gamma)_K}$ splits as $\bigoplus_i \omega^{\tensor k_i}$ if $l\neq 2,3$. Thus, 
 \begin{align*}
  M_*(\Gamma;K) &= H^0(\MMb(\Gamma)_K; g^*\omega^{\tensor *}) \\
 &=H^0(\MMb_{ell,K}; g_*g^*\omega^{\tensor *}) \\
 &=H^0(\MMb_{ell,K}; \omega^{\tensor *}\tensor_{\OO_{\MMb_{ell,K}}} g_*\OO_{\MMb(\Gamma)_K})\\
 &= \bigoplus_i H^0(\MMb_{ell,K}; \omega^{\tensor k_i+*}) \\
 &= \bigoplus_i M_{k_i+*}^K.
 \end{align*}
 In the case $l=3$, we obtain similarly a splitting of the form $M_*(\Gamma;K) \cong \bigoplus_i M_{k_i+*}(\Gamma_1(2);K)$ and for $l=2$ a splitting of the form $M_*(\Gamma;K) \cong \bigoplus_i M_{k_i+*}(\Gamma_1(3);K)$. The $M_*^K$-module $M_*(\Gamma; K)$ is free of the claimed rank by combining Lemma \ref{lem:freeness} with the following three observations:
 \begin{itemize}
  \item The canonical map $M_*(\Gamma;\F_l)\tensor_{\F_l}K \to M_*(\Gamma;K)$ is an isomorphism (as $K$ is flat over $\F_l$), 
  \item $\deg(g) = [SL_2(\Z):\Gamma]$,
  \item $[SL_2(\Z):\Gamma_1(2)] = 3$ and $[SL_2(\Z):\Gamma_1(3)] = 8$. 
 \end{itemize}
 For the estimate of the degrees of the basis elements, we will just treat the case $l=2$, the others being similar. By Proposition \ref{prop:coh}, $H^1(\MMb_{ell,K}; g_*\OO_{\MMb(\Gamma)}\tensor \omega^{\tensor i}) = 0$ for $i>1$. On the other hand, 
 $$H^1(\MMb_{ell,K}, (f_3)_*\OO_{\MMb_1(3)_K}\tensor \omega^{\tensor (-4)}) \cong H^1(\PP_K(1,3); \OO(-4)) \cong K$$
 by Example \ref{exa:moduli} and Theorem \ref{thm:fundamentalweighted}. Thus, $k_i\geq -5$ in the decomposition $g_*\OO_{\MMb(\Gamma)_K} \cong \bigoplus (f_3)_*\OO_{\MMb_1(3)_K} \tensor \omega^{\tensor k_i}$. 
 As the elements in an $(M_*^K)$-basis of $M_*(\Gamma_1(3);K)$ have degrees at most $|a_3^3| = 9$, the degree of the basis elements for $M_*(\Gamma; K)$ is bounded by $9+5 =14$. 
 
It remains to treat the case $\Gamma = \Gamma_0(n)$ if it is not tame. We have $M_*(\Gamma_0(n); K) \cong M_*(\Gamma_1(n); K)^{(\Z/n)^\times}$ by Proposition \ref{prop:square}. If we decompose $(\Z/n)^\times$ as a sum of cyclic groups $Z_1 \oplus \cdots \oplus Z_m$, we can inductively apply Lemma \ref{lem:kernel} to conclude that $M_*(\Gamma_1(n); K)^{Z_1\oplus \cdots \oplus Z_j}$ is free with generators in degrees at most $B_l$ for all $1\leq j \leq m$ and in particular this is true for \[M_*(\Gamma_1(n); K)^{Z_1\oplus \cdots \oplus Z_m} = M_*(\Gamma_0(n);K).\qedhere\]
\end{proof}

\begin{remark}
 We remark that the result of the theorem above is sharp. For example, if  $\Gamma = \Gamma_1(23)$, we have by Theorem \ref{thm:deccompact} a decomposition of the form 
 $$g_*\OO_{\MMb_1(23)} \cong \bigoplus_{i = 0}^7 \bigoplus_{k_i} (f_2)_*\OO_{\MMb_1(2)}\tensor \omega^{\tensor (-i)}$$
 after localizing at $3$. As there are weight $1$ cusp forms for $\Gamma_1(23)$ (see e.g.\ \cite{Buz}), the number $k_7$ is not zero. As the degrees of the elements in a basis $\BB$ for the free $M^{\F_3}_*$-module $M_*(\Gamma_1(23); \F_3)$ are unique, it follows easily by considering the basis is Lemma \ref{lem:freeness} that we need basis elements in degree $15$ in $\BB$. The arguments for $l=2$ and $l>3$ are similar.
 \end{remark}

\begin{remark}
 An analogous proof to Theorem \ref{thm:decfield} also shows that the $M_*^K$-module of cusp forms $S_*(\Gamma;K)$ is free. Indeed, Proposition \ref{cor:cusps} implies that $S_*(\Gamma;K)$ are the global sections of $\Omega^1_{\MMb(\Gamma)_K/K} \tensor \omega^{\tensor (\ast -2)}$ and Proposition \ref{prop:split} applies to all vector bundles on $\MMb(\Gamma)_K$. 
\end{remark}

\subsection{Integral results}

The proof of the following two theorems is very similar to that of Theorem \ref{thm:decfield}, using Theorem \ref{thm:deccompact} and Remarks \ref{rem:Z16} and \ref{rem:cuspy}. We just give some sample criteria; other criteria can be deduced from the referenced theorem and remarks. 

\begin{thm}\label{thm:decint6}Let $R$ be a $\Z[\frac1m]$-algebra with $6|m$ and let $\Gamma$ be tame. Then $M_*(\Gamma;R)$ is a free graded $M_*^R$-module if every weight-$1$ modular form for $\Gamma$ over $\F_l$ is liftable to $\Z_{(l)}$ for every prime $l\!\not|\,m$.\end{thm}

\begin{thm}\label{thm:decint23}
Let $R$ be a $\Z_{(l)}$-algebra and $\Gamma$ tame. 
 \begin{enumerate}
  \item Let $l=3$. Then $M_*(\Gamma;R)$ decomposes as a graded $M_*^R$-module into shifted copies of $M_*(\Gamma_1(2); R)$ if every weight-$1$ modular form for $\Gamma$ over $\F_3$ is liftable to $\Z_{(3)}$,
   \item Let $l=2$. Then $M_*(\Gamma;R)$ decomposes as a graded $M_*^R$-module into shifted copies of $M_*(\Gamma_1(3); R)$ if every weight-$1$ modular form for $\Gamma$ over $\F_2$ is liftable to $\Z_{(2)}$.
 \end{enumerate}
\end{thm}

By the same method, we also obtain the following result from Theorem \ref{thm:l}.
\begin{cor}\label{cor:456}
	Let $\Gamma = \Gamma_1(n)$ for $n\geq 5$ or $\Gamma = \Gamma(n)$ for $n\geq 3$ and let $q = 4,5$ or $6$. Let $l$ be a prime not dividing $n$ and $q$. Then $M_*(\Gamma;\Z_{(l)})$ decomposes into shifted copies of $M_*(\Gamma_1(q);\Z_{(l)})$ as a graded $M^{\Z_{(l)}}_*$-module. 
	
	For each $q>6$, there is an arbitrarily large $n$ such that $M_*(\Gamma_1(n);\Z_{(l)})$ does not decompose into shifted copies of $M_*(\Gamma_1(q);\Z_{(l)})$ as a graded $M^{\Z_{(l)}}_*$-module. 
\end{cor}
\begin{proof}The first part follows directly. 
	For the second part it suffices to show the corresponding rational statement, for which we use \cref{cor:dectoowild} and the following: Every vector bundle $\FF$ on $\MMb_{ell,\Q}$ is determined by the graded vector space $H^0(\MMb_{ell,\Q};\FF\tensor \omega^{\tensor *})$ as it decomposes by Proposition \ref{prop:ClassField} into a sum of powers of $\omega$. 
\end{proof} 

Next, we want to prove that rings of modular forms with level structure are finitely generated as modules over the ring of modular forms without level and to give bounds on the degrees of the generators. For the case of $\Gamma_0(n)$ we need the following lemma. 

\begin{lemma}\label{lem:iso}
The map $H^i(\MMb_0(n)_{\Z_{(3)}}; \omega^{\tensor m}) \to H^i(\MM_0(n)_{\Z_{(3)}}; \omega^{\tensor m})$ is an isomorphism for $i>1$, a surjection for $i=1$ and an injection for $i=1$ if $m\geq 2$. 
\end{lemma}
\begin{proof}
We will work implicitly $3$-locally throughout this proof and use the abbreviations $\Gamma = \Gamma_0(n)$ and $M_*(\Gamma)$ for $M_*(\Gamma;\Z_{(3)})$. We consider the Mayer--Vietoris sequence for the cover $D(g^*c_4)\cup \MM_0(n) = \MMb_0(n)$. Using that 
$$H^i(D(g^*c_4); g^*\omega^{\tensor m}) = H^i(D(g^*c_4)\cap \MM_0(n); g^*\omega^{\tensor m}) = 0$$
 for $i>0$ by Proposition \ref{prop:torsion}, it reduces to isomorphisms
 $$H^i(\MMb_0(n); g^*\omega^{\tensor *}) \xrightarrow{\cong} H^i(\MM_0(n); g^*\omega^{\tensor *})$$
 for $i>1$ 
 and to an exact sequence
 
  \hspace*{-3.6cm}
\begin{tikzpicture}
\node (g1) {$M_*(\Gamma)[c_4^{-1}]\,\oplus\, M_*(\Gamma)[\Delta^{-1}]$};
\node (gb) [left=of g1] {$M_*(\Gamma)$};
\node (g13) [right=of g1] {$M_*(\Gamma)[(c_4\Delta)^{-1}]$};
\node (p1) [below=of g1] {$H^1(\MM_0(n); g^*\omega^{\tensor *})$};
\node (pT) [left=of p1] {$H^1(\MMb_0(n); g^*\omega^{\tensor *})$};
\node (p3) [right=of p1] {$0$.};
\draw[->]
(gb) edge node[auto] {$\alpha_*$} (g1)
(g1) edge node[auto] {$\psi_*$} (g13)
(g13) edge[out=-14,in=166] (pT)
(pT) edge node[auto] {$\phi_*$} (p1)
(p1) edge (p3);
\end{tikzpicture}

 It remains to show that $\phi_*$ is injective for $\ast \geq 2$ or equivalently that $\psi_*$ is surjective for $\ast \geq 2$. From now on, we will always assume $\ast \geq 2$. By Proposition \ref{prop:torsion}, $H^1(\MMb_0(n); g^*\omega^{\tensor *})$ is $3$-torsion and thus it suffices to show that $\coker(\psi_*)/3$ is zero. Clearly, $\psi_*/3$ factors through $\coker(\alpha_*/3)$. This motivates the following diagram with exact rows:
 
 \hspace{-0.8cm}
 \[
 \scalebox{0.83}{
  \xymatrix{
  0 \ar[r] & M_*(\Gamma)/3\ar[r] \ar[d]^{\alpha_*/3}& M_*(\Gamma; \F_3) \ar[r]\ar[d]^{\alpha_*^{\F_3}} & H^1(\MMb_0(n); \omega^{\tensor *}) \ar[r]\ar[d]^{\Phi_*} & 0 \\
  0 \ar[r] & M_*(\Gamma)[c_4^{-1}]/3\oplus M_*(\Gamma)[\Delta^{-1}]/3\ar[r]\ar[d] & M_*(\Gamma; \F_3)[c_4^{-1}] \oplus M_*(\Gamma; \F_3)[\Delta^{-1}] \ar[r] \ar[d]& H^1(\MM_0(n); \omega^{\tensor *}) \ar[d]\ar[r] & 0 \\
  & \coker(\alpha_*)/3 \ar[r] & \coker(\alpha_*^{\F_3}) \ar[r] & \coker(\alpha_*^1)
  } 
  }
 \]

As $\Phi_*$ is surjective, we can use the snake lemma to deduce that $\coker(\alpha_*)/3 \to \coker(\alpha_*^{\F_3})$ is surjective as well. Thus, $\psi_*/3$ is surjective if and only if
$$\psi_*^{\F_3}\colon M_*(\Gamma; \F_3)[c_4^{-1}]\,\oplus\, M_*(\Gamma;\F_3)[\Delta^{-1}] \to M_*(\Gamma;\F_3)[(c_4\Delta)^{-1}]\cong M_*(\Gamma)[(c_4\Delta)^{-1}]/3 $$
is surjective. By Theorem \ref{thm:decfield}, we know that $M_*(\Gamma;\F_3)$ is free as a graded $M_*^{\F_3}$-module with basis elements in degrees at most $15$. Recall that $M_*^{\F_3} \cong \F_3[b_2,\Delta]$ and $c_4$ corresponds to $b_2^2$. The map 
$$\tilde{\psi}\colon  M_*^{\F_3}[c_4^{-1}]\,\oplus\, M_*^{\F_3}[\Delta^{-1}] \to M_*^{\F_3}[(c_4\Delta)^{-1}] $$
is surjective in degrees $\ast \geq -13$ as the first monomial not in the image is $b_2^{-1}\Delta^{-1}$ of degree $-14$. As the map $\psi^{\F_3}$ decomposes into shifted copies of $\tilde{\psi}$, we see that $\psi_*^{\F_3}$ is surjective for $\ast \geq 15-13 = 2$. 
\end{proof}

\begin{thm}
 For every noetherian $\Z[\frac1n]$-algebra $R$, the ring $M_*(\Gamma;R)$ is finitely generated as a graded $M_*^R$-module; if $\Gamma$ is tame it is generated in degrees $\leq 17$ and if $\Gamma = \Gamma_0(n)$ and $\frac12\in R$, it is generated in degree $\leq 21$. Moreover, $\widetilde{M}_*(\Gamma;R)$ is also finitely generated as a graded $\widetilde{M}_*^R$-module.
\end{thm}
\begin{proof}
 First note that 
 $$\widetilde{M}_*(\Gamma;R) \cong M_*(\Gamma;R)[\Delta^{-1}]\cong M_*(\Gamma;R) \tensor_{M_*^R}\widetilde{M}_*^R$$ 
 so that the last statement follows from the first. 
 
 To show that $M_*(\Gamma;R)$ is a finitely generated graded $M_*^R$-module (generated below a fixed degree), we can assume that $R$ is $l$-local for a prime $l$. It is enough to show that $M_*(\Gamma;R)/l$ is a finitely generated $M_*^R/l$-module (generated below a fixed degree) by the graded Nakayama lemma as every degree is a finitely generated $R$-module. Consider the short exact sequence
 \[0 \to M_*(\Gamma;R)/l \to M_*(\Gamma;R/l) \xrightarrow{\partial} H^1(\MMb(\Gamma)_R; \omega^{\tensor *})[l] \to 0\]
 as in \eqref{eq:fundamentalexact}.
 As $R/l$ is a flat $\F_l$-algebra, $M_*(\Gamma;R/l) \cong M_*(\Gamma;\F_l)\tensor_{\F_l} R/l$ and thus $M_*(\Gamma;R/l)$ is free over $M_*^{R/l}$ with basis elements in degrees at most $15$ by Theorem \ref{thm:decfield}. The $M_*^\Z$-module $M_*^{\F_l}$ is a generated in degrees $\leq 3$ if $l=2$, degrees $\leq 2$ if $l=3$ and in degree $0$ else by the remarks after Proposition \ref{prop:dapres}. As the $(M_*^\Z \tensor R)$-action on $M_*^{R/l}$ factors through $M_*^R$, we see that $M_*^{R/l}$ is a finitely generated $M_*^R$-module with generators in degrees at most $17$ if $l=2,3$ and at most $11$ else. In particular, $M_*(\Gamma;R)/l$ is a finitely generated $M_*^R$-module.
 
 It remains to provide bounds on the degrees of the generators. Assume first that $\Gamma$ is tame and let $x \in \ker(\partial) \cong M_*(\Gamma;R)/l$ of degree more than $17$. Write $x = \sum_i \lambda_ig_i$ with $\lambda_i \in M_*^R$ and $|g_i| \leq 17$ in $M_*(\Gamma;R/l)$. As $H^1(\MMb(\Gamma)_R; \omega^{\tensor m}) = 0$ for $m\geq 2$ by Proposition \ref{prop:coh}, $g_i \in \ker(\partial)$ if $|g_i| \geq 2$. If $|g_i| \leq 1$, then $|\lambda_i| \geq 17$ and $\lambda_i$ decomposes into $\sum_j \mu_j\nu_j$ with $|\nu_j|$ and $|\mu_j|$ at least $2$ and thus $\partial(\mu_jg_i) = 0$. We see that $x$ is a $M_*^R$-linear combination of elements of $\ker(\partial)$ of lower degree. Thus, $\ker(\partial)$ is generated in degrees at most $17$. 
 
 Next, assume that $\Gamma = \Gamma_0(n)$ and $l=3$. Let $x \in \ker(\partial)$ be of degree more than $21$ and write $x = \sum_i \lambda_ig_i$, where the $\lambda_i$ are monomials in $c_4, c_6$ and $\Delta$. Note that we can choose $g_i$ to be in degree at most $15$ or of the form $b_2b$ for some $b \in M_*(\Gamma_0(n);R/3)$. 
 
For degree reasons, every $\lambda_i$ must be divisible by $\Delta$, $c_4^2$ or $c_6$. By Proposition \ref{prop:torsion} and Lemma \ref{lem:iso}, $\partial(y) =0$ if $y$ is divisible by $c_4$ or $c_6$. Thus, if $\lambda_i = c_4^2z$, then $c_4z$ is an element in $\ker(\partial)$ of smaller degree than $x$. Assume that $\lambda_i$ is divisible by $c_6$, but not by $\Delta$. Then either $|g_i| = 17$ (and thus $g_i = b_2b$ so that $c_6g_i = c_4^2b$) or $\lambda_i$ is actually divisible by $c_4^2$, $c_4c_6$ or $c_6^2$ and in each case we can argue as in the last sentence. Thus, we can write $x = g + \Delta h$ with $g$ a $M_*^R$-linear combination of elements of $\ker(\partial)$ of lower degree than $x$. But Lemma \ref{lem:iso} implies that $h \in \ker(\Delta)$ as well, which implies the result. 
\end{proof}

We turn to modular functions, treating the cases of $\Gamma_1(2)$ and $\Gamma_1(3)$ first. 
\begin{lemma}
The module $\widetilde{M}_*(\Gamma_1(2); \Z[\frac12])$ is free over $\widetilde{M}_0^{\Z[\frac12]}$. Likewise, $\widetilde{M}_*(\Gamma_1(3); \Z[\frac13])$ is free over $\widetilde{M}_0^{\Z[\frac13]}$.
\end{lemma}
\begin{proof}
We claim that $M_*(\Gamma_1(2); \Z[\frac12]) \cong \Z[\frac12][b_2,b_4]$ is free over $\Z[\frac12][c_4,\Delta]$ 
of rank $6$. Counting ranks shows that it is enough to produce a generating system with generators in degrees $0,2,4,6,8$ and $10$. The images of $c_4$ and $\Delta$ in $\Z[\frac12][b_2,b_4]$ are $b_2^2-24b_4$ and $\frac14(b_2^2b_4^2-32b_4^3)$. Thus, $\Z[b_2,b_4]$ is free over $\Z[c_4]$ with basis elements $b_2^ib_4^j$ with $i\in\{0,1\}$. As $4\Delta - b_4^2c_4 = -8b_4^3$, we see that if $x$ and $b_4^2x$ are in a $\Z[\frac12][c_4,\Delta]$-submodule of $\Z[\frac12][b_2,b_4]$ that also $b_4^3x$ is in it. Thus, $\{1,b_2,b_4, b_2b_4, b_4^2, b_2b_4^2\}$ is a generating set. We see that $M_*(\Gamma_1(2); \Z[\frac12]) \cong \Z[\frac12][b_2,b_4]$ is free over $\Z[\frac12][c_4,\Delta]$. 
This implies that $\widetilde{M}_*(\Gamma_1(2);\Z[\frac12])$ is a free module over $\Z[\frac12,c_4,\Delta^{\pm 1}]$ and the latter is free over $\Z[\frac12][j] = \Z[\frac12][\frac{c_4^3}{\Delta}]$. 

Similarly, we claim that $\{a_1^ia_3^j\}_{0\leq i \leq 3, 0\leq j\leq 3}$ is a basis of $M_*(\Gamma_1(3); \Z[\frac13]) \cong \Z[\frac13][a_1, a_3]$ as a $\Z[\frac13][c_4,\Delta]$-module.  This implies that $\widetilde{M}_*(\Gamma_1(3);\Z[\frac13])$ is a free module over $\Z[\frac13][j]$. 
\end{proof}

\begin{prop}
The morphism $j\colon \MM_{ell} \to \A^1$ is flat.
\end{prop}
\begin{proof}
By the last lemma, the compositions 
$$\Spec \widetilde{M}_*(\Gamma_1(2);\Z[\frac12]) \simeq \MM_1^1(2) \to \MM_{ell} \xrightarrow{j} \mathbb{A}^1$$
and 
$$\Spec \widetilde{M}_*(\Gamma_1(3);\Z[\frac12]) \simeq \MM_1^1(3) \to \MM_{ell} \xrightarrow{j} \mathbb{A}^1$$
are flat. Here, $\MM_1^1(n)$ classifies elliptic curves with $\Gamma_1(n)$-level structure and chosen invariant differential. This suffices to show flatness as $\MM_1(2) \sqcup \MM_1(3) \to \MM_{ell}$ is an \'etale cover and hence $\MM_1^1(2)\sqcup \MM_1^1(3) \to \MM_{ell}$ is a smooth cover. 
\end{proof}

\begin{cor}
Let $\Gamma = \Gamma_1(n)$, $\Gamma(n)$ or $\Gamma_0(n)$ and $R$ a $\Z[\frac1n]$-algebra. Then $\widetilde{M}_0(\Gamma; R)$ is a finitely generated free module over $\widetilde{M}_0^R \cong R[j]$. 
\end{cor}
 \begin{proof}
By \cite[Proposition 6.4]{Ces17}, the base change
$$\widetilde{M}_0(\Gamma;\Z[\tfrac1n]) \tensor R \to \widetilde{M}_0(\Gamma; R)$$
is an isomorphism. Thus we can reduce to $R = \Z[\frac1n]$. 
 
 By the last proposition and (the proof of) Proposition \ref{prop:basicprops}, the composition
 $$\MM(\Gamma) \to \MM_{ell,\Z[\frac1n]} \xrightarrow{j} \mathbb{A}^1_{\Z[\frac1n]}$$
 is flat and furthermore finite if $\MM(\Gamma)$ is representable. Choose $\Gamma' \subset \Gamma$ such that $\MM(\Gamma')$ is representable and the resulting map $h\colon \MM(\Gamma') \to \MM(\Gamma)$ is surjective; it is also automatically finite and flat as both source and target are finite over $\MM_{ell,\Z[\frac1n]}$ and smooth over $\Z[\frac1n]$ by Proposition \ref{prop:basicprops}. Denoting the map $\MM(\Gamma) \to \mathbb{A}^1_{\Z[\frac1n]}$ by $j_\Gamma$, the sheaf $(j_\Gamma h)_*\OO_{\MM(\Gamma')}$ is coherent. It suffices to show that $(j_\Gamma)_*\OO_{\MM(\Gamma)}\to (j_\Gamma h)_*\OO_{\MM(\Gamma')}$ is injective to deduce that $(j_\Gamma)_*\OO_{\MM(\Gamma)}$  is coherent as well and thus a vector bundle. As $(j_\Gamma)_*$ is exact and the map $\OO_{\MM(\Gamma)}\to h_*\OO_{\MM(\Gamma')}$ is the inclusion of a direct summand after applying the faithfully flat map $h$, this injectivity follows.
 
 Hence, 
 $$\widetilde{M}_0(\Gamma; \Z[\frac1n]) \cong H^0(\MM(\Gamma); \OO_{\MM(\Gamma)})$$
 is a projective $\Z[\frac1n][j]$-module of finite rank. 
   As every projective $\Z[\frac1n][j]$-module is free by Seshadri's theorem \cite[Theorem II.6.1]{Lam06}, the result follows.  
 \end{proof}
 
 This finishes the proof of all theorems claimed in the introduction except for Corollary \ref{cor:CM} and the necessity parts in Part (2) of Theorem \ref{thm:main}, which will both be achieved in the next subsection. 

\subsection{Cohen-Macaulay properties}
In this subsection, we treat the question under which conditions the graded ring $M_*(\Gamma;R)$ is Cohen--Maucaulay. Here and in the following, all terms from commutative algebra are meant to be the ones suitable for graded rings, i.e.\ all ideals are assumed to be homogeneous etc. For simplicity, we only consider the cases where $R$ is a field or $\Z_{(l)}$. 
\begin{thm}\label{thm:CM}
 For $K$ a field with $\Char(K)$ not dividing $n$ and $\Gamma = \Gamma_0(n), \Gamma_1(n)$ or $\Gamma(n)$, the ring $M_*(\Gamma;K)$ is Cohen--Macaulay. If $\Gamma$ is additionally tame for $\Z_{(l)}$, the ring $M_*(\Gamma;\Z_{(l)})$ is Cohen--Macaulay if and only if every weight-$1$ modular form for $\Gamma$ over $\F_l$ is liftable to $\Z_{(l)}$.
\end{thm}
\begin{proof}
 Consider first the case of a field. The ring $M_*(\Gamma;K)$ is graded local. Furthermore, it is free of finite rank as a module over $M_*^K$ by Theorem \ref{thm:decfield} and thus of Krull dimension $2$ as $M_*^K$ has Krull dimension $2$ by Proposition \ref{prop:dapres}. Moreover, the sequence $(c_4,\Delta)$ is regular as $(c_4, \Delta)$ is a regular sequence on $M_*(\Gamma;K)$. 
 
 Now consider the case of $\Z_{(l)}$ and $\Gamma$ tame. The ring $M_*(\Gamma;\Z_{(l)})$ is graded local of Krull dimension $3$. Indeed, the inclusion $M_*(\Gamma;\Z_{(l)})/l \subset M_*(\Gamma;\F_l)$ induces a bijection on the set of prime ideals and thus this ring has Krull dimension $2$ as well; as $l$ is a non-zero divisor, this implies that $M_*(\Gamma;\Z_{(l)})$ has Krull dimension $3$. Assume that every weight-$1$ modular form for $\Gamma$ over $\F_l$ is liftable to $\Z_{(l)}$. We claim that $(l,c_4,\Delta)$ is a regular sequence. By an analogous argument to the above and Theorems \ref{thm:decint6} and \ref{thm:decint23}, we have to check for this only the cases of $\Gamma = \Gamma_1(3)$ (for $l=2$), $\Gamma = \Gamma_1(2)$ (for $l=3$) and $\Gamma = SL_2(\Z)$ (for $l>3$). For the first, note that we have $c_4 \equiv a_1^4$ and $\Delta \equiv a_3^4+a_1^3a_3^3 \mod 2$; this clearly forms a regular sequence in $\F_2[a_1,a_3]$. For the second case, note that we have $c_4 \equiv b_2^2$ and $\Delta \equiv b_2^2b_4^2-b_4^3 \mod 3$; this forms a regular sequence in $\F_3[b_2,b_4]$. The last case is the sequence $(c_4, \frac1{1728}(c_4^3-c_6^2))$ in $\F_{l}[c_4,c_6]$ for $l>3$. Thus, $M_*(\Gamma;\Z_{(l)})$ is Cohen--Macaulay if every weight-$1$ modular form for $\Gamma$ over $\F_l$ is liftable to $\Z_{(l)}$
 
 Assume now that there is a weight-$1$ modular form $f$ for $\Gamma$ over $\F_l$ that is not liftable to $\Z_{(l)}$. We know that $M_*(\Gamma;\Z_{(l)})/l$ injects into $M_*(\Gamma;\F_l)$ (and is an isomorphism for $*\neq 1$). The latter is an integral domain by Proposition \ref{prop:irreducible} and thus so is the former. Thus, $(l,c_4)$ forms a regular sequence on $M_*(\Gamma;\Z_{(l)})$. By \cite[Section 17.2]{Eis95} all maximal regular sequence in the augmentation ideal of $M_*(\Gamma;\Z_{(l)})$ have the same length. Thus, $M_*(\Gamma;\Z_{(l)})$ can only be Cohen--Macaulay if we can extend the regular sequence $(l,c_4)$ by another element $x$ of positive degree. The element $c_4f$ has a preimage $\tilde{f}$ in $M_*(\Gamma;\Z_{(l)})/l$. This is nonzero in $M_*(\Gamma;\Z_{(l)})/(l,c_4)$. Indeed, if $\tilde{f} = c_4y$, then $y$ would be a lift of $f$ as $c_4$ operates injectively on $M_*(\Gamma;\F_l)$. On the other hand, 
$x\tilde{f}$ is zero in $M_*(\Gamma;\Z_{(l)})/(l,c_4)$ for every element $x$ of positive degree. Indeed, $xf$ has a lift $z\in M_*(\Gamma;\Z_{(l)})/l$ and $c_4z = x\tilde{f}$ in $M_*(\Gamma;\Z_{(l)})/l$ because both elements have the same image in $M_*(\Gamma;\F_l)$. Thus, $(l,c_4,x)$ is not a regular sequence and $M_*(\Gamma;\Z_{(l)})$ is not Cohen--Macaulay. 
\end{proof}

The last proof also shows that $M_*(\Gamma;\Z_{(l)})$ must be Cohen--Maucaulay if $M_*(\Gamma;\Z_{(l)})$ splits into shifted copies of $M_*(\Gamma_1(3);\Z_{(l)})$, $M_*(\Gamma_1(2);\Z_{(l)})$ or $M_*^{\Z_{(l)}}$ (the latter rests on $(l,c_4,\Delta)$ also being a regular sequence on $M_*^{\Z_{(l)}}$ if $l=2$ or $3$). This gives the following corollary: 
\begin{cor}\label{cor:onlyiflift}
 Assume that $\Gamma$ is tame for $\Z_{(l)}$ and there is a weight-$1$ modular form for $\Gamma$ over $\F_l$ not liftable to $\Z_{(l)}$. Then $M_*(\Gamma;\Z_{(l)})$ does not split into shifted copies of $M_*(\Gamma_1(3);\Z_{(l)})$, $M_*(\Gamma_1(2);\Z_{(l)})$ or $M_*^{\Z_{(l)}}$. 
\end{cor}

To finish the proof of Theorem \ref{thm:main} (more precisely of the last remaining part, namely the statement that $M_*(\Gamma; \Z_{(l)})$ can only be free over $M_*^{\Z_{(l)}}$ if $l \geq 5$), the following suffices. 
\begin{prop}\label{prop:onlyif5}
 A direct sum of shifted copies of $M_*(\Gamma_1(3);\Z_{(2)})$ cannot be free over $M_*^{\Z_{(2)}}$ and a direct sum of shifted copies of $M_*(\Gamma_1(2);\Z_{(3)})$ cannot be free over $M_*^{\Z_{(3)}}$.
\end{prop}
\begin{proof}
 By \cite{Ati56}, graded $M^{\Z_{(l)}}_*/l$-modules satisfy the Krull--Schmidt theorem (i.e.\ finitely generated graded modules decompose uniquely into indecomposables). Thus it is enough to show that $M_*(\Gamma_1(3);\Z_{(2)})/2$ is not a free $M_*^{\Z_{(2)}}/2$-module and similarly that $M_*(\Gamma_1(2);\Z_{(3)})/3$ is not a free $M_*^{\Z_{(3)}}/3$-module. 
 
 We begin with the former case. The elements $c_4$ and $c_6$ in $M_*^\Z$ go to $a_1^4$ and $a_1^6$ in $M_*(\Gamma_1(3);\Z_{(2)})/2$. As $1$ and $a_1^2$ would have to be parts of any basis of $M_*(\Gamma_1(3);\Z_{(2)})/2$ over $M_*^{\Z_{(2)}}/2$, the former cannot be free over the latter. 
 
 In the other case, we have $c_4 = b_2^2+216b_4$ and $c_6 = b_2^3-576b_2b_4$ and thus we get $c_4 \equiv b_2^2$ and $c_6 \equiv b_2^3$ modulo $3$. A possible basis of $M_*(\Gamma_1(2);\Z_{(3)})/3$ as a module over $M_*^{\Z_{(3)}}/3$ must involve (up to scalar) the elements $1$ and $b_2$. But this cannot be as $c_4 \cdot b_2 = c_6 \cdot 1$. 
\end{proof}

\appendix
\section{Vector bundles on weighted projective lines}\label{sec:weighted}
The aim of this section is to generalize some well-known facts about coherent sheaves and vector bundles on projective spaces to weighted projective stacks. This is relevant for our purposes because several compactified moduli stacks of elliptic curves (with level structure) are weighted projective stacks (see Example \ref{exa:moduli}).

\begin{defi}\label{def:wps} For $w_0,\dots, w_n$ positive integers  and a commutative ring $R$, the \textit{weighted projective stack} $\PP_R(w_0,\dots, w_n)$ is the (stack) quotient of $\A_R^{n+1}-\{0\}$ by the multiplicative group $\G_m$ under the action which is the restriction of the map
\[ \phi\co \G_m\times \A_R^{n+1} \to \A_R^{n+1}, \]
which is induced by the ring map 
\begin{align*} \Z[t, t^{-1}]\otimes R[t_0,\dots, t_n] &\leftarrow R[t_0,\dots, t_n] \\
 t^{w_i}\tensor t_i &\mapsfrom t_i,
\end{align*}
to $\G_m\times (\A^{n+1}_R-\{0\})$. Here, $\A_R^{n+1}-\{0\}$ denotes the complement of the zero point, i.e.\ of the common vanishing locus of all $t_i$. On geometric points, the action corresponds to the map $(t, t_0,\dots, t_n) \mapsto (t^{w_0}t_0,\dots, t^{w_n}t_n)$. In the special case of $n=1$ we speak of a \textit{weighted projective line}.\end{defi}

As explained in \cite[Section 2]{Mei13}, this is a smooth and proper Artin stack over $\Spec R$, Deligne--Mumford if all $w_i$ are invertible on $R$. 

Recall that a grading on a commutative ring $A$ is equivalent to a $\Gm$-action on $\Spec A$. Moreover, there is an equivalence between graded $A$-modules and quasi-coherent sheaves on $\Spec A/\Gm$ given by pullback to $\Spec A$ and this equivalence is compatible with $\tensor$. The map $\phi$ above gives a $\G_m$-action on $\A^{n+1}_R$ and this corresponds to the grading $|t_i| = w_i$. The category of quasi-coherent sheaves on $\A^{n+1}_R/\G_m$ is thus equivalent to graded $R[t_0,\dots, t_n]$-modules. 

For $M$ a graded module, denote by $M[m]$ the graded module with $M[m]_k = M_{m+k}$. Then $R[t_0,\dots, t_n][m]$ is a graded $R[t_0,\dots, t_n]$-module, which corresponds to a line bundle on $\A^{n+1}_R/\G_m$ whose restriction to $\PP_R(w_0,\dots, w_n)$ we denote by $\OO(m)$. As usual, we set $\FF(m) = \FF \tensor \OO(m)$. It is easy to see that for a quasi-coherent sheaf $\FF$ on $\PP_R(w_0,\dots, w_n)$, the graded global sections $\Gamma_*(\FF) = H^0(\PP_R(w_0,\dots, w_n); \bigoplus_{m\in\Z}\FF(m))$ are exactly 
the graded $R[t_0,\dots, t_n]$-module corresponding to $\FF$. 

The following theorem summarizes some of the fundamental properties of $\OO(m)$:

\begin{thm}\label{thm:fundamentalweighted}Let $X = \PP_R(w_0,\dots, w_n)$.
 \begin{enumerate}
  \item \label{item:P1}The sheaf $\OO(1)$ is ample in the sense that for every coherent sheaf $\FF$ on $X$, there is a surjection from a sum of sheaves of the form $\OO(m)$ with $m\geq 1$ to $\FF$ .
    \item \label{item:P1.5}For any coherent sheaf $\FF$, there exist an $m\geq 1$ such that $H^i(X;\FF(m)) = 0$ for all $i>0$.
  \item \label{item:P2}The sheaf $\OO(-\sum_{i=0}^n w_i)$ is dualizing in the sense that there are natural isomorphisms
  \[\Hom_{\OO_X}(\FF,\OO(-\sum_{i=0}^n w_i)) \xrightarrow{\cong} \Hom_R(H^n(X; \FF),R)\]
  for all coherent sheaves $\FF$ on $X$. Moreover, $\OO(-\sum_{i=0}^n w_i)$ agrees with $\Lambda^n\Omega^1_{X/R}$. 
 \end{enumerate}
\end{thm}
\begin{proof}
 The proofs are analogous to the classical proofs for projective spaces. In some more detail:
 
 Let $\FF$ be a coherent sheaf on $\PP_R(w_0,\dots, w_n)$ and set $M = \Gamma_*(\FF)$. The stack $X$ is covered by the non-vanishing loci $D(t_i)$, where $t_i\in H^0(X; \OO(w_i))$. Furthermore, 
 $$D(t_i) \simeq \Spec R[t_0,\dots, t_n][t_i^{-1}] / \G_m.$$
 The restriction of $\FF$ to $D(t_i)$ corresponds to the graded $R[t_0,\dots, t_n][t_i^{-1}]$-module $M[t_i^{-1}]$.  Choose generators $s_{ij}$ of $M[t_i^{-1}]$. By multiplying with a power of $t_i$ we can assume that all $s_{ij}$ are actually in $M$ and of positive degree and thus define elements in $\Hom_X(\OO(m), \FF)$ for some $m\geq 1$. Taking the sum of all these maps defines a surjection, proving (\ref{item:P1}). 
 
  For (\ref{item:P1.5}), we can argue by downward induction on $i$ as in \cite[Thm III.5.2]{Har77}, once we know that $X$ has cohomological dimension $\leq n$. This is clear as $X$ can be covered by the $(n+1)$ open substacks $D(t_i)$, on which the global sections functor is exact on quasi-coherent sheaves (because it corresponds to taking the degree-$0$ piece of a graded module). 
 
 That $\OO(-\sum_{i=0}^n w_i)$ acts as a dualizing sheaf for all line bundles of the form $\OO(m)$ was shown in \cite[Prop 2.5]{Mei13}. The general case follows as in \cite[Thm III.7.1]{Har77} because $\OO(1)$ is ample. To identify $\Lambda^n\Omega^1_{X/R}$, consider its pullback to $\A_R^{n+1}-\{0\}$, which is free of rank one with basis $dt_0 \wedge \cdots \wedge dt_n$. The $\G_m$-action on this form identifies $\Lambda^n\Omega^1_{X/R}$ with $\OO(-\sum_{i=0}^n w_i)$. 
\end{proof}

We also want to recall the cohomology of $\OO(m)$ on $\PP_R(a,b)$ from \cite[Prop 2.5]{Mei13}.

\begin{prop}\label{prop:projcoh}
 Let $B(m)$ be the set of pairs $(\lambda, \mu)$ of negative integers with $\lambda a + \mu b = m$. Then $H^1(\PP_R(a,b);\OO(m))$ is isomorphic to the free $R$-module on $B(m)$. 
\end{prop}

I learned the following result from Angelo Vistoli \cite[Prop 3.4]{Mei13}.
\begin{prop}\label{prop:ClassField}Let $K$ be an arbitrary field, $w_0, w_1\in\N$. Then every vector bundle $\FF$ on $\PP_K(w_0, w_1)$ is a direct sum of line bundles of the form $\OO(m)$.\end{prop}

We want to prove a generalization to weighted projective lines over more general rings, which is in the spirit of \cite[Theorem 1.4]{H-S99b}. First we need three lemmas. 

\begin{lemma}\label{lem:divisor}
 Let $k$ be an algebraically closed field and let $s$ be a section of a vector bundle $\EE$ on $X = \PP_k(a,b)$. Assume that $s$ vanishes at some geometric point of $X$. Then $s$ is the image of a section of $\EE(-j)$ along the map $\EE(-j) \to \EE$ for some $j>0$. 
\end{lemma}
\begin{proof}
 This would easily follow from a suitable formalism of divisors on Artin stacks. We will argue in a more elementary way. 
 
 Let $M$ be the global sections of the pullback of $\EE$ to $\A_k^2-\{0\}$. This pullback can be extended to a vector bundle $\FF$ on $\A^2_k$ with the same global sections. The module $M$ is a (finite rank free) graded module over the polynomial ring $k[x,y]$ with $|x| = a$ and $|y| = b$. Assume that $s$ vanishes at a geometric point that is the image of $(u,v) \in \A^2_k-\{0\}$. Then (the pullback of) $s$ also vanishes on $f(\mathbb{G}_{m,k})$ for $f\colon \A^1_k \to \A^2_k$ the map described by the formula $\lambda \mapsto (\lambda^au, \lambda^bv)$ for $\lambda \in k$. We claim that $f(\mathbb{G}_{m,k})$ is closed in $\A^2_k-\{0\}$.
 
 First assume that $u=0$ or $v=0$, say $v=0$. Then $f$ can on $\mathbb{G}_{m,k}$ be written as the composition 
 $$\mathbb{G}_{m,k} \to \mathbb{G}_{m,k} \cong \mathbb{G}_{m,k} \times \{0\} \to \A^2_k-\{0\},$$
 where the first map is the surjection $\lambda \mapsto \lambda^au$ and the last map is obviously a closed immersion. 
 
 If $u$ and $v$ are nonzero, let $g$ be $\gcd(a,b)$. Because 
 $$\mathbb{G}_{m,k} \to \mathbb{G}_{m,k}, \qquad \lambda \mapsto \lambda^g$$
 is surjective, we can assume that $a$ and $b$ are coprime. Thus, $f$ defines closed immersions $\mathbb{G}_{m,k} \to \Spec k[x^{\pm 1},y]$ and $\mathbb{G}_{m,k} \to \Spec k[x,y^{\pm 1}]$. Hence $f(\mathbb{G}_{m,k})$ is closed in $\A^2_k-\{0\}$. 
 
 It follows that $A = f(\mathbb{A}^1_k)$ is closed and irreducible in $\A^2_k$ and thus must be the closure of $f(\mathbb{G}_{m,k})$. Thus, $s$ vanishes on $A$. The set $A$ corresponds to a prime ideal $\mathfrak{p} \subset k[x,y]$ of height $1$. As $k[x,y]$ is factorial, $\mathfrak{p}$ contains a prime element $q$ and thus $\mathfrak{p} = (q)$. 
 As $q = q(x,y)$ and $q(\lambda^ax, \lambda^by)$ for $\lambda \in k^\times$ have both the zero set $A$, they must be unit multiple of each other and it follows that $q$ is homogeneous of some positive degree $j$. Thus, the element $m\in M$ corresponding to $s$ must be of the form $qm'$ for $m'\in M$, where $|m'| = |m|-j$. 
\end{proof}

\begin{lemma}\label{lem:basechange}
 Let $\EE$ be a quasi-coherent sheaf on a quasi-compact Artin stack $X$ with affine diagonal and assume that $\EE$ is flat over $R$. Let $R\to S$ be a morphism of commutative rings and denote by $f$ the projection $$Y=X\times_{\Spec R}\Spec S\xrightarrow{f} X.$$
 If $H^i(X;\EE)$ is a flat $R$-module for $i> p$, then 
 $$H^p(Y;f^*\EE) \cong H^p(X;\EE)\tensor_R S.$$ 
 More generally, there is a spectral sequence 
 $$\Tor_s^R(H^t(X;\EE), S) \Rightarrow H^{t-s}(Y;f^*\EE).$$
\end{lemma}
\begin{proof}
 Let $\{U_i\to X\}_{0\leq i \leq n}$ be an fpqc covering by affine schemes and let $\check{C}(\EE)$ be the corresponding \v{C}ech complex, whose cohomology is $H^*(X;\EE)$. We can compute $H^*(Y,f^*\EE)$ as the cohomology of $\check{C}(\EE)\tensor_R S$. The resulting K{\"u}nneth spectral sequence
 $$\Tor_s^R(H^t(X;\EE), S) \Rightarrow H^{t-s}(Y;f^*\EE)$$
 implies the result. 
\end{proof}

\begin{lemma}\label{lem:constvector}
Let $\XX$ be a normal noetherian Artin stack and $\FF$ a coherent sheaf on $\XX$. Assume that there is an integer $n$ such that for every point $x\colon \Spec k \to \XX$, the pullback $x^*\FF$ is free of rank $n$. Then $\FF$ is a vector bundle.
\end{lemma}
\begin{proof}
 By taking a smooth cover, we reduce to the case of a noetherian normal scheme $X$. As a coherent sheaf over a noetherian scheme is a vector bundle if and only if its stalks are free over the stalks of the structure sheaf, we can assume that $\XX = \Spec A$ for a noetherian local domain $A$. Here, the statement is part of \cite[Thm 2.9]{Mil80}.
\end{proof}

\begin{thm}\label{thm:VectorBundle}
 Let $\EE$ be a vector bundle on $X = \PP_R(w_0,w_1)$ for $R$ a noetherian and normal ring. Then the following conditions are equivalent. 
 \begin{enumerate}
  \item \label{item:free} Both $H^0(X;\EE(m))$ and $H^1(X;\EE(m))$ are free $R$-modules for all $m\in\Z$. 
  \item \label{item:dec} The vector bundle $\EE$ decomposes into a sum of line bundles of the form $\OO(m)$. 
 \end{enumerate}

\end{thm}
\begin{proof}
By Proposition \ref{prop:projcoh}, the part (\ref{item:dec}) implies (\ref{item:free}) because the cohomology of $\OO(m)$ is a free $R$-module. 
 
 Now we assume (\ref{item:free}) and want to prove (\ref{item:dec}). The proof will be similar to one of the standard proofs for an unweighted projective line over a field. We will argue by induction on the rank of $\EE$ and assume that the theorem has been proven for all ranks $\leq r$ and that $\EE$ has rank $r+1$. 
 
 Denote by $\EE^\vee$ the $\OO_X$-dual of $\EE$. By Theorem \ref{thm:fundamentalweighted}, there is a maximal $m$ such that $H^1(X; \EE^\vee(m)) \neq 0$. Setting $m_0 = -m-w_0-w_1$, we claim that $m_0$ is the smallest index such that $H^0(X;\EE(m_0)) \neq 0$. Indeed: By Lemma \ref{lem:basechange}, we have for $j\colon \Spec k \to \Spec R$ (for $k$ a field) and every $i\in\Z$ an isomorphism $H^1(X; \EE^\vee(i))\tensor_R k \cong H^1(X_k; j^*\EE^{\vee}(i))$. Thus, for every $i>m$, the group $H^1(X_k; j^*\EE^{\vee}(i))$ vanishes and there exists a point $j$ of $\Spec R$ such that $H^1(X_k; j^*\EE^{\vee}(m))$ is nonzero. Serre duality implies that $H^0(X_k; j^*\EE(m_0)) \neq 0$ for this $j$ and $H^0(X_k;j^*\EE(i)) = 0$ for every $j\colon \Spec k \to \Spec R$ if $i<m_0$. By Lemma \ref{lem:basechange} again, $$H^0(X_k;j^*\EE(i)) \cong H^0(X;\EE(i)) \tensor_R k$$
  because $H^1(X;\EE(i))$ is a free $R$-module, which shows the claim that $m_0$ is minimal with $H^0(X;\EE(m_0)) \neq 0$. 
 
 By possibly tensoring $\EE$ with $\OO(-m_0)$, we can assume that $m_0 = 0$. Choose now an element $s \in H^0(X;\EE)$ that is part of an $R$-basis. Then we consider the sequence
 \begin{align}\label{align:vector}\OO_X \xrightarrow{s} \EE \to \FF \to 0.\end{align}
 We want to show that $s$ defines an injection and that its cokernel $\FF$ is a vector bundle. By Lemma \ref{lem:basechange}, we see that $s$ is still nonzero after base change to an arbitrary geometric point $j\colon \Spec k \to \Spec R$. We claim that $s$ does not vanish at any geometric point of $X_k$. Indeed, if $s$ had a zero on $X_k$, then $s$ would by Lemma \ref{lem:divisor} define a nonzero section of $j^*\EE(i)$ for some $i<0$. But by Lemma \ref{lem:basechange},
 $$H^0(X_k;j^*\EE(i)) \cong H^0(X;\EE(i))\tensor_R k = 0$$
 for $i<0$. 

Thus, $\OO_X \xrightarrow{s} \EE$ is an injection and $\FF$ has rank $r$ over every geometric point and is thus a vector bundle again by Lemma \ref{lem:constvector}. Thus $\FF \cong \OO(b_1)\oplus \cdots \oplus \OO(b_r)$ by induction. By shifting the sequence (\ref{align:vector}) by $(-i)$, it is easy to see that $H^0(X;\FF(-i)) = 0$ for $0<i<w_0+w_1$. Furthermore, for every $b>0$, take $i$ with $0<i\leq w_0$ and $i\equiv b \mod w_0$; then $H^0(X; \OO(b-i)) \neq 0$ because $H^0(X; \OO(\ast)) \cong R[t_0, t_1]$ with $|t_j| = w_j$. Thus, we see that $b_j\leq 0$ for all $1\leq j\leq r$. 

Therefore we get
\[\Ext^1_X(\FF, \OO_X) \cong \bigoplus_{j=1}^{r} H^1(X; \OO(-b_j)) = 0\]
by Proposition \ref{prop:projcoh}. Hence, (\ref{align:vector}) is a \emph{split} short exact sequence. 
\end{proof}

\section{Lifting the Hasse invariant}\label{app:Hasse}
This appendix does not contain an original contribution by the author. Besides a short introduction to the Hasse invariant it proves that the Hasse invariant is liftable to characteristic zero once we have chosen a $\Gamma_1(k)$-level structure for $k\geq 2$. This proof is (essentially) taken from \cite{Hasse} and the credit belongs to the mathoverflow user Electric Penguin. 

We begin by recalling the definition of the Hasse invariant from \cite[Section 2.0]{Kat73}. Let $f\colon E\to \Spec R$ be an elliptic curve. By the proof of \cite[II.1.6]{D-R73}, $R^1f_*\OO_E$ is locally free and thus Grothendieck duality implies a canonical isomorphism $R^1f_*\OO_E \cong \omega_{E/R}^{\tensor (-1)}$ of line bundles. Choosing a trivialization of $\omega_{E/R}$ thus induces a dual $R$-basis $x$ of $H^1(E;\OO_E)$. If $R$ is an $\F_p$-algebra, the absolute Frobenius $F_{abs}$ acts on $H^1(E;\OO_E)$ so that $F_{abs}^*(x) = \lambda x$ with $\lambda \in R$. Associating to each elliptic curve $E\to \Spec R$ over an $\F_p$-algebra $R$ with chosen trivialization of $\omega_{E/R}$ the element $\lambda \in R$ defines an $\F_p$-valued (holomorphic) modular form $A_p$ of weight $p-1$, as explained in \cite{Kat73}. This is the \emph{Hasse invariant}. Katz shows that the $q$-expansion of $A_p$ in $\F_p\llbracket q\rrbracket$ is identically $1$.

We want to prove that the Hasse invariant is liftable to characteristic $0$ in the presence of a level structure. More precisely, we have the following:

\begin{prop}\label{prop:Hasse}
 For every odd $k\geq 2$, there is a modular form $F$ of weight $1$ and level $\Gamma_1(k)$ over a cyclotomic ring $\Z_{(2)}[\zeta_k]$ such that $F\equiv A_2 \mod 2$.\footnote{It follows from Proposition \ref{prop:eta} that adjoining $\zeta_k$ is not necessary here, but is rather an artifact of the version of the $q$-expansion principle we are using.} 
\end{prop}
\begin{proof}
The proof we present is based on the mathoverflow post \cite{Hasse} by the user Electric Penguin and uses the theory of Eisenstein series. Let $\chi\colon (\Z/k)^\times \to \C^\times$ be an odd character, i.e.\ we require $\chi(-1) = -1$. Following \cite[Section 4.8]{D-S05}, we define its associated weight $1$ Eisenstein series by
\[E_1^{\chi}(\tau) = \frac12L(0,\chi) + \sum_{n=1}^{\infty}c_nq^n,\]
where $c_n = \sum_{d|n}\chi(n)$ (this is half of the normalization chosen in \cite{D-S05}). Here $L(s,\chi)$ denotes the Dirichlet $L$-series associated to $\chi$. Let $a$ be the \emph{conductor} of $\chi$, i.e.\ the smallest $a|k$ such that $\chi$ factors through $(\Z/a)^\times$. Then we have
\[L(0,\chi) = -\frac1{a}\sum_{n=1}^{a-1} n\chi(n).\]
It is proven in \cite{D-S05} that $E_1^{\chi}$ is a $\Gamma_1(k)$-modular form of weight $1$ over the ring $\C$ with character $\chi$ (although the latter fact will not be relevant for us). 

In general, if $K/\Q$ is a finite extension of degree $n$, we can extend the $2$-adic valuation from $\Q$ to $K$ by setting $v_2(x) = \frac1n v_2(N_{K/\Q}(x))$ for $x\in K$. For example, let $K/\Q$ be the cyclotomic extension $\Q(\zeta)$ with $\zeta =\zeta_{2^m}$ a primitive root of unity. Then $N_{K/\Q}(x-\zeta) = x^{2^{m-1}}+1$ (for $x$ fixed by the Galois group) as both have the same zeros. In particular, $v_2(1-\zeta) = \frac1{2^{m-2}}$. As $1-\zeta$ generates the maximal ideal of $\Z_{(2)}[\zeta]$, we see that $\frac1{2^{m-2}}$ is the minimal positive $2$-adic valuation in $\Q(\zeta)$ and thus every $2$-adic valuation is a multiple of $\frac1{2^{m-2}}$. 

For the proof of the proposition, we may assume $k$ to be an odd prime $p$ (as every $\Gamma_1(p)$-modular form is also a $\Gamma_1(k)$-modular form for $p|k$), which we will do in the following. Thus consider an odd character $\chi\colon (\Z/p)^\times \to \C^\times$ and the associated Eisenstein series $E_1^\chi$. We will assume that $\chi$ has order $2^m$ for $p-1 = 2^ml$ with $l$ odd. This implies that $\chi$ is surjective onto the $2^m$-th roots of unity and that $\ker(\chi)\subset (\Z/p)^\times$ has order $l$. Note $\chi(-1) = -1$.

\begin{claim}We have
 $$v_2(L(0,\chi)) = 1-\frac1{2^{m-2}},$$
 where the valuation is taken in $\Q(\zeta_{2^m})$. 
\end{claim}
\begin{myproof}
Choose $b_0,\dots, b_{2^{m-1}-1}$ with $\chi(b_j) = \zeta^j$, where we still use the notation $\zeta = \zeta_{2^m}$. We furthermore use the notation $\overline{x}$ to denote for an integer $x$ the integer $0\leq \overline{x}\leq p-1$ it is congruent to mod $p$. 

We have
\begin{align*}
L(0,\chi) &= -\frac1p\sum_{n=1}^{p-1}n\chi(n) \\
 &= -\frac1p\sum_{j=0}^{2^{m-1}-1} \sum_{[i]\in\ker(\chi)} \left(\overline{ib_j}\chi(ib_j) + (p-\overline{ib_j})\chi(p-ib_j)\right) \\
 &= -\frac1p\sum_{j=0}^{2^{m-1}-1} \sum_{[i]\in\ker(\chi)} \left(\overline{ib_j}\chi(b_j) + (p-\overline{ib_j})(-\chi(b_j))\right) \\
 &= -\frac1p\sum_{j=0}^{2^{m-1}-1} (-pl + 2\sum_{[i]\in\ker(\chi)} \overline{ib_j})\zeta^j \\
 &\equiv \sum_{j=0}^{2^{m-1}-1} \zeta^j \mod 2
\end{align*}
in $\Z_{(2)}[\zeta]$. As this is not congruent to $0$ mod $2$, this implies in particular that $v_2(L(0,\chi)) < 1$. Moreover,
 \[L(0,\chi)(1-\zeta) \equiv (1+\zeta + \cdots + \zeta^{2^{m-1}-1})(1-\zeta) \equiv 1+1 \equiv 0 \mod 2\]
 and thus $v_2(L(0,\chi)(1-\zeta)) = v_2(L(0,\chi)) + \frac1{2^{m-2}}\geq 1$. As every $2$-adic valuation $\Z_{(2)}[\zeta]$ is a multiple of $\frac1{2^{m-2}}$, this implies the result. 
\end{myproof}

We see that $E =(1-\zeta)E_1^{\chi}$ is a level $p$, weight $1$ modular form for $\Gamma_1(p)$ with $q$-expansion in the ring $\Z_{(2)}[\zeta]$. Furthermore, we know that $E \equiv 1 \mod (1-\zeta)$ in $\Z_{(2)}[\zeta]\llbracket q\rrbracket$.

Write 
\[ E = \sum_{i=0}^{2^{m-1}-1}\zeta^if_i \in \Z_{(2)}[\zeta]\llbracket q\rrbracket\]
with $f_i \in \Z_{(2)}\llbracket q\rrbracket$. 

\begin{claim}
Each $f_i$ is a $\Gamma_1(p)$-modular form.
\end{claim}
\begin{myproof}
The Galois group $(\Z/2^m)^\times = \Gal(\Q(\zeta)/\Q)$ acts on $\Q(\zeta)$-valued modular forms. In particular, $\sum_{g\in \Gal(\Q(\zeta)/\Q)}g(\zeta^{-i}E)$ is a modular form, namely $2^{m-1}f_i$. 
\end{myproof}

Thus, also $$F = \sum_{i=0}^{2^{m-1}-1}f_i\in\Z_{(2)}\llbracket q\rrbracket$$
is a $\Gamma_1(p)$-modular form. By the $q$-expansion principle from \cite[Thm 1.6.1]{Kat73}, $F$ is thus a modular form for $\Gamma_1(p)$ over the ring $\Z_{(2)}[\zeta_p]$ for a $p$-th root of unity $\zeta_p$. Furthermore, 
$$F \equiv E \equiv 1 \mod (1-\zeta).$$
As $(1-\zeta)\Z_{(2)}[\zeta]\llbracket q\rrbracket \cap \Z_{(2)}\llbracket q\rrbracket = (2)\Z_{(2)}\llbracket q\rrbracket$, we also get $F \equiv 1 \mod 2$. 

This proves the proposition.
\end{proof}

\section{Computation of weight $1$ cusp forms over $\F_2$}\label{app:CuspForms}
Throughout this appendix, a modular forms of level $n$ over a field $K$ will always mean a modular form for $\Gamma_1(n)$ over $K$ (we will always assume that the characteristic of $K$ does not divide the level and $n\geq 5$).\footnote{If $K$ does not contain an $n$-th root of unity, $\Gamma_1(n)$-level structures should be understood to provide an embedding of $\mu_n$ instead of $\Z/n$, or else one has to use a non-standard notion of $q$-expansions. See \cite[Section 2]{KatzRealEisenstein} or \cite[Appendix A]{M-OModular} for more information about the distinction between ``naive'' and ``arithmetic'' level structures.} As before, we denote the space of such modular forms of weight $k$ by $M_k(\Gamma_1(n); K)$ and the corresponding space of cusp forms by $S_k(\Gamma_1(n); K)$. The dimension of the spaces only depends on the characteristic of $K$, and 
$$\dim_{\Q} M_k(\Gamma_1(n); \Q) \leq \dim_{\F_p} M_k(\Gamma_1(n); \F_p)$$
 with equality for $k \geq 2$ and similarly for cusp forms (see Lemma \ref{lem:inequality}). 

For $k\geq 2$, the dimensions of these spaces are easily computable using Riemann--Roch (see Proposition \ref{prop:riemann}). It is much more tricky to compute the dimensions of  $S_1(\Gamma_1(n); K)$. For $K$ of characteristic zero, a fast algorithm was found and implemented by Buzzard and Lauder \cite{B-L17}. The associated webpage \url{http://people.maths.ox.ac.uk/lauder/weight1/} displays tables up to level 1500, containing not only the dimensions, but bases, associated Galois extensions etc.

In \cite{Buz} Buzzard also considers the case of $K$ of finite \emph{odd} characteristic. The smallest level, where he shows that the dimension of $S_1(\Gamma_1(n); \F_p)$ is bigger than that of $S_1(\Gamma_1(n); \Q)$ is level $74$ (with $p=3$). Buzzard restricts to modular forms where the $(\Z/n)^\times$ action is via a fixed Dirichlet character. As over $\F_p$ not every representation of $(\Z/n)^\times$ is $1$-dimensional, this is potentially a non-trivial restriction. Apart from this restriction, his search appears to be exhaustive, i.e.\ there is no smaller level with non-liftable forms in odd characteristic. 

Gabor Wiese also wrote a \texttt{MAGMA} package computing weight $1$ cusp forms in characteristic $2$, but only for $\Gamma_0(n)$. Schaeffer \cite{Sch14} has a fast algorithm as well, but requires again a Dirichlet character. 

Our modest aim in this appendix to complement these results for $K = \F_2$ without imposing any Dirichlet character. We do not propose a new algorithm, but rather give implementations of two variants of an algorithm proposed by Edixhoven in \cite{Edi06} (which was already the basis of Wiese's work). 

\begin{prop}[\cite{Edi06}, Prop 4.2]
Let $g = \sum_{i=1}^\infty a_i q^i$ be a weight $2$ cusp form of level $n$ over $\F_2$. Let $B = \frac{n^2}6\prod_{l|n}(1-\frac1{l^2})$, where the product runs over all primes dividing $n$ (the so-called \emph{Sturm bound}). Assume that $a_i = 0$ for all odd $i \leq B$. Then $g = f^2$ for a weight-$1$ cusp form $f$ of level $n$ over $\F_2$. 
\end{prop}

As the space of weight $2$ cusp forms is computable in \texttt{SAGE}, this leads easily to an algorithm that gives us even the $q$-expansions of a basis of the space of weight $1$ cusp forms over $\F_2$. A sample implementation is the following:\footnote{Little effort was spent to make this \texttt{SAGE} implementations optimal, but they sufficient for the small levels we are considering.}

\begin{verbatim}
def deg(n): #index of Gamma_1(n) in SL_2(Z)
    d = n^2
    for p in list(factor(n)):
        d = d*(1-1/(p[0]^2))
    return d

def fill(L, length): #fill list with zeros
    n = length - len(L)
    return L + ([0]*n)

def CuspF2q(n, speed=12): 
    #gives a basis of the space weight 1 cusp forms for Gamma1(n) over F_2 
    # in vector notation
    M = CuspForms(Gamma1(n),2, GF(2))
    Prec =2*(floor(deg(n)/speed))+2 #for speed =12 this is the Sturm bound
    
    V = VectorSpace(GF(2), Prec)
    L = [fill(list(f.qexp(Prec)),Prec) for f in M.basis()]
    LV = [V(l) for l in L]
    W = V.subspace(LV)
    
    def square(Listje): return list(reduce(lambda s, t: s} + t, 
                                           zip(Listje,[0]*Prec), ()))[:Prec]
    Lsquare = map(square, L)
    LVsquare = [V(l) for l in Lsquare]
    Wsquare = V.subspace(LVsquare)
    
    Meet = W.intersection(Wsquare)
    Base = [B[0::2] for B in Meet.basis()] 
    #Odd entries deleted = taking preimage of Frobenius 
    # to obtain weight 1 cusp form
    return Base

def CuspF2qexp(n, speed=12): #transforming vector notation into power series
    C = CuspF2q(n, speed)
    L = [[b[i] for i in range(b.length())] for b in C]
    R.<q> = PowerSeriesRing(GF(2))
    qexps = [R(l) for l in L]
    return qexps
\end{verbatim}

This can be sped up by ignoring $q$-expansions and computing directly with Hecke algebras (as this can be done via modular symbols). The proposition above becomes in this language:
\begin{prop}[\cite{Edi06}, Prop 4.10]
Let again $B = \frac{n^2}6\prod_{l|n}(1-\frac1{l^2})$. Set $V = S_2(\Gamma_1(n);\F_2)$. Let $\mathbb{T}_{odd}$ be the sub vector space of $\End_{\F_2}(V)$ generated by the Hecke operators $T_i$ with $i\leq B$ odd and $\mathbb{T}$ be the full Hecke algebra, i.e.\ the subspace of $\End_{\F_2}(V)$ generated by \emph{all} Hecke operators. Then 
$$\dim_{\F_2}S_1(\Gamma_1(n);\F_2) = \dim_{\F_2}\mathbb{T}-\dim_{\F_2}\mathbb{T}_{odd}.$$
\end{prop}
Note that $\dim_{\F_2}\mathbb{T}$ agrees with $\dim_{\F_2}S_2(\Gamma_1(n);\F_2)$, which is half the dimension of that of cuspdial symbols.

A sample implementation of the resulting algorithm is the following (using the same function \texttt{deg(n)} as above):\\

\begin{verbatim}
def vect(A, dim): #transforms a matrix into a list of n^2 elements
    v = [ ]
    for i in range(dim):
        v.extend(list(A.row(i)))
    return v

def CuspF2(n):
    M = ModularSymbols(Gamma1(n),2,base_ring=GF(2)).cuspidal_subspace()
    Prec =floor(deg(n)/12)+1 #Half the Sturm bound
    
    T = M.hecke_algebra()
    dim = T.module().dimension()
	
    Lodd = [T.hecke_matrix(2*n+1) for n in range(Prec)]
    V = VectorSpace(GF(2), dim^2)
    LoddV = [V(vect(l, dim)) for l in Lodd]
    Vodd = V.subspace(LoddV)
    dimodd = Vodd.dimension()
    return dim/2- dimodd
\end{verbatim}

\begin{remark}
There is a number of other approaches possible to compute weight $1$ cusp forms, but some of these involve weight $k$ cusp forms for $k>2$ and these spaces of cusp forms grow very fast in dimension and are thus expensive to compute.
\end{remark}

Running the second algorithm shows that 
$$\dim_{\F_2} S_1(\Gamma_1(n);\F_2) = \dim_{\Q} S_1(\Gamma_1(n);\Q)$$
for all odd $n <70$ except for $n=65$, where we have $\dim_{\F_2} S_1(\Gamma_1(65);\F_2) = 2$ while $S_1(\Gamma_1(65);\Q) = 0$ (which be obtain by \texttt{MAGMA} or the tables by Buzzard and Lauder). 

Running the first algorithm for $n=65$, gives us the $q$-expansions of a basis of $S_1(\Gamma_1(n);\F_2)$.
\vspace{0.1cm}
\begin{align*}
f_1 = q^{2} &+ q^{10} + q^{12} + q^{14} + q^{16} + q^{26} + q^{28} + q^{34} + q^{38} + q^{42} + q^{44} + q^{50} + q^{54} + q^{60} + q^{66} \\
 &+ q^{68} + q^{70} + q^{76} + q^{80} + q^{86} + q^{92} + q^{96} + q^{102} + q^{112} + q^{114} + q^{116} + q^{118} + q^{122} \\
&+ q^{128} + q^{130} + q^{132} + q^{138} + q^{140} + q^{142} + q^{148} + q^{154} + q^{156} + q^{164} + q^{170} + q^{172} + \cdots \\
f_2 = q^{4} &+ q^{6} + q^{12} + q^{14} + q^{20} + q^{22} + q^{30} + q^{32} + q^{34} + q^{38} + q^{44} + q^{46} + q^{48} + q^{52} + q^{58} \\
 &+ q^{60}+ q^{66} + q^{70} + q^{74} + q^{78} + q^{82} + q^{84} + q^{86} + q^{92} + q^{96} + q^{100} + q^{108} + q^{110} + q^{112}  \\
&+ q^{116}+ q^{118} + q^{122} + q^{134} + q^{138} + q^{142} + q^{148} + q^{150} + q^{156} + q^{160} + q^{162} + q^{164} + q^{170} +\cdots
\end{align*}

\newpage

\section{Tables of decomposition numbers}\label{sec:tables}
Let $f_n\colon \MMb_1(n)_{\C} \to \MMb_{ell,\C}$ be the projection and $$(f_n)_*\OO_{\MMb_1(n)_\C} \cong \bigoplus_{i\in\Z}\bigoplus_{l_i}\omega^{\tensor -i}.$$

\renewcommand{\arraystretch}{0.98}
  \begin{tabular}{lccccccccccccc}
$n$ & genus & $l_0$ & $l_1$ & $l_2$ & $l_3 $& $l_4$ & $l_5$ & $l_6$ & $l_7$ & $l_8$ & $l_9$ & $l_{10}$ & $l_{11}$ \\\hline\\
2 & 0 & 1 & 0 & 1 & 0 & 1 & 0 & 0 & 0 & 0 & 0 & 0 & 0 \\
3 & 0 & 1 & 1 & 1 & 2 & 1 & 1 & 1 & 0 & 0 & 0 & 0 & 0 \\
4 & 0 & 1 & 1 & 2 & 2 & 2 & 2& 1 & 1 & 0 & 0 & 0 & 0 \\
5 & 0 & 1 & 2 & 3 & 4 & 4 & 4 & 3 & 2 & 1 & 0 & 0 & 0 \\
6 & 0 & 1 & 2 & 3 & 4 & 4 & 4 & 3 & 2 & 1 & 0 & 0 & 0 \\
7 & 0 & 1 & 3 & 5 & 7 & 8 & 8 & 7 & 5 & 3 & 1 & 0 & 0 \\
8 & 0 & 1 & 3 & 5 & 7 & 8 & 8 & 7 & 5 & 3 & 1 & 0 & 0 \\
9 & 0 & 1 & 4 & 7 & 10 & 12 & 12 & 11 & 8 & 5 & 2 & 0 & 0 \\
10 & 0 & 1 & 4 & 7 & 10 & 12 & 12 & 11 & 8 & 5 & 2 & 0 & 0 \\
11 & 1 & 1 & 5 & 10 & 15 & 19 & 20 & 19 & 15 & 10 & 5 & 1 & 0 \\
12 & 0 & 1 & 5 & 9 & 13 & 16 & 16 & 15 & 11 & 7 & 3 & 0 & 0 \\
13 & 2 & 1 & 6 & 13 & 20 & 26 & 28 & 27 & 22 & 15 & 8 & 2 & 0 \\
14 & 1 & 1 & 6 & 12 & 18 & 23 & 24 & 23 & 18 & 12 & 6 & 1 & 0 \\
15 & 1 & 1 & 8 & 16 & 24 & 31 & 32 & 31 & 24 & 16 & 8 & 1 & 0 \\
16 & 2 & 1 & 7 & 15 & 23 & 30 & 32 & 31 & 25 & 17 & 9 & 2 & 0 \\
17 & 5 & 1 & 8 & 20 & 32 & 43 & 48 & 47 & 40 & 28 & 16 & 5 & 0 \\
18 & 2 & 1 & 8 & 17 & 26 & 34 & 36 & 35 & 28 & 19 & 10 & 2 & 0 \\
19 & 7 & 1 & 9 & 24 & 39 & 53 & 60 & 59 & 51 & 36 & 21 & 7 & 0 \\
20 & 3 & 1 & 10 & 22 & 34 & 45 & 48 & 47 & 38 & 26 & 14 & 3 & 0 \\
21 & 5 & 1 & 12 & 28 & 44 & 59 & 64 & 63 & 52 & 36 & 20 & 5 & 0 \\
22 & 6 & 1 & 10 & 25 & 40 & 54 & 60 & 59 & 50 & 35 & 20 & 6 & 0 \\
23 & 12 & 1 & 12 & 33 & 55 & 76 & 87 & 87 & 76 & 55 & 33 & 12 & 1 \\
24 & 5 & 1 & 12 & 28 & 44 & 59 & 64 & 63 & 52 & 36 & 20 & 5 & 0 \\
25 & 12 & 1 & 14 & 39 & 64 & 88 & 100 & 99 & 86 & 61 & 36 & 12 & 0 \\
26 & 10 & 1 & 12 & 33 & 54 & 74 & 84 & 83 & 72 & 51 & 30 & 10 & 0 \\
27 & 13 & 1 & 15 & 42 & 69 & 95 & 108 & 107 & 93 & 66 & 39 & 13 & 0 \\
28 & 10 & 1 & 15 & 39 & 63 & 86 & 96 & 95 & 81 & 57 & 33 & 10 & 0 \\
29 & 22 & 1 & 14 & 49 & 84 & 118 & 140 & 139 & 126 & 91 & 56 & 22 & 0 \\
30 & 9 & 1 & 16 & 40 & 64 & 87 & 96 & 95 & 80 & 56 & 32 & 9 & 0 \\
31 & 26 & 1 & 16 & 55 & 95 & 134 & 159 & 159 & 144 & 105 & 65 & 26 & 1 \\
32 & 17 & 1 & 16 & 48 & 80 & 111 & 128 & 127 & 112 & 80 & 48 & 17 & 0 \\
33 & 21 & 1 & 20 & 60 & 100 & 139 & 160 & 159 & 140 & 100 & 60 & 21 & 0 \\
34 & 21 & 1 & 16 & 52 & 88 & 123 & 144 & 143 & 128 & 92 & 56 & 21 & 0 \\
35 & 25 & 1 & 24 & 72 & 120 & 167 & 192 & 191 & 168 & 120 & 72 & 25 & 0 \\
36 & 17 & 1 & 20 & 56 & 92 & 127 & 144 & 143 & 124 & 88 & 52 & 17 & 0 \\
37 & 40 & 1 & 18 & 75 & 132 & 188 & 228 & 227 & 210 & 153 & 96 & 40 & 0 \\
38 & 28 & 1 & 18 & 63 & 108 & 152 & 180 & 179 & 162 & 117 & 72 & 28 & 0 \\
39 & 33 & 1 & 25 & 80 & 136 & 191 & 223 & 223 & 199 & 144 & 88 & 33 & 1 \\
40 & 25 & 1 & 24 & 72 & 120 & 167 & 192 & 191 & 168 & 120 & 72 & 25 & 0 \\
41 & 51 & 1 & 20 & 90 & 160 & 229 & 280 & 279 & 260 & 190 & 120 & 51 & 0 \\
42 & 25 & 1 & 24 & 72 & 120 & 167 & 192 & 191 & 168 & 120 & 72 & 25 & 0 
 \end{tabular}

 \newpage
\begin{tabular}{cc}
\begin{tabular}{lcccccccc}
$n$ &  $k_0$ & $k_1$ & $k_2$ & $k_3 $& $k_4$ & $k_5$ & $k_6$ & $k_7$\\ \hline
4 & 1 & 1 & 1 & 1 & 0 & 0 & 0 & 0 \\
5 & 1 & 2 & 2 & 2 & 1 & 0 & 0 & 0 \\
6 & 1 & 2 & 2 & 2 & 1 & 0 & 0 & 0 \\
7 & 1 & 3 & 4 & 4 & 3 & 1 & 0 & 0 \\
8 & 1 & 3 & 4 & 4 & 3 & 1 & 0 & 0 \\
9 & 1 & 4 & 6 & 6 & 5 & 2 & 0 & 0 \\
10 & 1 & 4 & 6 & 6 & 5 & 2 & 0 & 0 \\
11 & 1 & 5 & 9 & 10 & 9 & 5 & 1 & 0 \\
12 & 1 & 5 & 8 & 8 & 7 & 3 & 0 & 0 \\
13 & 1 & 6 & 12 & 14 & 13 & 8 & 2 & 0 \\
14 & 1 & 6 & 11 & 12 & 11 & 6 & 1 & 0 \\
15 & 1 & 8 & 15 & 16 & 15 & 8 & 1 & 0 \\
16 & 1 & 7 & 14 & 16 & 15 & 9 & 2 & 0 \\
17 & 1 & 8 & 19 & 24 & 23 & 16 & 5 & 0 \\
18 & 1 & 8 & 16 & 18 & 17 & 10 & 2 & 0 \\
19 & 1 & 9 & 23 & 30 & 29 & 21 & 7 & 0 \\
20 & 1 & 10 & 21 & 24 & 23 & 14 & 3 & 0 \\
21 & 1 & 12 & 27 & 32 & 31 & 20 & 5 & 0 \\
22 & 1 & 10 & 24 & 30 & 29 & 20 & 6 & 0 \\
23 & 1 & 12 & 32 & 43 & 43 & 32 & 12 & 1 
\end{tabular}
&
\begin{minipage}{6cm}\begin{center} Decomposition numbers for $$(f_n)_*\OO_{\MMb_1(n)_{(3)}}$$ for decompositions into $k_i$ copies of $$(f_2)_*\OO_{\MMb_1(2)_{(3)}}\tensor \omega^{\tensor -i}$$ \end{center}\end{minipage}\\\\

\begin{tabular}{lccccccc}
$n$ &  $k_0$ & $k_1$ & $k_2$ & $k_3 $& $k_4$ & $k_5$  \\ \hline
5 & 1 & 1 & 1 & 0 & 0 & 0 \\
6 & 1 & 1 & 1 & 0 & 0 & 0 \\
7 & 1 & 2 & 2 & 1 & 0 & 0 \\
8 & 1 & 2 & 2 & 1 & 0 & 0 \\
9 & 1 & 3 & 3 & 2 & 0 & 0 \\
10 & 1 & 3 & 3 & 2 & 0 & 0 \\
11 & 1 & 4 & 5 & 4 & 1 & 0 \\
12 & 1 & 4 & 4 & 3 & 0 & 0 \\
13 & 1 & 5 & 7 & 6 & 2 & 0 \\
14 & 1 & 5 & 6 & 5 & 1 & 0 \\
15 & 1 & 7 & 8 & 7 & 1 & 0 \\
16 & 1 & 6 & 8 & 7 & 2 & 0 \\
17 & 1 & 7 & 12 & 11 & 5 & 0 \\
18 & 1 & 7 & 9 & 8 & 2 & 0 \\
19 & 1 & 8 & 15 & 14 & 7 & 0 \\
20 & 1 & 9 & 12 & 11 & 3 & 0 \\
21 & 1 & 11 & 16 & 15 & 5 & 0 \\
22 & 1 & 9 & 15 & 14 & 6 & 0 \\
23 & 1 & 11 & 21 & 21 & 11 & 1 

\end{tabular}
& \begin{minipage}{7cm}\begin{center} Decomposition numbers for $$(f_n)_*\OO_{\MMb_1(n)_{(2)}}$$ for decompositions into $k_i$ copies of $$(f_3)_*\OO_{\MMb_1(3)_{(2)}}\tensor \omega^{\tensor -i}$$\end{center}\end{minipage}
\end{tabular}

\vspace{0.6cm}
For an explanation of the visible symmetry in the entries for the entries $n=5,6,7,8, 11, 14, 15$ and $23$ see \cite[Section 5.2]{MeiTopLevel}.

\bibliographystyle{alpha}
\bibliography{../Chromatic}
\end{document}